\documentclass[12pt,centertags,oneside]{amsart} 
\usepackage{tikz-cd}
\usepackage{hyperref}
\hypersetup{
    colorlinks=true,
    linkcolor=red,
    filecolor=magenta,      
    urlcolor=cyan,
}
\usepackage{amssymb}
\usepackage{graphicx}
\usepackage{mathrsfs}
\usepackage[all]{xy}
\usepackage{tocvsec2}
\usepackage{mathtools}
\usepackage{amsthm}
\usepackage{theoremref}
\usepackage{enumitem}
\usepackage{verbatim}
\usepackage{todonotes}
\usepackage[a4paper,width=14.66cm,top=2.52cm,bottom=2.52cm]{geometry}

\title[Equivariant RRG for coherent sheaves]{Equivariant Chern character for coherent sheaves and Riemann-Roch-Grothendieck}
\author{Guangzhe Xu }

\newtheorem{thm}{Theorem}[section]

\newtheorem{prop}[thm]{Proposition}

\theoremstyle{definition}
\newtheorem{defn}{Definition}[section]
\newtheorem{exmp}{Example}[section]

\theoremstyle{remark}
\newtheorem{rmk}{Remark}[section]

\newcommand{\ie}{\emph{i.e.}, }

\newcommand{\R}{\mathbf{R}}

\newcommand{\DD}{\mathcal{D}}

\newcommand{\RRR}{\mathscr{R}}

\newcommand{\CXG}{\mathrm{C} (X, G)}
\newcommand{\MXG}{\mathrm{M}(X,G)}
\newcommand{\MIXG}{\mathrm{M}^\infty(X,G)}

\newcommand{\MIUG}{\mathrm{M}^\infty(U,G)}
\newcommand{\CbXG}{\mathrm{C}^{\mathrm{b}} (X,G)}

\newcommand{\CXE}{C^{\infty}(X, E)}

\newcommand{\CXD}{C^{\infty}(X, D)}
\newcommand{\CXDU}{C^{\infty}(X, \underline{D})}

\newcommand{\DbCohX}{\mathrm{D}^{\mathrm{b}}_{\mathrm{coh}} (X)}
\newcommand{\CbCohXG}{\mathrm{C}^{\mathrm{b}}_{\mathrm{coh}} (X, G)}
\newcommand{\DbCohXG}{\mathrm{D}^{\mathrm{b}}_{\mathrm{coh}} (X, G)}

\newcommand{\CbCohYG}{\mathrm{C}^{\mathrm{b}}_{\mathrm{coh}} (Y, G)}
\newcommand{\DbCohYG}{\mathrm{D}^{\mathrm{b}}_{\mathrm{coh}} (Y, G)}
\newcommand{\DbCohMG}{\mathrm{D}^{\mathrm{b}}_{\mathrm{coh}} (M, G)}
\newcommand{\DbCohSG}{\mathrm{D}^{\mathrm{b}}_{\mathrm{coh}} (S, G)}

\newcommand{\DbCohZG}{\mathrm{D}^{\mathrm{b}}_{\mathrm{coh}} (Z, G)}

\newcommand{\DbCohYgG}{\mathrm{D}^{\mathrm{b}}_{\mathrm{coh}} (Y_g, G)}

\newcommand{\DbCohSgG}{\mathrm{D}^{\mathrm{b}}_{\mathrm{coh}} (S_g, G)}

\newcommand{\DbCohUG}{\mathrm{D}^{\mathrm{b}}_{\mathrm{coh}} (U, G)}

\newcommand{\DbCohPG}{\mathrm{D}^{\mathrm{b}}_{\mathrm{coh}} (P, G)}

\newcommand{\Td}{\mathrm{Td}}
\newcommand{\E}{\mathscr{E}}

\newcommand{\C}{\mathbf{C}}
\newcommand{\Z}{\mathbf{Z}}
\newcommand{\N}{\mathbf{N}}
\newcommand{\PPP}{\mathbf{P}}

\newcommand{\OO}{\mathcal{O}}

\newcommand{\br}[1]{\left(#1\right)}
\newcommand{\brr}[1]{\left[#1\right]}
\newcommand{\brrr}[1]{\left\{#1\right\}}
\newcommand{\anbr}[1]{\left\langle #1\right\rangle }

\newcommand{\AAAA}{\mathcal{A}}

\newcommand{\HH}{\mathscr{H}}

\newcommand{\F}{\mathscr{F}}
\newcommand{\FF}{\mathbb{F}}

\newcommand{\G}{\mathscr{G}}

\newcommand{\LL}{\mathscr{L}}

\newcommand{\CC}{\mathscr{C}}

\newcommand{\lxgc}{\Lambda (T_\mathbf{C}^* X_g)}

\newcommand{\lsb}{\Lambda (\overline{T^* S})}

\newcommand{\lxb}{\Lambda \left(\overline{T^* X}\right)}
\newcommand{\lx}{\Lambda \left(T^* X\right)}
\newcommand{\lxc}{\Lambda \left(T_\mathbf{C}^* X\right)}

\newcommand{\lyb}{\Lambda \left(\overline{T^* Y}\right)}
\newcommand{\lxpc}{\Lambda^p \left(T_{\mathbf{C}}^* X\right)}

\newcommand{\lsc}{\Lambda \left(T_\mathbf{C}^* S\right)}

\newcommand{\lxib}{\Lambda^i \left(\overline{T^* X}\right)}

\newcommand{\lxcHomDDbar}{\Lambda \left(T_\mathbf{C}^* X\right) \widehat{\otimes} \mathrm{Hom} (D, \overline{D}^*)}

\newcommand{\lxcEndD}{\Lambda \left(T_\mathbf{C}^* X\right) \widehat{\otimes} \mathrm{End} (D)}

\newcommand{\OEXC}{\Omega^{(=)}(X, \mathbf{C})}
\newcommand{\OEXgC}{\Omega^{(=)}(X_g, \mathbf{C})}

\newcommand{\HEXgBCC}{H^{(=)}_{\mathrm{BC}}(X_g, \mathbf{C})}

\newcommand{\HEXBCC}{H^{(=)}_{\mathrm{BC}}(X, \mathbf{C})}
\newcommand{\HEXBCR}{H^{(=)}_{\mathrm{BC}}(X, \mathbf{R})}

\newcommand{\HEYgBCC}{H^{(=)}_{\mathrm{BC}}(Y_g, \mathbf{C})}

\newcommand{\HESgBCC}{H^{(=)}_{\mathrm{BC}}(S_g, \mathbf{C})}
\newcommand{\HESBCC}{H^{(=)}_{\mathrm{BC}}(S, \mathbf{C})}
\newcommand{\HEXYgBCC}{H^{(=)}_{\mathrm{BC}}(X_g \times Y_g, \mathbf{C})}

\newcommand{\lmb}{\Lambda \left(\overline{T^* M}\right)}

\newcommand{\dbx}{\overline{\partial}^X}

\newcommand{\dbs}{\overline{\partial}^S}

\newcommand{\dx}{\partial^X}
\newcommand{\ds}{\partial^S}

\newcommand{\CXHomEUEG}{C^{\infty} \left( X, \mathrm{Hom} \left(E, \underline{E}\right)\right)^G}

\newcommand{\CXHombEUEG}{C^{\infty} \left( X, \mathrm{Hom}^\bullet \left(E, \underline{E}\right)\right)^G}

\newcommand{\ox}{\mathcal{O}_X}
\newcommand{\oy}{\mathcal{O}_Y}

\newcommand{\ou}{\mathcal{O}_U}

\newcommand{\om}{\mathcal{O}_M}

\newcommand{\cgBC}{\mathrm{ch}_{g,\mathrm{BC}}}
\newcommand{\cgBCBar}{\overline{\mathrm{ch}}_{g,\mathrm{BC}}}

\newcommand{\chgBC}{\mathrm{ch}_{g,\mathrm{BC}}}
\newcommand{\tdgBC}{\mathrm{Td}_{g,\mathrm{BC}}}
\newcommand{\Hom}{\mathrm{Hom}}
\newcommand{\End}{\mathrm{End}}
\newcommand{\Aut}{\mathrm{Aut}}
\newcommand{\HomG}{\mathrm{Hom}_G}

\newcommand{\cok}{\operatorname{coker}}

\newcommand{\MDG}{\mathscr{M}^D_G}

\newcommand{\McalD}{\mathscr{M}^\mathcal{D}_G}

\newcommand{\chg}{\mathrm{ch}_g}
\newcommand{\ch}{\mathrm{ch}}
\newcommand{\Trs}{\mathrm{Tr}_\mathrm{s}}

\newcommand{\dMDG}{\mathbf{d}^{\mathscr{M}^D_G}}

\newcommand{\AEpp}{A^{E\prime\prime}}
\newcommand{\AFpp}{A^{F\prime\prime}}
\newcommand{\AEUEpp}{A^{\mathrm{Hom} (E, \underline{E}) \prime\prime}}

\newcommand{\ApEpp}{A^{p_* \mathscr{E}\prime\prime}}
\newcommand{\ApEzeropp}{A^{p_* \mathscr{E}_0\prime\prime}}
\newcommand{\ApEzerop}{A^{p_* \mathscr{E}_0\prime}}
\newcommand{\ApEzero}{A^{p_* \mathscr{E}_0}}
\newcommand{\ApEzerosq}{A^{p_* \mathscr{E}_0, 2}}
\newcommand{\AESpp}{A^{E_{S}\prime\prime}}

\newcommand{\AUEpp}{A^{\underline{E}\prime\prime}}

\newcommand{\AEUzeropp}{\underline{A}^{E_0\prime\prime}}

\newcommand{\AEzerop}{A^{E_0\prime}}
\newcommand{\AEzero}{A^{E_0}}
\newcommand{\AEzerosq}{A^{E_0,2}}

\newcommand{\AEzeropp}{A^{E_0\prime\prime}}

\newcommand{\NDpp}{\nabla^{D\prime\prime}}
\newcommand{\NUDpp}{\underline{\nabla}^{D\prime\prime}}

\newcommand{\AEppsq}{A^{E\prime\prime,2}}

\newcommand{\OAXC}{\Omega^{0,\bullet} (X, \mathbf{C})}

\newcommand{\WOX}{\widehat{\otimes}_{\mathcal{O}_X}}

\newcommand{\cone}{\mathrm{cone}}
\newcommand{\Bdg}{\mathrm{B_{dg}}}
\newcommand{\BB}{\mathrm{B}}
\newcommand{\FOI}{\overline{\mathscr{F}}^\infty}
\newcommand{\BXG}{\mathrm{B} (X,G)}
\newcommand{\BUXG}{\underline{\mathrm{B}} (X,G)}

\newcommand{\wTX}{\widehat{TX}}
\newcommand{\X}{\mathcal{X}}
\newcommand{\wTXb}{\overline{\widehat{TX}}}
\newcommand{\wTXbs}{\overline{\widehat{T^* X}}}
\newcommand{\wTRX}{\widehat{T_{\mathbf{R}}X}}
\newcommand{\wTXs}{\widehat{T^*X}}
\newcommand{\lwxbs}{\Lambda\left(\overline{\widehat{T^* X}} \right)}

\newcommand{\lxxbs}{\Lambda\left(\overline{T^* \mathcal{X}} \right)}

\newcommand{\db}{\overline{\partial}}

\newcommand{\dbv}{\overline{\partial}^V}

\newcommand{\dbvs}{\overline{\partial}^{V*}}

\newcommand{\wy}{\widehat{y}}
\newcommand{\wyb}{\overline{\widehat{y}}}
\newcommand{\wY}{\widehat{Y}}

\newcommand{\II}{\textbf{I}}

\newcommand{\wb}[1]{\overline{\widehat{#1}}}
\newcommand{\M}{\mathcal{M}}
\newcommand{\upi}{\underline{\pi}}
\newcommand{\up}{\underline{p}}
\newcommand{\uq}{\underline{q}}
\newcommand{\AAp}{\mathcal{A}^{\prime}}

\newcommand{\AApp}{\mathcal{A}^{\prime\prime}}

\newcommand{\AAppy}{\mathcal{A}_Y^{\prime\prime}}

\newcommand{\ndpp}{\nabla^{D\prime\prime}}
\newcommand{\ndp}{\nabla^{D\prime}}

\newcommand{\ndII}{\nabla^{D\widehat{\otimes} q^* \textbf{I} }}

\newcommand{\nIIp}{\nabla^{\textbf{I} \prime}}

\newcommand{\ndIIp}{\nabla^{D\widehat{\otimes} q^* \textbf{I} \prime}}

\newcommand{\wo}{\widehat{\otimes}}

\newcommand{\wgTX}{\widehat{g}^{TX}}
\newcommand{\gwTX}{g^{\widehat{TX}}}
\newcommand{\wnTX}{\widehat{\nabla}^{TX}}
\newcommand{\nwTX}{\nabla^{\widehat{TX}}}

\newcommand{\wnlTCX}{\widehat{\nabla}^{\Lambda \left(T^*_\mathbf{C} X\right)}}

\newcommand{\wnF}{\widehat{\nabla}^{\mathbb{F}}}

\newcommand{\wn}{\widehat{\nabla}}

\newcommand{\wnFmo}{\widehat{\nabla}^{\mathbb{F},-1}}

\newcommand{\wnFF}{\widehat{\nabla}^{\mathbf{F}}}

\newcommand{\wnFFone}{{}^1\widehat{\nabla}^{\mathbf{F}}}

\newcommand{\AApyb}{\mathcal{A}^\prime_{Y,b}}
\newcommand{\AAyb}{\mathcal{A}_{Y,b}}

\newcommand{\utht}{\underline{\theta}_t}

\numberwithin{equation}{section}

\begin{document}
\maketitle
\begin{abstract}
	In this paper, we develope an equivariant theory of  Chern characters for coherent sheaves on compact complex manifolds with finite group actions, taking values in Bott-Chern cohomology classes. Furthermore, we establish the corresponding Riemann-Roch-Grothendieck theorem in this context.
\end{abstract}
\tableofcontents
\settocdepth{subsection}
\date{\today}

\section{Introduction}

\subsection{The main results}
Let $X$ be a compact complex manifold, and let $\ox$ denote the sheaf of holomorphic functions on $X$.
Let $G$ be a finite group acting holomorphically on $X$.

A $G$-coherent sheaf refers to a coherent sheaf equipped with a $G$-action.
Let $K \br{ X,G }$ denote the corresponding Grothendieck group.

Let $\DbCohXG$ be the derived category of bounded $G$-equivariant complexes of $\ox$-modules with coherent cohomology and let $K \br{\DbCohXG}$ be the corresponding $K$-group.
Classically, we have 
\begin{align} \label{Equ0-2}
  K \br{\DbCohXG} \simeq K(X, G).
\end{align}

 When $G$ is trivial, we use the notation $K\br{X}$ and  $K \br{\DbCohX}$ instead.

If $Y$ is another compact complex $G$-manifold,
let $f: X\rightarrow Y $ be a $G$-equivariant holomorphic map, then the derived pullback and the derived direct image,
\begin{align}
&Lf^*: \DbCohYG \rightarrow \DbCohXG	,  && Rf_*: \DbCohXG \rightarrow \DbCohYG	
\end{align}
induce corresponding morphisms of Grothendieck groups,
\begin{align}
&f^!: K(Y, G) \rightarrow K(X, G), && f_!: K(X, G) \rightarrow K(Y, G).	
\end{align}
Also, the derived tensor product on $\DbCohXG$ induces a corresponding tensor product on $K(X, G)$, so that $K(X, G)$ is a commutative ring.

Let $H_{\mathrm{BC}}(X, \mathbf{C})$ be the Bott-Chern cohomology of $X$. Set 
\begin{align}
  H^{(=)}_{\mathrm{BC}}(X, \mathbf{C}) = \bigoplus_{p=0}^{\mathrm{dim} X} H^{p,p}_{\mathrm{BC}}(X, \mathbf{C}).
\end{align}
Given $g \in G$, let $X_g \subset X, Y_g \subset Y$ be the fixed point sets of $g$. 
Let 
\begin{align}
&f^*: H^{(=)}_{\mathrm{BC}}(Y_g, \mathbf{C}) \rightarrow H^{(=)}_{\mathrm{BC}}(X_g, \mathbf{C}), &&f_*: H^{(=)}_{\mathrm{BC}}(X_g, \mathbf{C}) \rightarrow H^{(=)}_{\mathrm{BC}}(Y_g, \mathbf{C})
\end{align}
 be the pullback and pushforward maps. Let $\tdgBC\br{TX} \in \HEXgBCC $  and  $\tdgBC\br{TY} \in \HEYgBCC $ be the equivariant Todd classes.

Let $\anbr{g} \subset G$ be the cyclic group generated by $g$.

This paper aims to prove the following theorem.
\begin{thm}[See Sections \ref{EquiChernCharacter}, \ref{Section7-5} and Theorem \ref{MainTheorem}]\label{Thm0-7}
There is a map	
\begin{align} \label{Equ0-5}
  \chgBC: K \br{ X,G } \rightarrow  H^{(=)}_{\mathrm{BC}}(X_g, \mathbf{C})
\end{align}
satisfies the following properties.
\begin{enumerate}  [start=1]
\item \label{Con0-1} It factors through  $K \br{ X_g,\anbr{g} }$. That is, the following diagram
\begin{equation}
\begin{tikzcd}
{K(X,G)} \arrow[d] \arrow[rd, "\cgBC"] &   \\
{K(X_g,\anbr{g})} \arrow[r, "\cgBC"]          & \HEXgBCC
\end{tikzcd}	
\end{equation}
commutes.

\item \label{Con0-2}	If $E$ is an equivariant holomorphic vector bundle, $\cgBC$ is defined in usual way, e.g. by classical Bott-Chern's Chern-Weil theory.
\item \label{Con0-5} $\cgBC$ is a ring morphism, \ie if $\F, \underline{\F} \in K(X, G)$, then
\begin{align} \label{Equ0-8}
\begin{aligned}
\chgBC \br{\F + \underline{\F}} = \chgBC\br{\F} +  \chgBC \br{\underline{\F}}	 \; \text{in} \; \HEXgBCC. \\
\chgBC \br{\F \cdot \underline{\F}} = \chgBC\br{\F} \chgBC \br{\underline{\F}}	 \; \text{in} \; \HEXgBCC.
\end{aligned}
\end{align}
\item \label{Con0-3} $\cgBC$ is functorial under equivariant pullbacks, \ie if $\F \in K(Y, G)$, then
\begin{align}
	\chgBC \br{f^! \F} = f^* \chgBC \br{\F} \; \text{in} \; \HEXgBCC.
\end{align}
\item \label{Con0-4} $\cgBC$ verifies the Riemann-Roch-Grothendieck formula under equivariant direct image, \ie if $\F \in K(X, G)$, then 
\begin{align} \label{Equ0-9}
\tdgBC \br{TY} \chgBC \br{f_! \F} = f_* \brr{\tdgBC \br{TX} \chgBC  \br{\F}}	\; \text{in} \; \HEYgBCC.
\end{align}	
\end{enumerate}
\end{thm}
Furthermore, we will prove that our equivariant Chern character is unique in following sense. 
\begin{thm}[See Theorem \ref{Thm8-1}] \label{Thm0-6}
There is a unique map $K \br{ X,G } \rightarrow \HEXgBCC$ satisfies \eqref{Con0-1}, \eqref{Con0-2}, \eqref{Con0-3} in Theorem \ref{Thm0-7} and the weaker conditions,

\eqref{Con0-5}$'$ It is a group morphism, \ie the first equation in \eqref{Equ0-8} holds.

\eqref{Con0-4}$'$ It satisfies the Riemann-Roch-Grothendieck theorem  for equivariant embeddings, \ie in \eqref{Equ0-9}, $f: X\rightarrow Y$ is an equivariant embedding.
\end{thm}

We will describe the earlier work, the ideas and techniques that are used to obtain Theorems \ref{Thm0-7} and \ref{Thm0-6}.

\subsection{Earlier work}
Coherent sheaves were first introduced by Oka \cite{Oka50}, providing a powerful generalization of vector bundles that has become fundamental in complex and algebraic geometry. 
The study of Chern classes for coherent sheaves, extending the classical theory for vector bundles, emerged as an important problem in the late 1950s.
In the algebraic case, the Chern class in the Chow ring was defined by Grothendieck \cite{G58} and the Riemann-Roch theorem, now is called Riemann-Roch-Grothendieck theorem was established by him in \cite{BS58}.

The first systematic treatment of Chern classes for coherent sheaves on complex manifolds was developed by Atiyah and Hirzebruch \cite{AH61, AH62}.
In these papers, they studied a real analytic resolution of a coherent sheaf to define the Chern classes with values in the de Rham cohomology and obtain the Riemann-Roch-Grothendieck theorem for embeddings.

In \cite{Gre80}, Green introduced simplicial resolution for coherent sheaves and got Chern classes in de Rham cohomology.
In \cite{H23, H24}, Hosgood constructed an $\br{\infty, 1}$ categorial framework to generalize Green's work.

In \cite{BTT81}, O'Brian, Toledo and Tong defined a Chern character map from $K(X)$ to the Hodge cohomology class. In \cite{BTT85}, they prove a corresponding Riemann-Roch-Grothendieck theorem.

In \cite{L87}, Levy constructed a morphism $\alpha_X$ from $K\br{X}$ to  topological $K$-group $K^{\mathrm{top}} \br{X}$ which extends the obvious forgetful functor from holomorphic vector bundles to smooth vector bundles and which is compatible with direct images.
The constructions are valid for general complex spaces.

In \cite{G10},  Grivaux constructed Chern classes on $K(X)$ that take values in any cohomology theory satisfying a certain axioms, in particular, the rational Deligne cohomology. 
 His construction proves these classes satisfy Riemann-Roch-Grothendieck for projective morphisms.
He also proved a unicity result for Chern classes on $K(X)$.
The formulation and proof of our unicity Theorem \ref{Thm0-6} are inspired on this work.

In \cite{Wu23}, Wu proved that rational Bott-Chern cohomology verifies the axioms given by Grivaux \cite{G10}. In particular, Chern classes of coherent sheaves in rational Bott-Chern cohomology can be defined and the corresponding Riemann-Roch-Grothendieck theorem holds for projective morphism between compact complex manifolds.

In \cite{Qu85}, Quillen introduced superconnections and generalized the classical Chern-Weil theory to superconnections.
In \cite{BGS2}, Bismut, Gillet and Soulé adapted Quillen's superconnection to complex geometry using a version of antiholomorphic superconnections.
In \cite{B110}, Block further generalized the concept of antiholomorphic superconnections in a more abstract way.
This object can be used to study the coherent sheaves.
Following his work, Qiang \cite{Q16, Q17} and Bismut, Shen, Wei \cite{BSW} defined a Chern character from $K\br{X}$ to the real Bott-Chern cohomology class.
In \cite{BSW}, they also proved the corresponding Riemann-Roch-Grothendieck theorem for any arbitrary holomorphic maps between compact complex manifolds. 

Recently, in \cite{MTTW25}, Ma, Tang, Tseng and Wei used flat antiholomorphic superconnections to study orbifold Chern character. They proved the Riemann-Roch-Grothendieck theorem for orbifold embeddings and gave a corresponding unicity theorem.

\subsection{The construction of equivariant Chern character}
In this section, we briefly review the techniques for constructing Chern characters of coherent sheaves developed in \cite{B110} and \cite{BSW}. 
Additionally, we outline how these constructions can be generalized to the equivariant setting, with detailed elaborations provided in Sections \ref{SectionEquOX}-\ref{EquiGeneMetr1}.
\subsubsection{The antiholomorphic superconnections}

Let $E$ be a $\Z$-graded smooth vector bundle of finite rank on $X$.
Let $\lxb$ be the exterior algebra of antiholomorphic cotangent vector bundles. 
We suppose that $\lxb$ acts on the left on $E$ so that $\Lambda^i \br{\overline{T^* X}}$ increases the degree of $E$ by $i$.

We assume that $E$ is free over $\lxb$, that is, there exists a $\Z$-graded smooth bundle $D$ on $X$ such that 
\begin{align} \label{Equ0-1}
  E \simeq \Lambda \br{\overline{T^* X}} \wo D
\end{align}
as $\lxb$-modules.
The identification \eqref{Equ0-1} is non-canonical.

\begin{defn} 
An antiholomorphic superconnection is a first order differential operator $\AEpp$ acting on $\CXE$ such that it acts as an operator of degree $1$, it is flat and it satisfies the Leibniz rule in the sense that if $\alpha \in \OAXC$, $s \in \CXE$, then
    \begin{align} 
    \AEpp (\alpha s) = \br{\dbx \alpha} s + \br{-1}^{\deg \alpha} \alpha \AEpp s.
    \end{align}
\end{defn}

Thanks to the identification \eqref{Equ0-1}, the antiholomorphic superconnection $\AEpp$ on $E$ gives non-canonically an antiholomorphic superconnection $\AEzeropp$ on $E_0 = \Lambda \br{\overline{T^* X}} \wo D$.
This non-canonicity appears naturally in geometric setting and however will induces conceptual difficulties.

 
In \cite{B110}, \cite{BSW}, the authors defined a homotopy category $\underline{\mathrm{B}} \br{X}$ of antiholomorphic superconnections on $X$.
The key result to define the Chern character for coherent sheaves are the following theorem \cite[Theorem 4.3]{B110}, \cite[Theorem 6.5.1]{BSW}.
\begin{thm} \label{Thm0-2}
There is a natural functor $\underline{F}_X : \underline{\mathrm{B}} \br{X} \rightarrow \DbCohX$	 which is an equivalence of triangulated categories.
\end{thm}

One of the main step for proving the theorem is to show that the functor $\underline{F}_X$ is essentially surjective, that is, to find an antiholomorphic superconnection associated to an object $\E$ in $\DbCohX$.

Let $\ox^\infty$ be the sheaves of smooth functions on $X$.
Let $\E^\infty$ be the $\ox^\infty$-complex defined by
\begin{align} \label{Equ0-3}
	\E^\infty = \ox^\infty \otimes_{\ox} \E.
\end{align}

Even though the sheaf of holomorphic sections of $\E$ does not admit global resolution by complex of holomorphic vector bundles \cite{V02}, the  sheaf $\E^\infty$ does always admit a global resolution by complex of smooth vector bundles.
That is, we have the following result established  in \cite[Proposition II.2.3.2]{Col71} and also proved in \cite[Proposition 6.3.2]{BSW}. 
A corresponding result in real analytic category was established in \cite[Proposition 2.6]{AH61}.
\begin{prop} \label{Prop0-1}
There exist a bounded  complex of smooth vector bundles of finite rank $Q$ and a quasi-isomprhism of $\ox^\infty$-complexes $\ox^\infty Q \rightarrow \E^\infty$.
\end{prop}


 By a recursion argument established in \cite[Theorem 6.3.6]{BSW}, we can construct an antiholomorphic superconnection on $ \Lambda \br{\overline{T^* X}} \wo Q$ such that the complex of the sheaves of local sections is isomorphic to $\E$ in  $\DbCohX$.

The category $\underline{\mathrm{B}} \br{X}$ has pullbacks and tensor products defined as follows.
Let $X, Y$ be two compact complex manifolds and let $f: X\to Y$ be a holomorphic map.
If $\F = \br{F, \AFpp}$ is an antiholomorphic superconnection on $Y$ and if $\E = \br{E, \AEpp}$ and $\underline{\E} = \br{\underline{E}, \AUEpp}$ are two antiholomorphic superconnections on $X$, then as stated in \cite[Sections 5.7, 5.8]{BSW}, we have
\begin{thm}
	There is a natural pullback of antiholomorphic superconnection $\br{f_b^* F, A^{f_b^* F \prime\prime}}$ on $X$ with $f_b^* F = \lxb \wo_{\Lambda \br{\overline{f^* T^* Y}}} f^* F$.
	
	There is a natural tensor product of antiholomorphic superconnections $\E \wo_b \underline{\E} = \br{E \wo_b \underline{E}, A^{E\wo_b \underline{E} \prime\prime}}$ on $X$ with $E \wo_b \underline{E} = E \wo_{\lxb} \underline{E}$.
\end{thm}
Moreover, according to \cite[Propositions 6.6.1 and 6.7.1]{BSW}, the functor $\underline{F}_X$ is compatible with pullbacks and tensor products.
\subsubsection{The Chern character of coherent sheaves}

Thanks to Theorem \ref{Thm0-2}, to define the Chern character of coherent sheaves, we need to develop Chern-Weil theory for antiholomorphic superconnections.

To handle the non-canonicity of \eqref{Equ0-1}, in \cite[Section 7]{BSW}, the authors introduced a generalized metric on $D$. 
To a generalized metric on $D$, we can associate the adjoint $\AEzerop$ of $\AEzeropp$, so that 
\begin{align}
\AEzero = \AEzeropp + \AEzerop
\end{align}
is a superconnection in the sense of Quillen \cite{Qu85}.

Let $\mathrm{Tr_s}: \lxc \widehat{\otimes } \End (D) \rightarrow \lxc$ be the supertrace. Classically it vanishes on supercommutaters.
We fix a square root of $i = \sqrt{-1}$.
Let $\varphi$ denote the morphism of $\lxc$ that maps $\alpha$ to $(2i \pi)^{- \mathrm{deg} \alpha /2} \alpha$.
\begin{defn} \label{def0-1}
Set
\begin{align} 
\ch \br{\AEzeropp, h} = \varphi \Trs \brr{ \exp\br{- A^{E_0 , 2}}} \in \Omega \br{X, \C}.
\end{align}
\end{defn}

The results of \cite[Proposition 3.28]{BC65}, \cite[Section 2]{Qu85} and \cite[Theorems 1.15 and 1.24]{BGS1} can be extended to $\ch \br{\AEzeropp, h}$ and it was stated in \cite[Theorem 8.12]{BSW}.
\begin{thm} \label{Thm0-8}
The form $\ch \br{\AEzeropp, h}$ lies in $\bigoplus_{p=0}^{\mathrm{dim} X}\Omega^{(p,p)} \br{X, \R}$, it is closed, and its Bott-Chern cohomology class does not depend on the generalized metric and the non-canonical identification \eqref{Equ0-1}.
Moreover, it is compatible with pullbacks and tensor products.	

In particular, the Chern character $\mathrm{ch}_{\mathrm{BC}} \br{E, \AEpp} \in \HEXBCR$  for the antiholomorphic superconnection  $\br{E, \AEpp}$ is well-defined.

\end{thm}

Thanks to Theorem \ref{Thm0-8} , the authors  proved that there is a well-defined Chern character with values in Bott-Chern cohomology class for objects in $\DbCohX$ and the map descends to 
\begin{align} \label{Equ0-3}
K \br{X} \rightarrow  H^{(=)}_{\mathrm{BC}}(X, \mathbf{R})	.
\end{align}
This Chern character map satisfies \eqref{Con0-2}-\eqref{Con0-3} in Theorem \ref{Thm0-7} with trivial $G$.

\subsubsection{Our construction of the equivariant Chern character }
The first major contribution of this paper is to extend the aforementioned construction to the $G$-equivariant setting. 
We inspired by the approach of \cite{BSW}, first constructing antiholomorphic $G$-superconnections, and then proving the corresponding category equivalence result analogous to Theorem \ref{Thm0-2}. 

While defining antiholomorphic $G$-superconnections and the category $\underline{\mathrm{B}} \br{X, G}$ is relatively straightforward, establishing this category equivalence presents significant challenges.

Our first step  remains to generalize Proposition \ref{Prop0-1}. To achieve it, we have the following key observation.
\begin{prop}[See Proposition \ref{LocalResolution}] \label{prop0-5}
If $\F$ is a $G$-coherent sheaf on $X$, then $\F$ admits local projective resolutions by $G$-equivariant holomorphic vector bundles.
Moreover, in this resolution, except for the most left nontrivial one, the remaining vector bundles can be chosen to be $G$-trivial.
\end{prop}
The advantage of a locally $G$-trivial vector bundle is that we can extend it to the entire manifold. This fact will play an important role in the subsequent proof.

Let $\E$ be an object in  $\DbCohXG$,  we define $\E^\infty$ as in \eqref{Equ0-3}. Now $\E^\infty$ is also $G$-equivariant.
Thanks to Proposition \ref{prop0-5}, we can generalize Proposition \ref{Prop0-1}.
\begin{prop}[See Proposition \ref{Illusie}]  \label{Prop0-2}
The sheaf $\E^\infty$ admits a global resolution of $G$-equivariant complex of finite dimensional smooth vector bundles.
\end{prop}
 According to Proposition \ref{Prop0-2}, by applying recursion argument similar to those used in \cite[Theorem 6.3.6]{BSW}, we obtain a  category equivalence result extending Theorem \ref{Thm0-2}, 
\begin{thm}[See Theorem \ref{ThmEquivalenceCategory}] \label{Thm0-3}
There is a natural functor $\underline{F}_X : \underline{\mathrm{B}} \br{X, G} \rightarrow \DbCohXG$	 which is an equivalence of triangulated categories.
\end{thm}
The pullbacks and tensor products of antiholomorphic $G$-superconnections are naturally $G$-equivariant, and also the functor $\underline{F}_X$ is compatible with pullbacks and tensor products.

Having established this category equivalence, the construction of equivariant Chern characters for $G$-coherent sheaves becomes relatively straightforward. 
Essentially, we define $G$-invariant generalized metrics on $D$ and we can follow the construction pattern used above for Chern characters of coherent sheaves, ultimately allowing us to define a map
$\chgBC: K \br{X,G}  \rightarrow  H^{(=)}_{\mathrm{BC}}(X_g, \mathbf{C})$ with desired compatibility properties \eqref{Con0-1}-\eqref{Con0-3} in Theorem \ref{Thm0-7}. 

\subsection{The Riemann-Roch-Grothendieck theorem}
Our second major result in this paper is proving the corresponding version of the Riemann-Roch-Grothendieck theorem, namely \eqref{Con0-4} in Theorem \ref{Thm0-7}.

Let $f: X \rightarrow Y $ be a $G$-equivariant holomorphic map between compact complex $G$-manifolds.
Let $i: x \in X \rightarrow \br{x, f(x)} \in X\times Y$ be the embedding associated to $f$, let $p: X\times Y \rightarrow Y $ be the obvious projection, then $i, p$ are $G$-equivariant. 
We have the following diagram,
\begin{equation} 
\begin{tikzcd}[column sep=4em, row sep=2.5em]
{K(X,G)} \arrow[r, "i_!"] \arrow[d, "\tdgBC(TX) \chgBC" description] \arrow[rr, "f_!", bend left=30] & {K(X\times Y, G)} \arrow[r, "p_!"] \arrow[d, "\tdgBC(T(X\times Y)) \chgBC" description] & {K(Y,G)} \arrow[d, "\tdgBC(TY) \chgBC" description] \\
\HEXgBCC \arrow[r, "i_*"] \arrow[r] \arrow[rr, "f_*", bend right=30] & \HEXYgBCC \arrow[r, "p_*"] & \HEYgBCC .    
\end{tikzcd}
\end{equation}
As in \cite{BS58}[Section 7], according to the above diagram and the functoriality of the equivariant Chern characters, we can reduce the proof to the case where $f$ is an equivariant embedding or $f$ is an equivariant projection.

\subsubsection{The Riemann-Roch-Grothendieck for equivariant embeddings}
Let $i_{X,Y}: X \rightarrow Y$ be a $G$-equivariant holomorphic embedding of compact complex $G$-manifolds.
Then for any $g\in G$, $i_{X_g, Y_g}: X_g \rightarrow Y_g$ is an embedding of compact complex manifolds.

Let $\F$ be an object in $\DbCohXG$. 
\begin{thm}[See Theorem \ref{Thm7-4}] \label{Thm0-4}
For any $g\in G$, the following identity holds,
\begin{align} 
\chgBC \br{Ri_{X,Y,*} \F} = i_{X_g,Y_g, *}	\frac{\chgBC \br{\F}}{\tdgBC \br{N_{X/Y}}} \quad \text{in} \; \HEYgBCC.
\end{align}
\end{thm}

The proof can be divided into several steps.

\textbf{Step $1$}: In Section \ref{MainImmersion}, we reduce the problem to the case where $G$ acts trivially on $X, Y$.

\textbf{Step $2$}:  Through the deformation to the normal cone method \cite{BFM}, we can reduce our embedding problem to that of embedding of $X$ into a projectivization space $P$.

\textbf{Step $3$}: We can replace $\F \in \DbCohXG$ by $\ox$ by using the projection formula.

\textbf{Step $4$}: We can then complete our proof by explicitly constructing a Koszul resolution for the direct image of $\ox$ on $P$.

\subsubsection{The Riemann-Roch-Grothendieck for equivariant projections}
Let $X$, $S$ be compact complex $G$-manifolds of dimension $n$, $n^\prime$.
Put
\begin{align}
M = X \times S.
\end{align}
Let $p: M \rightarrow S, q: M \rightarrow X$ be the projections and let $p_g: M_g \rightarrow S_g, q_g: M_G \rightarrow X_g $ be the induced projections on fixed points.
Let $\F$ be an object in $\DbCohMG$.
\begin{thm}[See Theorem \ref{Thm9-1}] \label{Thm0-5}
For $g\in G$, the following identity holds,
\begin{align} 
\chgBC \br{Rp_*\F} = p_{g, *} \brr{q_g^* \tdgBC \br{TX} \chgBC \br{\F}}	 \quad \text{in} \;\HESgBCC.
\end{align}
\end{thm}
\textbf{Step $1$}:
In Section  \ref{Section91}, by the base change theorem \cite[Proposition 3.1.0]{SGA6} and functoriality of $\chgBC$ with respect to equivariant pullback, we can reduce the proof of Theorem \ref{Thm0-5} to the case when $G$ acts trivially on $S$. 

\textbf{Step $2$}:
By Theorem \ref{Thm0-3}, we may replace $\F$ by some antiholomorphic $G$-superconnection $\br{E, \AEpp}$. Denote $\E = \br{\om^\infty \br{E}, \AEpp}$.

\textbf{Step $3$}:
By a theorem of Grauert \cite[Theorem 10.4.6]{GrR84}, $Rp_* \E$ is an object in $\DbCohSG$. By \cite[Theorem 9.5]{Bor87} and by \cite[Proposition IV.4.14]{demailly1997complex}, and since $\E$ is a bounded complex of soft $\om$-modules, we have
\begin{align}
Rp_* \E = p_* \E.	
\end{align}
In fact, $p_* \E$ it is not an object in $\underline{\mathrm{B}} \br{S, G}$ since it is infinite-dimensional.
We can endow $p_* \E$ with an infinite-dimensional antiholomorphic $G$-superconnection.
Formally adapting the construction used for finite-dimensional antiholomorphic $G$-superconnections and using the elliptic theory, we can construct a Chern character forms which will be called the equivariant elliptic superconnection forms.

In \cite[Sections 8.10 and 11.1]{BSW}, the authors developed a technique of spectral truncations.
Utilizing this result, we can prove that the equivariant elliptic superconnection forms computes the equivariant Chern character forms of $p_* \E$ constructed previously.

\textbf{Step $4$}:
We  employ the techniques of the local family index theorem to calculate the equivariant elliptic Chern character of infinite-dimensional $G$-superconnections.

When $\E$ is a $G$-equivariant holomorphic vector bundle and the direct image $Rp_* \E$ is locally free, the problem was fully solved by Bismut \cite{B13} using the infinite-dimensional elliptic superconnections \cite{BGS2}, \cite{BGS3} and their hypoelliptic deformations, as well as local index theory.
Specifically, when $X$ is K\"{a}hler, the theorem is exactly the equivariant version of the family local index theorem for K\"{a}hler fibration \cite[Theorem 2.11]{BGS2}.
 Furthermore, if $X$ has a metric such that the corresponding fundamental (1,1)-form is $\dbx \dx $-closed, the required result is an adiabatic limit of the results of \cite{B89} as explained in \cite{B13}.
For the general case, Bismut constructed an exotic superconnection in \cite{B13} with hypoelliptic curvature on the enlarged fiber $TX$, which ultimately prove the theorem. 

When $\E$ is an antiholomorphic superconnection and $G$ is trivial, Bismut, Shen, Wei \cite{BSW} followed the above constructions and generalized the theorem to this case.

Our proof is to further extend these results to antiholomorphic $G$-superconnections. The fundamental technique employed is Bismut's hypoelliptic Laplacian, and our proof essentially combines the approaches from \cite{B13} and \cite{BSW}.

\subsection{The unicity theorem}
We will also prove a unicity theorem for the equivariant Chern character as stated in Theorem \ref{Thm0-6}. This theorem is a generalization of \cite[Theorem 2]{G10} and \cite[Theorem 9.4.1]{BSW} when $G$ is trivial.

This proof can be divided into the following steps.
Firstly, since $\cgBCBar$ satisfies the condition \eqref{Con0-1} in Theorem \ref{Thm0-7}, we can reduce to the case when $G$ acts trivially on $X$.

Secondly, if $\F$ is a $G$-coherent sheaf on $X$, then by Hironaka's flattening theorem and desingularization, and conditions \eqref{Con0-2} and \eqref{Con0-3} in Theorem \ref{Thm0-6}, we can reduce the problem to the case where $\F$ is a $G$-coherent sheaf with supported in a normal crossing divisor $D = \sum_i D_i$.

Then we have the following fact.
\begin{prop}[See Proposition \ref{Prop8-2}]  \label{Prop0-3}
Let $\iota_i: D_i \rightarrow X$ be the embedding.
If $\F$ is a $G$-coherent sheaf on $X$ and $\mathrm{supp}\F \subset D $, then there exist $G$-coherent sheaves $\xi_i $ on $ D_i$ such that
\begin{align}
  \sum_{i=1}^{k} \iota_{i, *} \xi_i = \F  \quad \text{in} \; K(X,G).
\end{align}

\end{prop}
Finally by Condition \eqref{Con0-4}$'$ in Theorem \ref{Thm0-6} and Proposition \ref{Prop0-3}, the proof of the unicity theorem can be completed by induction on the dimension of the manifold.

\subsection{The organization of the article}
The paper is organized as follows.
Section \ref{SectionPre} is the preliminaries, we will recall various properties of group algebra, equivariant holomorphic vector bundles and their characteristic classes with values in Bott-Chern cohomology.

In Section \ref{SectionEquOX}, we define equivariant $\ox$-modules and their morphisms.
We prove an essential theorem to give a stronger local resolution for $G$-coherent sheaves.

In Section \ref{SectionEquDer}, we study various properties of the derived category $\DbCohXG$.

In Section \ref{EquiAntiSupe}, we introduce antiholomorphic $G$-superconnection $\br{E, \AEpp}$ which is a generalization of antiholomorphic superconnection introduced in \cite{B110} and \cite{BSW}. We also study the morphisms, pullbacks and tensor products of antiholomorphic $G$-superconnections.

In Section \ref{SectionEquivalence}, we generalize the result of Block \cite{B110} and \cite{BSW} to prove the equivalence of categories $\underline{\mathrm{B}} \br{X, G} \simeq \DbCohXG$.

In Section \ref{EquiGeneMetr}, we generalize the constructions in \cite[Sections 7 and 8]{BSW}.
Given an identification $E \simeq E_0$, with $E_0 = \lxb \wo D$, we define the $G$-invariant generalized metric $h$ on $D$, and the adjoint $\AEzerop$ of the antiholomorphic $G$-superconnection $\AEzeropp$. 
We introduce the curvature $\brr{\AEzeropp, \AEzerop}$ and construct the equivariant Chern character form $\chg \br{\AEzeropp, h} \in \Omega^{(=)} \br{X_g, \C}$, whose Bott-Chern cohomology class does not depend on $h$ or on the identification and will be denoted as $\chgBC \br{\AEpp}$.
We will also show that the construction gives a map $\DbCohXG \rightarrow \HEXgBCC$ and therefore gives a map $K \br{\mathrm{coh} \br{X, G}} \rightarrow \HEXgBCC$.

In Section \ref{SectionMainResult}, we state the main theorem of this paper and we prove that we can reduce the problem to equivariant embeddings or equivariant projections.

In Section \ref{SectionEmbedding}, we prove the equivariant RRG theorem for equivariant embeddings.
 We first reduce the proof to the case where $G$ acts trivially on the manifolds, then complete the proof by the method of deformation to the normal cone.
 
 In Section \ref{SectionUnicity}, we prove that under certain conditions, our definition of equivariant Chern character with values in Bott-Chern cohomology class is unique.
 
 In Section \ref{SectionSubmersion1}, when $f$ is the projection $p: M = S \times X \rightarrow S$ and $\E$ is an antiholomorphic $G$-superconnection on $M$, we first prove that we can reduce the problem to the case when $G$ acts trivially on S.
 We then generalized the constructions of \cite[Section 10]{BSW}. 
 In particular, we regard $p_* \E$ as an infinite-dimensional antiholomorphic $G$-superconnection on $S$ and formally adapt the construction used for finite-dimensional antiholomorphic $G$-superconnections to construct the Chern character forms on it. Finally we prove the construction is exactly the same the equivariant Chern character defined in Section \ref{EquiGeneMetr}.
 
 In Section \ref{SectionHypo1}, we follow the constructions in \cite[Sections 13 and 16]{BSW} to construct superconnections with hypoelliptic curvatures.
 Specifically, we consider a new fibration $\M \rightarrow S$ with fiber $\X$, the total space of $TX$.
 We also consider the tautological Koszul complex associated with the embedding $M \rightarrow \M$. 
 We extend the $G$ action to a $G\times \Z_2$ action. Given an identification $E \simeq E_0$, and $G$-invariant metrics on certain objects, we can construct a $G\times \Z_2$-invariant non-degenerate Hermitian for $\epsilon_{X, \theta}$ depending on two parameters $c, d$.
 The adjoint of our new antiholomorphic $G\times \Z_2$-superconnection $\AAAA_Y''$ over $S$ is constructed using  $\epsilon_{X, \theta}$. 
 The corresponding curvature is a fiberwise hypoelliptic operator.
 We define the corresponding hypoelliptic superconnection forms and prove it coincides with the elliptic superconnection forms defined in Section \ref{SectionSubmersion1}.
 
 In Section \ref{SectionSubmersion2}, we prove the Riemann-Roch-Grothendieck Theorem for $p$ which completes the proof of our theorem.

\section{Preliminaries} \label{SectionPre}
The purpose of this section is to recall some elementary facts which will be used in the whole paper. This section is organized as follows.

In Section \ref{SecGroupAlgebra}, for a finite group $G$, we recall the definition and some basic properties of the group algebra $R(G)$.

In Section \ref{RepAndVec}, we recall some fundamental properties of equivariant holomorphic vector bundles.

In Section \ref{Section2-2}, we recall some basic facts on the Bott-Chern cohomology class.

Finally, in Section \ref{EquiCharClass}, we explain the construction of characteristic classes of equivariant holomorphic vector bundles with values in Bott-Chern cohomology.

\subsection{Group algebra} \label{SecGroupAlgebra}
We denote $R(G)$ to be the group algebra of $G$ over $\C$, that is, $R(G)$ is the $\C$-algebra consisting of all formal linear combinations of the distributions concentrate on each element $g \in G$, with the operations of addition and multiplication in a quite natural way:
\begin{align}
\sum_{g \in G} s_g \delta_g	+ \sum_{g \in G} t_g \delta_g = \sum_{g \in G} (s_g+t_g) \delta_g, \\
\sum_{g \in G} s_g \delta_g	\cdot \sum_{g \in G} t_g \delta_g =  \sum_{g \in G} (\sum_{h \in G} s_h t_{h^{-1} g}) \delta_g,
\end{align}
where $s_g, t_g \in \C$ and $\delta_g$ is the distribution concentrates on $g \in G$.
Then it is clear that the category of complex representations of $G$ is equivalent to the category of $R(G)$-modules.

In the sequel, we will sometimes regard $R(G)$ as a $G$-representation by the action
\begin{align}
(g, \sum_{h \in G} s_h \delta_h) \in (G, R(G)) \rightarrow \sum_{h \in G} s_{h g} \delta_{h} \in R(G).	
\end{align} 

\subsection{Equivariant holomorphic vector bundles} \label{RepAndVec}

Let $X$ be a compact complex manifold of dimension $n$.
Let $G$ be a finite group with identity element $e \in G$. 
Assume $G$ acts holomorphically on $X$.

For $x \in X$, we denote $G_x \subset G$ the stabilizer of $x$,
\ie 
\begin{align}
	G_x = \brrr{g \in G \mid g \cdot x = x}.
\end{align}
Then $G_x$ is a subgroup of $G$.

If $V$ is a finite dimensional complex vector space, and if $\rho: G \rightarrow \mathrm{GL}(V)$ is a representation of $G$, we can construct a $G$-equivariant vector bundle $X \times V$ on $X$ by setting 
\begin{align}
g \cdot \br{x,v} = \br{g \cdot x, \rho\br{g} \br{v}}.
\end{align}
A $G$-equivariant holomorphic vector bundle isomorphic to $X \times V$ will be called $G$-trivial. 
This definition is also justified by the following proposition.

Let $E$ be a $G$-equivariant holomorphic vector bundle on $X$. 
\begin{prop} \thlabel{LocTri}
Locally $E$ is a trivial $G_x$-equivariant vector bundle, \ie for any $x \in X$, we can find a $G_x$-invariant open neighborhood $U_x$ of $x$ such that $E_{U_x}$ is $G_x$-trivial.
\end{prop}
\begin{proof}
Given a trivialization of $E$ on a $G_x$-invariant open neighborhood, $E_{U_x} \simeq U_x \times V$, then we can write the $G_x$-action as
\begin{align}
g \cdot (y,v) = \br{g \cdot y, g_y \cdot v},
\end{align}
where $g_y \in \mathrm{End}(V)$ emphasis that the action of $g \in G_x$ on $V$ is dependent on $y \in U_x$ in general.

Set 
\begin{align}
\phi_y =\frac{1}{|G_x|} \sum_{g \in G_x} g_x g_y^{-1}.
\end{align}
It is clear that $\phi_x =1$, so in a probably smaller $G_x$-invariant open neighborhood of $x$, $\phi_y$ is invertible.
Moreover $ \phi_{g \cdot y} g_y  = g_x \phi_y$ for any $y \in U_x$ and $g \in G_x$.
Therefore, under the frame transformation
\begin{align}
\phi (y, v)	= (y, \phi_y v)
\end{align}
the $G_x$ action is independent of $y \in U_x$ under the new trivialization.
The proof of our proposition is complete.
\end{proof}

Let $U_x$ be a $G_x$-invariant neighborhood of $x$ in $X$ as in \thref{LocTri}, and assume $E_{\mid U_x} \simeq U_x \times V$ for certain $G_x$-representation
\begin{align}
\rho: G_x \rightarrow \mathrm{GL} (V).	
\end{align}

Set
\begin{align}
U = \bigcup_{g \in G} g \cdot U_x.
\end{align}
Then $U$ is $G$-invariant.
We may assume for any $g \neq e$, $g \cdot U_x \cap U_x = \emptyset$.

Set
\begin{align}
\widetilde{V} = G \times_{G_x} \br{U_x \times V}	,
\end{align}
then $\widetilde{V}$ is a $G$-equivariant holomorphic vector bundle on $U$. 
Moreover, $\widetilde{V} \simeq E_{\mid U}$ .
\begin{prop}
If $\rho$ is a restriction of a $G$-representation
\begin{align}
\widetilde{\rho}: G \rightarrow \mathrm{GL}(V),
\end{align}
then we have a $G$-equivariant identification $E_{\mid U} \simeq U \times V$.
In particular, $E_{\mid U}$ is a trivial $G$-equivariant vector bundle on $U$.
\end{prop}
\begin{proof}
The map
\begin{align}
\br{g, \br{y, v}} \in G \times \br{U_x \times V } \rightarrow \br{g \cdot y, \widetilde{\rho}\br{g} \br{v}}	\in U \times V
\end{align}
induces the required $G$-equivariant identification $E_{\mid U} \simeq U \times V$, and the proposition holds.
\end{proof}


\subsection{The Bott-Chern cohomology} \label{Section2-2}
Let $X$ be a compact complex manifold of dimension $n$. Let $TX$ be the holomorphic tangent bundle of $X$, and let $T_\R X$ be the corresponding real tangent bundle.
If $T_\C X = T_\R X \otimes_\R \C$, then 
\begin{align}
T_\C X = TX \oplus \overline{TX}. 
\end{align}

Let $\Lambda \br{T^*_\C X}$ be the exterior algebra of $T^*_\C X$. 
Then $\Lambda \br{T^*_\C X}$ is bigraded, $\Lambda \br{T^*_\C X} = \bigoplus_{0\leq p,q \leq n } \Lambda^{p,q} \br{T^*_\C X} $.


Let $\Omega^{p,q} (X, \C)$ be the space of smooth sections of $\Lambda^{p,q} \br{T^*_\C X}$.
Note that the exterior differential $d^X$ splits as $d^X = \dbx + \partial^X$.



\begin{defn}
The Bott-Chern cohomology $H^{p,q}_{\mathrm{BC}} \br{X, \C}$ are given by
\begin{align}
	H^{p,q}_{\mathrm{BC}} \br{X, \C} = \br{ \Omega^{p,q} \br{X, \C} \cap \ker d^X } / \dbx \partial^X \Omega^{p-1,q-1} (X,\C).
\end{align}
\end{defn}
If $\alpha \in \Omega^{p,q} (X, \C)$ is closed, let $\brrr{\alpha}$ be the class in $H^{p,q}_{\mathrm{BC}} \br{X, \C}$.
Note that $H^{\bullet,\bullet}_{\mathrm{BC}}(,X,\C)$ is a bigraded algebra.

\begin{rmk}
There is a canonical injective map $H^{p,q}_{\mathrm{BC}} \br{X, \C} \rightarrow H^{p+q}(X,\C)$ so that 
the Bott-Chern class is a refinement of the de Rham class.
\end{rmk}

Put 
\begin{align}
&\OEXC = \bigoplus_{0 \leq p \leq n} \Omega^{p,p} (X, \C),	&& \HEXBCC = \bigoplus_{0 \leq p \leq n} H^{p,p}_{\mathrm{BC}} (X, \C).
\end{align}
Also $H^{p,p}_{\mathrm{BC}} (X, \C)$ ia an algebra.

Let $\mathscr{D} \br{X, \C}$ be the vector space of currents on $X$.
We can also define $\partial^X$ and $\overline{\partial}^X$ on $\mathscr{D}\br{X, \C}$ which map $\mathscr{D}^{p,q}\br{X, \C}$ to $\mathscr{D}^{p+1,q}\br{X, \C}$ and $\mathscr{D}^{p,q+1}\br{X, \C}$ respectively. Also, let $d^X = \partial^X + \overline{\partial}^X$.

By the results of \cite[Theorem 2.2]{SC07} and \cite[Section 2.2]{BSW}, we can define Bott-Chern cohomology using instead $\mathscr{D} \br{X, \C}$. That is,
\begin{align}
	H^{p,q}_{\mathrm{BC}} \br{X, \C} \simeq \br{ \mathscr{D}^{p,q} \br{X, \C} \cap \ker d^X } / \dbx \partial^X \mathscr{D}^{p-1,q-1} (X,\C).
\end{align}

Let $Y$ be another compact complex manifold of dimension $n'$, and let $f: X\to Y$ be a holomorphic map.
Then $f^*$ maps $\Omega\br{Y,\C}$ into $\Omega \br{X, \C}$ as a morphism of bigraded algebras. Therefore $f^*$ induces a morphism of bigraded algebras
\begin{align} \label{PullbackCoh}
	H_{\mathrm{BC}} \br{Y, \C} \rightarrow H_{\mathrm{BC}} \br{X, \C}.
 \end{align}

By duality, $f_*$ maps $\mathscr{D}^{p, q} \br{X, \C}$ into $\mathscr{D}^{n'-n+p, n'-n+q}  \br{Y, \C}$. Since the Bott-Chern cohomology can be defined using currents, we get a morphism of bigraded vector spaces
\begin{align} \label{PushforwardCoh}
	H_{\mathrm{BC}} \br{X, \C} \rightarrow H_{\mathrm{BC}} \br{Y, \C}.
 \end{align}
The pullback map \eqref{PullbackCoh} and the pushforward map \eqref{PushforwardCoh} are functorial.
 
\subsection{Equivariant characteristic class} \label{EquiCharClass}

Let $G$ be a finite group. We assume that $G$ acts holomorphically on $X$. Let $g^{TX}$ be a $G$-invariant Riemannian metric on $T_\R X$, and let $\nabla^{TX}$ be the Chern connection on $\br{TX, g^{TX}}$, then $\nabla^{TX}$ is $G$-invariant.

If $g \in G$, let $X_g \subset X$ be the fixed point set of $g$.

Let $N_{X_g/X}$ be the normal bundle to $ T X_g$ in $TX|_{X_g}$ and we identify it with the orthogonal bundle to $ T X_g$ in $TX|_{X_g}$ with respect to $g^{TX}$.
Then $g$ acts on $N_{X_g/X}$.
Since $TX_g$ is the eigenbundle of the action of $g$ on $TX|_{X_g}$ with eigenvalue $1$ and $g^{TX}$ is $G$-invariant, the splitting
\begin{align} \label{Split2}
TX |_{X_g} = TX_g \oplus N_{X_g/X}
\end{align}
is holomorphic.

Let $e^{i \theta_1}, \cdots, e^{i \theta_q}\; (0<\theta_j<2 \pi)$ be the locally constant distinct eigenvalues of $g$ acting on $N_{X_g/X}$, and let $N_{X_g/X}^{\theta_1}, \cdots, N_{X_g/X}^{\theta_q}$ be the corresponding eigenbundles.
Then $N_{X_g/X}$ splits holomorphically as 
\begin{align}
N_{X_g/X} = N_{X_g/X}^{\theta_1} \oplus \cdots N_{X_g/X}^{\theta_q}.	
\end{align}

Let $g^{TX_g}$, $g^{N_{X_g/X}^{\theta_1}}$, $\cdots$ be the Hermitian metrics induced by $g^{TX \mid_{X_g}}$ on $TX_g$, $N_{X_g/X}^{\theta_1}$, $\cdots$.
Then $\nabla^{TX \mid_{X_g}}$ induces the Chern connections $\nabla^{TX_g}$, $\nabla^{N_{X_g/X}^{\theta_1}}$, $\cdots$ on $\br{TX_g, g^{TX_g}}$, $\br{N_{X_g/X}^{\theta_1}, g^{N_{X_g/X}^{\theta_1}}}$, $\cdots$.
Let $R^{TX_g}$, $R^{N_{X_g/X}^{\theta_1}}$, $\cdots$ be their curvatures.

If $A$ is a matrix, set
\begin{align}
&\Td \br{A} = \det \br{\frac{A}{1- e^{-A}}},  && e(A) = \det (A).
\end{align}

\begin{defn} \label{DefTdg}
Set
\begin{align} 
\Td_g \br{TX, g^{TX}} &= \Td \br{-R^{TX_g} / 2 i \pi} \prod_{j=1}^q \br{\frac{\Td}{e}} \br{-\frac{R^{N_{X_g/X}^{\theta_j}}}{2 i \pi} + i \theta_j} \in \OEXgC \label{EquTodd} 
\end{align}
\end{defn}
Let $E$ be a $G$-equivariant holomorphic vector bundle on $X$. Let $g^E$ be a $G$-invariant Hermitian metric on $E$. 
Let $\nabla^E$ be the Chern connection on $(E, g^E)$ and let $R^E$ be the corresponding curvature. Note that $\nabla^E$ is $G$-invariant.
\begin{defn} \label{DefChg}
Set
\begin{align} \label{EquChern}
	\chg \br{E, g^E} &= \mathrm{Tr} \brr{g \exp \br{\frac{-R^{E \mid_{X_g}}}{2 i \pi}}} \in \OEXgC.
\end{align}
\end{defn}

The forms in (\ref{EquTodd}) and (\ref{EquChern}) are closed forms in $\OEXgC$.

The class $\brrr{\Td_g \br{TX, g^{TX}}}  \in \HEXgBCC$ and $\brrr{\chg \br{E, h^E}}  \in \HEXgBCC$ are independent of the choice of $g^{TX}$ and $g^E$ \cite{BC65, BGS1}.
These classes are denoted by $\tdgBC (TX)$ and $\chgBC (E)$.

\section{The category of equivariant $\ox$-modules}  \label{SectionEquOX}
The purpose of this section is to introduce the category of equivariant $\ox$-modules and establish some essential properties that will be used in the following sections.

In Section \ref{Section2-1}, for a compact complex manifold $X$ with a finite group action, we introduce the basic concepts of equivariant $\ox$-modules and $\ox$-morphisms.

In Section \ref{ForExt}, we introduce a canonical equivariant extension of $\ox$-morphisms.

In Section \ref{Complex}, we recall the constructions of cones in homological algebra for equivariant $\ox$-complexes.

In Section \ref{GPullBack}, we construct the direct image and pullback functors on the category of equivariant $\ox$-modules and $\ox$-complexes.

In Section \ref{Section3-5}, for a equivariant $\ox$-module which is simultaneously coherent, we study the local resolution property of it. 
\subsection{Equivariant $\ox$-modules and $\ox$-morphisms} \label{Section2-1}

Let $X$ be a compact complex manifold of dimension $n$.
Denote $\ox$ and $\ox^\infty$ the sheaves of holomorphic and smooth functions on $X$ respectively.
If $E$ is a holomorphic vector bundle, denote $\ox(E)$ or $\ox^\infty(E)$ the  sheaf of holomorphic or smooth sections of $E$.

Let $G$ be a finite group with identity element $e \in G$. 
Assume $G$ acts holomorphically on $X$.
If $g \in G$, we denote 
\begin{align}
m_g: x \in X \mapsto g \cdot x \in X,
\end{align}
the corresponding holomorphic map.
Then $m_g$ induces a morphism of sheaves of rings
\begin{equation} \label{RingMorphism} 
g: \ox \rightarrow m_g^{-1} \ox.
\end{equation}
If $U$ is an open subset in $X$, the map ($\ref{RingMorphism}$) sends $s \in \ox \br{U}$ to $s \br{g^{-1}  \cdot} \in \ox \br{g U} $.

Let $\F$ be an $\ox$-module, then $m_g^{-1} \F$ is  an $m_g^{-1} \ox$-module. 
This defines a covariant functor between the categories of $\ox$-modules and $m_g^{-1} \ox$-modules. 
\begin{defn} \thlabel{GSheaf}
 We call $\F$ is $G$-equivariant if for any $g \in G$, there exists an isomorphism of modules\footnote{This is an isomorphism between modules over different sheaves of rings.}
\begin{equation} \label{GActionOnSheaf}
g: \F \rightarrow m_g^{-1} \F,
\end{equation}
such that
\begin{enumerate}
\item the identity element $e \in G$ acts as identity; \label{GModuleOne} 
\item for any $g, h \in G$, the diagram  \label{GModuleTwo} 
\begin{equation}
\begin{tikzcd}
\mathscr{F} \arrow[rr, "hg"] \arrow[rd, "g"'] &                                             & m_{hg}^{-1} \mathscr{F} \\
                                              & m_g^{-1} \mathscr{F} \arrow[ru, "m_g^{-1}(h)"'] &                      
\end{tikzcd}	
\end{equation}
commutes.
\end{enumerate}
Equivalently, by extension of scalars, for any $g \in G$, there exists an isomorphism of $\ox$-modules
\begin{equation} \label{GroupAction}
g: \F \rightarrow m_g^* \F	
\end{equation} 
such that similar properties (\ref{GModuleOne}) and (\ref{GModuleTwo}) hold.
\end{defn}

In the sequel, we assume $\F$ is a $G$-equivariant $\ox$-module.
If there is no confusion, we denote $m_g^{-1} \F$, $m_g^* \F$ by $g^{-1} \F$ and $g^* \F$.

If $U_x$ is a $G_x$-invariant neighborhood of $x \in X$, then 
$m_g^{-1} \F \br{U_x} = \F (U_x)$, so that (\ref{GActionOnSheaf}) induces a $G_x$-representation  
\begin{align}
\theta: G_x \rightarrow \mathrm{GL}(\F (U_x)). 	
\end{align}

Set
\begin{align}
U = \bigcup_{g \in G} g \cdot U_x.
\end{align}
Then $\F(U)$ is a $G$-representation.
If $U_x$ is small enough, we have  
\begin{align}
\F (U) = \bigoplus_{g \in G / G_x} \F (g \cdot U_x).	
\end{align}

The map
\begin{align}
(g,s) \in G \times \F(U_x )\rightarrow g \cdot s \in \F (U)
\end{align}
induces an isomorphism of $G$-representations
\begin{align} \label{InducedRepresentation}
G \times_{G_x} \F (U_x) \xrightarrow{\sim} \F (U). 
\end{align}
\begin{rmk} \thlabel{InducedRep}
The left hand side of (\ref{InducedRepresentation}) is the representation $\mathrm{Ind}_{G_x}^G \br{\theta}$ induced by $\theta$ of $G_x$ in $\F (U_x)$.
\end{rmk}

If $Q$ is a $G$-equivariant holomorphic vector bundle on $X$, and if $\F = \ox (Q)$, then $\F$ is a $G$-equivariant $\ox$-module.

Conversely, if $\F$ is a locally free $\ox$-module of finite rank, then it is the sheaf of holomorphic sections of a holomorphic vector bundle $Q$. 
If $x \in X$, denote $\mathbf{m}_x$ the maximal ideal of $\mathcal{O}_{X,x}$.
Then we have the canonical identification of vector spaces
\begin{align}
Q_x \simeq \F_x / \mathbf{m}_x .
\end{align}
If in addition $\F$ is $G$-equivariant, then for any $g \in G$, the map (\ref{GActionOnSheaf}) induces a linear map 
\begin{align} \label{GActionOnFiber}
Q_x \rightarrow Q_{gx},
\end{align}
so that $Q$ is a $G$-equivariant holomorphic vector bundle.

If $F: \E \rightarrow \F$ is a morphism of locally free $
\ox$-modules. Assume $\E$ and $\F$ are the sheaf of holomorphic sections of holomorphic vector bundles $P$ and $Q$ respectively, then $F$ induces a morphism $\tilde{F}: P \rightarrow Q$.
If $F_x: \E_x \rightarrow \F_x$ is surjective, then $\tilde{F}_x: P_x \rightarrow Q_x$ is surjective.
However, if $F_x$ is injective, we can not deduce that $\tilde{F}_x$ is injective.

Let $\E$, $\F$ be two $G$-equivariant $\ox$-modules. Let $F: \E \rightarrow \F$ be a morphism of $\ox$-modules.

\begin{defn}

We say $F$ is $G$-equivariant if for any $g \in G$, the following diagram

\begin{equation} \label{GMorphism}
\begin{tikzcd} 
\mathscr{E} \arrow[r, "F"] \arrow[d, "g"'] & \mathscr{F} \arrow[d, "g"]  \\
m_g^{-1} \mathscr{E} \arrow[r, "m_g^{-1} (F)"]         & m_g^{-1} \mathscr{F}                 
\end{tikzcd}
\end{equation}
commutes.
\end{defn}

From the above, we actually constructed a category $\mathrm{M}(X,G)$.
Its objects consist $G$-equivariant $\ox$-modules and its morphisms are $G$-equivariant $\ox$-morphisms.
If $G$ is trivial, we obtain the category of $\ox$-modules $\mathrm{M} (X)$. 
If $\E, \F$ are two $G$-equivariant $\ox$-modules, then $\E \otimes_{\ox} \F$ and $\Hom_{\ox}(\E, \F)$ are also $G$-equivariant.

Note that if we replace $\ox$ by $\ox^\infty$, the category $\MIXG$ of $G$-equivariant $\ox^\infty$-modules can be defined in the same way. 

\subsection{A canonical equivariant extension of $\ox$-morphisms} \label{ForExt}


If $\E$ is an object in $\mathrm{M} (X)$, 
set
\begin{align} \label{Extension}
\E_G = \bigoplus_{g \in G} m_g^* \E.
\end{align}
The module $\E_G$ equips a canonical $G$-action by translation. 
If $s \in \E_{G,x}$, then we can write $s = (s_{g^\prime})_{g^\prime \in G}$, where $s_{g^\prime} \in (m_{g^\prime}^* \E )_x$.
Then the action of $g \in G$ on $s$ is defined to be 
\begin{align}
g \cdot s = (s_{g\prime g} )_{g^\prime \in G}.	
\end{align}

Since $\E = m_e^* \E$, by (\ref{Extension}), we have the canonical inclusion,
\begin{align} \label{CanInc}
\iota: \E \rightarrow \E_G.	
\end{align}
If $\F$ is a $G$-equivariant $\ox$-module, and if $F: \E \rightarrow \F$ is a $\ox$-morphism, the following proposition gives us the way to extend $F$ to be $G$-equivariant.
\begin{prop} \label{PropExtend}
The following statements hold.
\begin{enumerate}
\item \label{LiftOfExt} There exists a unique $G$-equivariant $\ox$-morphism
$F_G: \E_G \rightarrow \F$
such that the diagram
\begin{equation} \label{LiftExt}
\begin{tikzcd}
\mathscr{E} \arrow[d, "\iota"'] \arrow[r, "F"] & \mathscr{F} \\
\mathscr{E}_G \arrow[ru, "F_G"', dashed]                     &            
\end{tikzcd}
\end{equation}
commutes.
\item The extension $\E_G$ is unique up to canonical $G$-equivariant isomorphism. 
\end{enumerate}	
\end{prop}
\begin{proof} 
The morphism $F: \E \rightarrow \F$ induces the morphism 
\begin{align}
m_g^* (F): m_g^* \E \rightarrow m_g^* \F .
\end{align}
By (\ref{GActionOnSheaf}), we get a morphism
\begin{align}
m_g^* \E \rightarrow \F.	
\end{align}
By (\ref{Extension}), we get 
\begin{align} \label{ExtMap}
F_G: \E_G \rightarrow \F.	
\end{align}


It is easy to see that the morphism $F_G$ is $G$-equivariant and satisfies (\ref{LiftExt}). The uniqueness is clear from our constructions.

The proof of the second statement is standard.
\end{proof}

\begin{exmp} \thlabel{ExtLine}
The $G$-equivariant $\ox$-module $\brr{\ox}_G$ is the sheaf of holomorphic sections of the trivial $G$-equivariant 	vector bundle $X \times R(G)$.
\end{exmp}

\subsection{Complex and mapping cone} \label{Complex}
We follow \cite[Section 4.3]{BSW} and extend the corresponding results to equivariant case.

If $\E$, $\F$ are two $G$-equivariant $\ox$-modules, and if  $F: \E \rightarrow \F$	is a $G$-equivariant $\ox$-morphism, then $\ker \br{F}, \cok \br{F}$ are $G$-equivariant $\ox$-modules.
It is easy to see that $\MXG$ is an abelian category.



In particular, we can define the category of $G$-equivariant $\ox$-complexes $\CXG$.
Denote $\CbXG$ the category of bounded $G$-equivariant $\ox$-complexes.
This is a full subcategory of $\CXG$ whose objects consist complexes of finite length.



If $(\F^\bullet, d^{\F})$ is a $G$-equivariant $\ox$-complex, we denote by $\HH^\bullet \F$ its cohomology.
Since $d^\F$ is $G$-invariant, $\HH^\bullet \F$ is a $\Z$-graded $G$-equivariant $\ox$-module.

Let $\br{\E^\bullet, d^{\E}}$, $\br{ \F^\bullet, d^{\F}} $ be two $G$-equivariant $\ox$-complexes.
Let 
\begin{align}
\phi: \br{\E^\bullet, d^{\E}} \rightarrow \br{ \F^\bullet, d^{\F}}
\end{align}
 be a $G$-equivariant morphism of complexes, then $\phi$ preserves the degree, and
\begin{equation}
\phi d^\E = d^\F \phi	.
\end{equation}

Therefore $\phi$ induces a $\Z$-graded equivariant $\ox$-morphism
\begin{align} \label{MapOnCoh}
 \phi: \HH^\bullet \E \rightarrow \HH^\bullet \F. 
 \end{align}
Also $\phi$ is said to be a quasi-isomorphism if (\ref{MapOnCoh}) is an isomorphism.

Put
\begin{equation}
\CC^{\bullet} = \mathrm{cone} \br{\E^\bullet, \F^\bullet}.	
\end{equation}
Then $\br{\CC^\bullet, d_\phi^{\CC}}$ is a $G$-equivariant $\ox$-complex such that 
\begin{equation}
\CC^{\bullet} = \E^{\bullet +1} \oplus \F^{\bullet}.	
\end{equation}
with $d_\phi^{\CC}$ is given by
\begin{equation} \label{Cone}
d_\phi^{\CC} = 
\left[
\begin{array}{c c}
d^\E  &  0 \\
\phi (-1)^{\mathrm{deg}}  &  d^\F \\
\end{array}
\right]	.
\end{equation}

We have the exact sequences of complexes,
\begin{equation} \label{ShortExactSequenceCone}
\begin{tikzcd}
0 \arrow[r] & \F^\bullet \arrow[r] & \CC^\bullet \arrow[r] & \E^{\bullet + 1} \arrow[r] & 0.
\end{tikzcd}
\end{equation}
So there is a corresponding long exact sequence in cohomology,
\begin{equation}
\begin{tikzcd}
\cdots \arrow[r] & \HH^\bullet \F \arrow[r] & \HH^\bullet \CC \arrow[r] & \HH^{\bullet+1} \E \arrow[r, "\phi (-1)^{\mathrm{deg}}"] & \HH^{\bullet +1} \F \arrow[r] & \cdots,
\end{tikzcd}
\end{equation}
from which we see $\phi$ is a quasi-isomorphism if and only if $\HH^\bullet \CC = 0$.



Note that if the complexes $\br{\E^\bullet, d^{\E}}$, $\br{ \F^\bullet, d^{\F}} $ are bounded, then all the associated objects $\HH^\bullet \E$, $\HH^\bullet \F$ and $\CC^\bullet$ are also bounded.

Finally, we note that if $\E^\bullet$, $\F^\bullet$ are $G$-equivariant $\ox^\infty$-complexes, similar constructions still hold.

\subsection{Pullback and direct image} \label{GPullBack}
Let $Y$ be another compact complex $G$-manifold.
Let $f: X \rightarrow Y$ be a holomorphic map.
Then the pull back $f^*$ defines a functor from $\mathrm{M} (Y)$ to $\mathrm{M} (X)$, and extends to a functor $\mathrm{C}(Y)$ to $\mathrm{C}(X)$.
Similarly, the direct image $f_*$ defines a functor from  $\mathrm{M} (X)$ to $\mathrm{M} (Y)$ and extends to a functor from $\mathrm{C} (X)$ to $\mathrm{C} (Y)$.

\begin{prop}
If $f$ is $G$-equivariant, then $f^*$, $f_*$ induce the corresponding $G$-equivariant functors.
\end{prop}
\begin{proof}
The proof is elementary and is left to the readers.	
\end{proof}

\subsection{Equivariant coherent sheaves} \label{Section3-5}

Let $\F$ be a $G$-equivariant $\ox$-module. 
Assume $\F$ is $\ox$-coherent \cite[p.~696]{GH}.
We generalize the classic theorem \cite[p.~696]{GH}.
\begin{prop} \thlabel{LocalResolution}
For any $x \in X$, there exist
\begin{enumerate}
\item a $G$-invariant open neighborhood $U \subset X$ of $x$,
\item a complex of $G$-equivariant holomorphic vector bundles on $U$, concentrated at degree $[-k, 0]$,
\begin{equation}\label{LocRes}
\begin{tikzcd}
 0 \arrow[r] & E^{-k}_{U} \arrow[r, "F_{k}"] & \cdots \cdots \arrow[r, "F_{1}"] & E^0_{U} \arrow[r] & 0,
\end{tikzcd}
\end{equation}
\item a $G$-equivariant $\ou$-augmentation $F_0 : \mathcal{O}_U (E^0_U) \rightarrow \F_{\mid_U}$,
\end{enumerate}
such that the induced $\ou$-complex 
\begin{equation}
\begin{tikzcd}
0 \arrow[r] & \mathcal{O}_U (E^{-k}_{U}) \arrow[r, "F_k"] & \cdots \cdots \arrow[r, "F_1"] & \mathcal{O}_U (E^0_{U}) \arrow[r, "F_0"] & \F_{\mid_U} \arrow[r] & 0
\end{tikzcd}
\end{equation}
is exact.
Moreover, we can choose $k \leq n$, and for any $i \leq k-1$, $E_U^{-i}$ is a trivial $G$-equivariant vector bundle associated to certain $G$-representation.
\end{prop}
\begin{proof}
If $\F_x$ is free over $\mathcal{O}_x$, then $\forall g \in G$, $\F_{gx}$ is free over $\OO_{gx}$, from which we can deduce our proposition with $k=0$.

Otherwise, since $\F$ is coherent, we can find $n_0 \in \N^*$, a small $G_x$-invariant open neighborhood $U_x$ of $x$ and a surjective morphism of $\mathcal{O}_{U_x}$-modules 
\begin{align} \label{CohLocMap}
\tilde{F_0}: \mathcal{O}_{U_x}^{\oplus n_0} \rightarrow \F_{\mid U_x}.
\end{align}
Note that $\tilde{F}_0$ can not be injective since $\F_x$ is not free over $\OO_x$.

We apply the constructions (\ref{ExtMap}) for the pair $\br{G_x, U_x}$.
In particular, the morphism (\ref{CohLocMap}) lifts to a $G_x$-equivariant morphism 
\begin{align} \label{LocSur}
\tilde{F}_{0,G_x}: 
\brr{\mathcal{O}_{U_x}^{\oplus n_0} }_{G_x} \rightarrow \F_{\mid U_x}
\end{align}
such that the diagram
\begin{equation}
\begin{tikzcd}
\mathcal{O}_{U_x}^{\oplus n_0}  \arrow[d] \arrow[r] & \F_{\mid U_x} \\
\brr{\mathcal{O}_{U_x}^{\oplus n_0} }_{G_x} \arrow[ru, "\tilde{F}_{0,G_x}"', dashed]                     &            
\end{tikzcd}
\end{equation}
commutes.
Clearly, $\tilde{F}_{0,G_x}$ is surjective.

By \thref{ExtLine}, we can write (\ref{LocSur}) as
\begin{align} \label{LocSurTwo}
 \OO_{U_x} \br{U_x \times R \br{G_x}^{n_0}} \rightarrow \F_{\mid U_x}.	
\end{align}


Using the canonical morphism of $G_x$-representations $R(G) \rightarrow R(G_x)$, we can extend the map (\ref{LocSurTwo}) to a surjective $G_x$-equivariant $\mathcal{O}_{U_x}$-morphism
\begin{align} \label{ExtLocSur}
\mathcal{O}_{U_x} \br{U_x \times R\br{G}^{n_0} } \rightarrow \F_{\mid U_x}	.
\end{align}

Since $U_x$ is small, and since $R(G)$ is a $G$-representation, it follows from the discussion in Subsection \ref{RepAndVec} that the map (\ref{ExtLocSur}) extends to a surjective $G$-equivariant $\ou$-morphism
\begin{align}
F_0: \ou \br{U \times R (G)^{n_0}} \rightarrow \F_{\mid U}.
\end{align}

If $\ker(F_0)$ is locally free, by the discussion after \thref{InducedRep}, we finish our proof with $k=1$.
Otherwise, $\ker(F_0)$ is still coherent. So we can proceed the same construction for $\ker (F_0)$
in a probably smaller $G$-invariant neighborhood $U^\prime$ of $x$.

By recursion argument, we can construct the complex (\ref{LocRes}), such that $k \leq n$ and $\rm{Ker}(F_{k-1})$ is free on $U$.
\end{proof}

\section{The equivariant derived category} \label{SectionEquDer}
The purpose of this section is to recall basic properties of the derived category $\DbCohX$ of bounded complexes of $G$-equivariant $\ox$-modules with coherent cohomology.

In Section \ref{DerivedCategory}, we give the definition of $\DbCohXG$ and of the associated $K$-group $K(X, G)$.

In Section \ref{DerPullback}, if $Y$ is a compact complex $G$-manifold, and if $f: X\to Y$ is a $G$-equivariant holomorphic map, we define the derived pullback $Lf^*: \DbCohYG \to \DbCohXG$.

In Section \ref{DerivedTensor}, we consider the derived tensor products in $\DbCohXG$.

In Section \ref{DerivedDirect}, we recall some basic properties of the derived direct image $Rf_*: \DbCohXG \to \DbCohYG$.

\subsection{Definition of the derived category $\DbCohXG$} \label{DerivedCategory}

Let $X$ be a compact complex manifold.
Let $G$ be a finite group acts holomorphically on $X$.
Let $\mathrm{coh} (X, G)$ be the abelian category of $G$-equivariant $\ox$-coherent sheaves on $X$.
Let $K \br{ \mathrm{coh}\br{X,G} }$ denote the corresponding Grothendieck group.


Let $\CbCohXG$ be the full subcategory of $\CbXG$ whose objects have coherent cohomology, let $\DbCohXG$ be the corresponding derived category and let $K \br{\DbCohXG}$ be the corresponding $K$-group.

The same arguments as in \cite[\href{https://stacks.math.columbia.edu/tag/0FCP}{Tag0FCP}]{stacks-project} show that the map
\begin{align} \label{MapDbcohToKXG}
\E \in \DbCohXG \rightarrow \sum_i (-1)^i \HH^i \E \in K(X, G)	
\end{align}
induces an isomorphism of $K$-groups,
\begin{align} \label{IsoKDbcohAndKXG}
K \br{\DbCohXG} \simeq K \br{\mathrm{coh}(X,G)}.	
\end{align}
In the sequel, we denote the above group as $K (X, G)$.

\subsection{Pullbacks} \label{DerPullback}
Let $Y$ be another compact complex manifold. 
We use the notations and assumptions as in Subsection \ref{GPullBack}.

By Grauert-Remmert \cite[Section 1.2.6]{GrR84}, if $\E$ is an object in $\CbCohYG$, then $f^* \E$ is coherent, so it is an object in $\CbCohXG$.
We can define the left-derived functor $L f^*$, which to $\E \in \DbCohYG$ associates $Lf^* \E \in \DbCohXG$.
If $\E$ is a bounded complex of $G$-equivariant flat $\oy$-module, by \cite[\href{https://stacks.math.columbia.edu/tag/06YJ}{Tag06YJ}]{stacks-project}, we have the canonical isomorphism,
\begin{align}
Lf^* \E \simeq f^* \E.	
\end{align}
Also $Lf^*$ induces a morphism on Grothendieck groups,
\begin{align}
f^{!}: K (Y, G) \rightarrow K (X, G).	
\end{align}

If $Z$ is another compact complex manifold with holomorphic $G$-actions.
If $h: Y \rightarrow Z$ is a $G$-equivariant holomorphic map, using  \cite[Proposition I.9.15]{Bor87}, there is a canonical isomorphism between the functors $Lf^* Lh^*$ and $L(hf)^*$.
In particular, we get an identity of morphisms of Grothendieck groups,
\begin{align}
f^! h^! = (hf)^!: K (Z, G) \rightarrow K (X, G).	
\end{align}
We note here that $m_g^{!}$ is identity on $K (X,G)$.

\subsection{Tensor products} \label{DerivedTensor}
Let $\E, \F$ be objects in $\DbCohXG$. In \cite[\href{https://stacks.math.columbia.edu/tag/064M}{Tag064M}]{stacks-project}, a derived tensor product $\E \widehat{\otimes}^L_{\ox} \F$, also an object in $\DbCohXG$ is defined.
By \cite[\href{https://stacks.math.columbia.edu/tag/079U}{Tag079U}]{stacks-project}, if $Y$ is a compact complex manifold with holomorphic $G$ actions, and if $f: X \rightarrow Y$ is a $G$-equivariant holomorphic map, if $\E, \F$ are objects in $\DbCohYG$, then we have canonical isomorphism in $\DbCohXG$,
\begin{align} \label{DerTen}
Lf^* \br{\E \widehat{\otimes}^L_{\oy} \F} \simeq Lf^* \E \widehat{\otimes}^L_{\ox} Lf^* \F.
\end{align}

Let $i: X\rightarrow X \times X$ be the diagonal embedding, and let $p_1$, $p_2: X \times X \rightarrow X$ be the two projections. 
Since $p_1 i$, $p_2 i$ are the identity in $X$, using the results of Subsection \ref{DerPullback} and equation (\ref{DerTen}), we find that there is a canonical isomorphism,
\begin{align}
\E \widehat{\otimes}^L_{\ox} \F \simeq Li^* \br{L p_1^* \E \widehat{\otimes}^L_{\mathcal{O}_{X \times X}} L p_2^* \F}.	
\end{align}

By \cite[\href{https://stacks.math.columbia.edu/tag/06XY}{Tag06XY}]{stacks-project}, if $\E$, $\F$ are objects in $\DbCohXG$, if one of them consists of flat modules over $\ox$, then we have canonical isomorphism
\begin{align}
\E \widehat{\otimes}^L_{\ox} \F \simeq \E \widehat{\otimes}_{\ox} \F.	
\end{align}

\subsection{Direct images} \label{DerivedDirect}
Let $Rf_*$ be the right-derived functor of the direct image $f_*$.
By Grauert's direct image theorem \cite[Theorem 10.4.6]{GrR84}, if $\E$ is an object in $\DbCohXG$, $Rf_* \E$ is an object in $\DbCohYG$.

By \cite[Theorem 9.5]{Bor87} and by \cite[Proposition IV.4.14]{demailly1997complex}, if $\E$ is a bounded $G$-equivariant complex of soft $\ox$-modules, then we have the canonical isomorphism in $\DbCohYG$
\begin{align} \label{SoftDirect}
Rf_* \E \simeq f_* \E.	
\end{align}

If $f$ is an embedding, by 
\cite[Corollary 13.9]{demailly1997complex},
\begin{align}
R f_*  = 	f_* .
\end{align}

Also $Rf_*$ induces a morphism of Grothendieck groups,
\begin{align} \label{DirectK}
f_! : K(X, G) \rightarrow K(Y, G).	
\end{align}

If $Z$ is another compact complex manifold with holomorphic $G$-actions.
If $h: Y \rightarrow Z$ is a $G$-equivariant holomorphic map, by \cite[Proposition I.9.15]{Bor87}, there is a canonical isomorphism between the functors $Rh_* Rf_*$ and $R(hf)_*$.
In particular, we have an identity of morphisms of Grothendieck groups,
\begin{align} \label{Equ3-1}
h_! f_! = (hf)_!: K(X, G) \rightarrow K(Z, G).	
\end{align}
Also, $m_{g!} $ is identity on $K (X,G)$.

\section{The equivariant antiholomorphic superconnection} \label{EquiAntiSupe}
The purpose of this section is to describe the antiholomorphic $G$-superconnection. The construction is a natural extension of the antiholomorphic superconnection introduced in \cite{B110} and \cite{BSW}.

In Section \ref{EquAntSup}, we define the antiholomorphic $G$-superconnection $\br{E, \AEpp}$ on a compact complex $G$-manifold $X$.

In Section \ref{TwoCat}, we study the morphism between antiholomorphic $G$-superconnections and thus form a category $\BXG$. Furthermore, we introduce a homotopy category $\BUXG$ associated to $\BXG$.

In Section \ref{BlockCone}, given a morphism of antiholomorphic $G$-superconnections, we construct the corresponding cone.

In Section \ref{BlockPullback}, we introduce the pullbacks and tensor products of antiholomorphic $G$-superconnections.

\subsection{The equivariant antiholomorphic superconnection} \label{EquAntSup}
Let $X$ be a compact complex manifold of dimension $n$.
Let $G$ be a finite group acting holomorphically on $X$.

Observe that $\lxb$ is a $\Z$-graded algebra.
We recall some fundamental constructions in \cite{B110} and \cite[Section 5.1]{BSW}.

Let $E$ be a $\Z$-graded smooth vector bundle of finite rank on $X$.
If $i \in \Z$, we denote $E^i$ the smooth subbundle of $E$ consisting elements of degree $i$.
Then there exist $r$, $r^\prime \in \Z$, $r \leq r^\prime$, such that
\begin{align}
E = \bigoplus_{r\leq i \leq r^\prime} E^i.	
\end{align}

We assume that the exterior algebra $\lxb$ acts on the left on $E$ so that $\Lambda^i \br{\overline{T^* X}}$ increases the degree of $E$ by $i$.
Equivalently, this means that $E$ is a $\Z$-graded module over $\lxb$.
Observe that $E$ is canonically filtrated, \ie
\begin{align} \label{FiltrationE}
E \supseteq \br{\overline{T^*X}} E \supseteq \Lambda^2 \br{\overline{T^*X}} E \supseteq \cdots
\end{align}

Set 
\begin{align} \label{Diagonal}
D = E /  \br{\overline{T^* X}} E.
\end{align} 
Then $D$ equips with an induced $\Z$-grading, and is called the diagonal associated with $E$.

We assume that $D$ is a vector bundle and $E$ is free over $\lxb$. 
Then we have a non-canonical identification of $\lxb$-modules,
\begin{align} \label{EIsoE0}
E \simeq  \Lambda \br{\overline{T^* X}} \widehat{\otimes}  D.
\end{align}
If $i \in \Z$, by (\ref{EIsoE0}), we have 
\begin{align} \label{EIsoE02}
E^i \simeq \bigoplus_{p+q=i}	 \Lambda^p \br{\overline{T^* X}} \widehat{\otimes}  D^q.
\end{align}

In the sequel, we will also use the notation 
\begin{align} \label{DefE0}
E_0 =	\Lambda \br{\overline{T^* X}} \widehat{\otimes}  D
\end{align}
as in \cite[(5.1.1)]{BSW}. Then $E_0$ is bi-graded and has all the properties that $E$ has above with the same diagonal as $E$.

Let $\underline{E}$ be another $\Z$-graded vector bundle like $E$. 
We underline the objects associated with $\underline{E}$.
If $k \in \Z$, if $\phi$ is a linear map from $E$ to $\underline{E}$ that maps $E^\bullet$ to $\underline{E}^{\bullet + k}$, we will say that $\phi$ is of degree $k$.
As in \cite[Section 5.1]{BSW},
we denote $\Hom \br{E, \underline{E}}$ the $\Z$-graded vector bundle of $\lxb$-morphisms from
$E$ to $\underline{E}$, that is if $\alpha \in \lxb$, then
\begin{align}
\phi \alpha - \br{-1}^{\deg \alpha \deg \phi} \alpha \phi = 0.
\end{align}
Then $\Hom \br{E, \underline{E}}$ is an object like $E$. The corresponding diagonal is given by $\Hom \br{D, \underline{D}}$.

When $E = \underline{E}$, we use the notation $\End (E)= \Hom (E, E)$.
Note that 
\begin{align} \label{EndE0}
\End \br{E_0} = \lxb \widehat{\otimes} \End \br{D}.
\end{align}

We assume that $G$ acts as automorphism of $E$, that is, the $G$-action is of degree $0$ and preserves the multiplication by $\lxb$. For any $g \in G$, $\alpha \in \Omega^{0, \bullet} (X, \C)$ and $s \in \CXE$,
\begin{align}
g \cdot \br{\alpha  s} =  \br{g_* \alpha}   \br{g \cdot s}. 	
\end{align}
By (\ref{Diagonal}), $G$ acts naturally on $D$ and preserves the degree.
In particular, $G$ acts on the both sides of (\ref{EIsoE0}) and (\ref{EIsoE02}).

If $G$ acts on $\underline{E}$ as before, then $G$ acts on $\Hom \br{E, \underline{E}}$ in an obvious way. 
\begin{prop} \label{GIdenExist}
\begin{enumerate}
\item There exists a non-canonical $G$-equivariant identifications such that (\ref{EIsoE0}) and (\ref{EIsoE02}) hold.
\item If $\phi$, $\underline{\phi}: E \rightarrow  E_0$ are two such $G$-equivariant identifications as above, then there exists a $G$-invariant smooth section $A$ of $\Hom^0 \br{D,  \br{\overline{T^* X}} E_0}$ such that $\phi = (1+A) \underline{\phi} $.

\end{enumerate}
\end{prop}
\begin{proof}
Let $p: E \rightarrow D$ be the canonical projection. We need to construct a degree preserving $G$-equivariant linear map $\psi: D\rightarrow E$ such that $ p \psi = \mathrm{Id}_D$.

By \cite[Section 4.1]{BSW}, a possibly non $G$-equivariant lifting $\underline{\psi}$ exists. Then $\psi = \frac{1}{|G|} \sum_{g \in G} g \underline{\psi} g^{-1}$ is the desired $G$-equivariant lifting. 
Since $\phi \underline{\phi}^{-1}$ is a $G$-invariant section of $\End \br{E_0}$, and 
$\phi \underline{\phi}^{-1}|_D = \mathrm{Id}_D$.
By (\ref{EndE0}), we deduce the second assertion.
\end{proof}

Now we extend \cite[Definition 2.4]{B110} and \cite[Definition 5.1.1]{BSW}.
\begin{defn} \label{DefAntGSup}
A first order differential operator $\AEpp$ acting on $\CXE$ is said to be an antiholomorphic $G$-superconnection if it acts as an operator of degree $1$ on $\CXE$ such that 
\begin{enumerate}
    \item If $\alpha \in \OAXC$, $s \in \CXE$,
    \begin{align} \label{AntGSup1}
    \AEpp (\alpha s) = \br{\dbx \alpha} s + \br{-1}^{\deg \alpha} \alpha \AEpp s.
    \end{align}
    \item The following identity holds,
    \begin{align} \label{AntGSup2}
    \AEppsq = 0.
    \end{align}
    \item For any $g \in G$,
    \begin{align} \label{AntGSup3}
    g \AEpp g^{-1} = \AEpp.
    \end{align}
\end{enumerate}
\end{defn}

If $(\underline{E}, \AUEpp)$ is another antiholomorphic $G$-superconnection, then $\Hom (E, \underline{E})$ is also equipped with the induced antiholomorphic $G$-superconnection $\AEUEpp$.

By (\ref{Diagonal}), (\ref{AntGSup1})-(\ref{AntGSup3}), it is easy to find that $\AEpp$ acts on $D $ like a $G$-invariant smooth section $v_0 \in \End(D)$ which is of degree $1$ and such that $v_0^2 = 0$. 


Assume for the moment that we fix a $G$-invariant non-canonical identification as in  (\ref{EIsoE0}) and (\ref{EIsoE02}). 
Let $\AEzeropp$ be the corresponding antiholomorphic $G$-superconnection on $E_0$. 
We can write $\AEzeropp$ in the form
\begin{align} \label{DecOfAnSupCon}
\AEzeropp = v_0 + \NDpp + \sum_{i \geq 2} v_i,
\end{align}
where $\NDpp$ is a degree preserving antiholomorphic $G$-connection on $D$, and for $i \geq 2$, $v_i$ is a $G$-invariant smooth section of $  \lxib \widehat{\otimes} \End^{1-i}(D) $.
Since $\AEppsq = 0$, from (\ref{DecOfAnSupCon}), we get
\begin{align}
&v_0^2 = 0, && \brr{\NDpp, v_0} = 0, && (\NDpp)^2 + \brr{v_0, v_2} = 0.	
\end{align}
In the sequel, we will set 
\begin{align} \label{DefB}
B = v_0 + 	\sum_{i \geq 2} v_i.
\end{align}

\subsection{Two categories} \label{TwoCat}
Let $\br{\underline{E}, \AUEpp}$ be another couple similar to $\br{E, \AEpp}$.
A morphism $\phi: (E,\AEpp) \rightarrow (\underline{E}, A^{\underline{E} \prime\prime})$ is a smooth $G$-invariant section of $\Hom^0 (E, \underline{E})$, which is such that $\AUEpp \phi = \phi \AEpp$.
Then we can define a category $\BXG$ of antiholomorphic $G$-superconnections and above morphisms.


Let $\CXHomEUEG $ be the $\Z$-graded vector space of $G$-invariant smooth sections of $\Hom (E, \underline{E})$ on $X$.

Since $A^{\mathrm{Hom} (E, \underline{E}) \prime\prime,2} = 0$ and $G$-invariant, $\br{\CXHombEUEG, \AEUEpp }$ forms a cochain complex.
Then the morphisms in $\BXG$ are exactly the 0-th cocycles of $\br{\CXHombEUEG, \AEUEpp }$.

Let $\underline{\BB} (X, G)$ be the associated homotopy category. The objects of $\BUXG$ coincide with the objects of $\BXG$, and the morphisms are given by the $0$-th cohomology of $\br{\CXHombEUEG, \AEUEpp }$.
If $\phi_1$, $\phi_2: (E,\AEpp) \rightarrow (\underline{E}, A^{\underline{E} \prime\prime})$ are two $G$-equivariant morphisms, then $\phi_1 = \phi_2$ in $\BUXG$ if and only if there exists $\psi \in C^\infty \br{X, \mathrm{Hom}^{-1} \br{E, \underline{E}}}^G$ such that 
\begin{align} \label{HomoCate}
\phi_1-\phi_2 = \AUEpp \psi + \psi \AEpp.
\end{align}

\subsection{Equivariant superconnections, morphisms, and cones} \label{BlockCone}

We use the notations in Subsection \ref{TwoCat} and follow the conventions of Subsection \ref{Complex}.
We form the cone $C = 	\cone \br{E, \underline{E}}$
so that
\begin{align}
C^\bullet = E^{\bullet +1} \oplus \underline{E}^\bullet.	
\end{align}
As in (\ref{Cone}), put
\begin{equation} 
A_\phi^{C \prime \prime} = 
\left[
\begin{array}{c c}
\AEpp  &  0 \\
\phi (-1)^{\mathrm{deg}}  &  \AUEpp \\
\end{array}
\right]	.
\end{equation}
Then $\br{C, A^{C\prime\prime}_\phi}$ is an object in $\BXG$.

Classically, $\BUXG$ is a triangulated category.

 

\subsection{Pullbacks and Tensor products} \label{BlockPullback}

\subsubsection{Pullbacks of equivariant antiholomorphic superconnections}
Let $Y$ be another compact complex manifold.
Assume the finite group $G$ acts also holomorphically on $Y$.
Let $f: X \rightarrow Y$ be a $G$-equivariant holomorphic map.

Assume $\br{F, \AFpp}$ is an antiholomorphic $G$-superconnection on $Y$ with diagonal bundle $D_F$. Set $\F = \br{\oy^\infty\br{F}, \AFpp}$

In \cite[Section 5.7]{BSW}, the authors define the pullback antiholomorphic superconnection $\br{f_b^* F, A^{f_b^* F \prime\prime}}$ on $X$ with associated diagonal bundle $D_E = f^* D_F$.

By the construction itself, $\br{f_b^* F, A^{f_b^* F \prime\prime}}$ is $G$-equivariant.
We will also use the notation 
\begin{align}
f_b^* \F = \br{\ox^\infty\br{f_b^*F}, A^{f_b^* F \prime\prime}}.	
\end{align}

\subsubsection{Tensor products of equivariant antiholomorphic superconnections} \label{BlockTensor} \label{TensorSuper}
Let $ \br{E,\AEpp}$,  $\br{\underline{E}, \AUEpp}$ be two antiholomorphic $G$-superconnections on $X$, with associated diagonal vector bundles $D_E$, $D_{\underline{E}}$. We can define a tensor product of antiholomorphic $G$-superconnections
$ \br{E \widehat{\otimes}_b \underline{E}, A^{E \widehat{\otimes}_b \underline{E} \prime \prime}}$
as in \cite[(5.8.2)]{BSW}.
The corresponding diagonal vector bundle is $D_E \widehat{\otimes} D_{\underline{E}}$.
As before, $ \br{E \widehat{\otimes}_b \underline{E}, A^{E \widehat{\otimes}_b \underline{E} \prime \prime}}$ is $G$-equivariant.
If we denote $\E = \br{\ox^\infty\br{E}, \AEpp}$ and $\underline{\E} = \br{\ox^\infty \br{\underline{E}}, \AUEpp}$, we will use the notation 
\begin{align}
\E \widehat{\otimes}_b \underline{\E} = \br{\ox^\infty \br{E \widehat{\otimes}_b \underline{E}}, A^{E \widehat{\otimes}_b \underline{E} \prime \prime}}.	
\end{align}

\section{An equivalence of categories} \label{SectionEquivalence}
The purpose of this section is to establish an equivalence of the homotopy category $\BUXG$ and $\DbCohXG$.

In Section \ref{Section6-1}, we construct a natural functor $\underline{F}_X: \BUXG \to \DbCohXG $ and state the main theorem of this section that $\underline{F}_X$ is an equivalence of triangulated categories.

In Section \ref{EssSur}, we prove that the functor $\underline{F}_X$ is essentially surjective.

In Section \ref{SecFullFaith}, we prove that the functor $\underline{F}_X$ is fully faithful.

In Section \ref{Section6-4}, we show that the functor $\underline{F}_X$ is compatible with pullbacks, tensor products and direct images.

\subsection{The main theorem of this section} \label{Section6-1}
Let $X$ be a compact complex manifold. 
Let $G$ be a finite group.
Assume $G$ acts holomorphically on $X$.

Assume $\br{E, \AEpp}$ is an antiholomorphic $G$-superconnection on $X$.
Let $\E$ be the sheaf of $\ox$-complex $\br{\ox^\infty(E), \AEpp}$.

Since $\AEpp$ is $G$-invariant, $\E$ is $G$-equivariant.
By \cite[Lemma 4.5]{B110} and \cite[Theorem 5.3.4]{BSW}, $\E$ has coherent cohomology.
Therefore $\E$ is an object in $\DbCohXG$.

Let $(\underline{E}, \AUEpp)$ be another antiholomorphic $G$-superconnection on $X$, and let $\underline{\E}$ be the sheaf of $\ox$-complex $\br{\ox^\infty(\underline{E}), \AUEpp}$.
A morphism $\phi: (E, \AEpp) \rightarrow (\underline{E}, \AUEpp)$ in $\BXG$ induces a morphism $\E \rightarrow \underline{\E}$ in $\DbCohXG$.

So we actually construct a functor $F_X: \BXG \rightarrow \DbCohXG$.

If $\phi_1, \phi_2: \br{E,\AEpp} \rightarrow \br{\underline{E}, \AUEpp}$ are homotopic, by (\ref{HomoCate}), they induce homotopic morphisms in $\CbCohXG$.
In particular, the images in $\DbCohXG$ are isomorphic.
 
\begin{thm} \label{ThmEquivalenceCategory}
The functor $\underline{F}_X: \BUXG \rightarrow \DbCohXG$ is an equivalence of triangulated categories.	
\end{thm}
\begin{proof}
We will prove that $\underline{F}_X$ is essentially surjective in Subsection \ref{SecEssSur} and is fully faithful in Subsection \ref{SecFullFaith}.
By construction $\underline{F}_X$ is exact so that $\underline{F}_X$ is an equivalence of triangulated categories. 
\end{proof}

\subsection{Essential surjectivity} \label{SecEssSur}

\begin{thm} \label{EssSur}
The functor $F_X$ is essentially surjective, \ie if $\F$ is an object in $ \DbCohXG$, there is an object $\br{E, \AEpp} $ in $ \BXG$ and an isomorphism $\E \simeq \F$ in $\DbCohXG$.	
\end{thm}
\begin{proof}
The proof of our theorem is divided into the next two steps.	
\end{proof}

Let $\F^\infty$ be the $G$-equivariant $\ox^\infty$-complex defined by
\begin{align}
	\F^\infty = \ox^\infty \otimes_{\ox} \F.
\end{align}
The differentials $d^{\F^\infty}$ are naturally induced by the differentials of $\F$.

The following result is established by Illusie in \cite[Proposition II.2.3.2]{Col71} when $G$ is trivial.
Here we extend the result to the situation when $G$ is non-trivial following \cite[Proposition 6.3.2]{BSW}.
\begin{prop} \label{Illusie}
There exist a bounded $G$-equivariant complex of finite dimensional smooth vector bundles $\br{Q, d^Q}$ and a quasi-isomprhism of $G$-equivariant $\ox^\infty$-complexes $\phi_0: \ox^\infty Q \rightarrow \F^\infty$.
\end{prop}

\begin{proof}
Without loss of generality, we may assume $\F^i = 0$ for $i <0$.
Set
\begin{equation}
k = \mathrm{sup} \brrr{i : H^i (\F) \neq 0}.
\end{equation}
Then $k \in \N$.
We will prove the proposition by induction on $k$, while also proving that $Q$ can be chosen such that $Q^i = 0$ for $i > k$.

Assume $k = 0$. 
We will show in this case additionally that
\begin{equation} \label{AdditionalCondition}
\text{all the} \; Q^i \; \text{is} \; G\text{-trivial except the first non-zero one}. 	
\end{equation}

Since $H^0 \br{\F}$ is a $G$-equivariant coherent sheaf, 
by \thref{LocalResolution}, for any $x \in X$, there exist a $G$-invariant open neighborhood $U$ of $x$, and a bounded $G$-equivariant complex $R_U$ of  holomorphic vector bundles in nonpositive degrees and a $G$-equivariant $\ou$-morphism $r_U: \OO_U R_U^0 \rightarrow H^0 (\F)_{\mid U}$ such that the $G$-equivariant $\ou$-complex 
\begin{equation} \label{IllusieHolomorphicResolution}
\begin{tikzcd}
	0 \arrow[r] & \OO_U R_U^{-\ell_U} \arrow[r] & \cdots \arrow[r] & \OO_U R_U^{0} \arrow[r,  "r_U"] & H^0 (\F)_{\mid_U} \arrow[r] & 0
\end{tikzcd}
\end{equation}
is exact.
Moreover, for $0 \leq i \leq \ell_U-1$, $R_U^{-i}$ are $G$-trivial on $U$.

Since $\OO_U^\infty$ is flat over $\OO_U$ \cite[Corollary VI.1.12]{MR}, by (\ref{IllusieHolomorphicResolution}), we have the corresponding exact sequence in $\MIUG$,
\begin{equation} \label{SmoothLocalResolution}
\begin{tikzcd}
	0 \arrow[r] &\OO_U^\infty R_U^{-\ell_U} \arrow[r] & \cdots \longrightarrow \OO_U^\infty R_U^{0} \arrow[r, "r_U"] & H^0 (\F^\infty)_{\mid_U} \arrow[r] & 0.
	\end{tikzcd}
\end{equation}
Since $X$ is compact, actually we can choose a finite covering $\mathcal{U}$ of $X$ by such $U$.

We can assume that for any $U \in \mathcal{U}$, $\ell_U = \ell$, 
because we can always extend shorter complexes, and the lengths of all the complexes are no longer than $n$.
Then we prove the case $k =0$ with the additional assumption (\ref{AdditionalCondition}) by induction on $\ell$ \footnote{In this induction argument, we only assume $\F^\infty$ has the properties (\ref{AdditionalCondition}) and (\ref{SmoothLocalResolution}) and do not require that $\F^\infty$ comes from an $\ox$-complex\label{fnlabel2}}

If $\ell = 0$, it means that $H^0 (\F^\infty)$ is a locally free $\OO_X^\infty$-module.
By the discussion after \thref{InducedRep}, the proof of our proposition is complete.

Assume now that $\ell \geq 1$ and that the proposition holds with the additional assumption (\ref{AdditionalCondition}).

The local resolution in (\ref{SmoothLocalResolution}) gives a quasi-isomorphism $\OO_U^\infty R_U \rightarrow \F^\infty_{\mid_U}$ in $\MIUG$. If we take $\CC_U = \mathrm{cone} \br{ \OO_U^\infty R_U, \F^\infty_{\mid_U}}$, then $\CC_U$ is exact. Actually $\CC_U$ is just the complex
\begin{equation} \label{ComplexRF}
\begin{tikzcd}
0 \arrow[r] & \OO_U^\infty R_U^{-\ell} \arrow[r] & \cdots \arrow[r] & \OO_U^\infty R_U^{0} \arrow[r, "r_U"] & \F^{0,\infty}_{\mid U} \arrow[r] & \F^{1,\infty}_{\mid U} \arrow[r] & \cdots 	
\end{tikzcd}
\end{equation}

Since $\ell \geq 1$, $R_U^0$ is $G$-trivial on $U$. 
We can extend it to a trivial $G$-equivariant vector bundle on $X$ which we also denote it as $R_U^0$. Set
\begin{equation} \label{ExtendR}
R = \bigoplus_{U \in \mathcal{U}} R_U^0	.
\end{equation}
Then $R$ is still $G$-trivial on $X$, and we regard it as a complex in degree $0$ with an obvious morphism of $G$-equivariant complexes of smooth vector bundles $R_{\mid U} \rightarrow R_U$.

Let $\br{\varphi_U }_{U \in \mathcal{U}}$ be a smooth partition of unity subordinated to $\mathcal{U}$. 
Set $\widetilde{\varphi}_U (x) = \frac{1}{|G|} \sum_{g \in G} \varphi_U \br{g \cdot x}$.  
Then $\widetilde{\varphi}_U$ is $G$-invariant and $\br{\widetilde{\varphi}_U }_{U \in \mathcal{U}}$ forms another partition of unity subordinated to $\mathcal{U}$.

Put
\begin{equation} \label{IllusieGlobalMap}
r = \sum_{U\in \mathcal{U}} \widetilde{\varphi}_U r_U.	
\end{equation}
By (\ref{SmoothLocalResolution}), (\ref{ExtendR}), and (\ref{IllusieGlobalMap}), $r: \ox^\infty R \rightarrow H^0 (\F^{\infty}) $ is surjective.
As in (\ref{ComplexRF}), we can consider $r$ as a morphism of $G$-equivariant $\ox^\infty$-complexes $\ox^\infty R \rightarrow \F^\infty$.

We claim that the morphism $r_{\mid_U}$ lifts to a $G$-equivariant morphism of $\ou^\infty$-complexes $\psi_U: \ou^\infty R \rightarrow \OO_U^\infty R_U$, such that the following diagram 
\begin{equation} \label{LocalLift}
\begin{tikzcd}
                                                                          & \mathcal{O}_U^\infty R_U \arrow[d, "r_U"] \\
\mathcal{O}_U^\infty R \arrow[r, "r_{\mid_U}"] \arrow[ru, "\psi_U", dotted] & \mathscr{F}^\infty_{\mid_U}                
\end{tikzcd}	
\end{equation}
commutes.
Indeed, since the vertical arrow $r_U$ in (\ref{LocalLift}) is a quasi-isomorphism, according to \cite[(6.3.9)]{BSW}, a possible non $G$-equivariant lift $\widetilde{\psi}_U$ exists.
Then $\psi_U = \frac{1}{|G|} \sum_{g \in G} g^{-1} \widetilde{\psi}_U g$ is the desired lift.


For simplicity, we will use the notation $\RRR_U$, $\RRR_{\mid_U}$ instead of $\OO_U^\infty R_U$, $\OO_U^\infty R$.

Using ($\ref{LocalLift}$), the fact $r_U$ is a quasi-isomorphism and the cone construction, we have the following quasi-isomorphism
\begin{equation} \label{QuasiComplex}
\begin{tikzcd}
0 \arrow[r] & \mathscr{R}_U^{-\ell} \arrow[r] \arrow[d] & \cdots \arrow[r] \arrow[d] & \mathscr{R}_U^{-1}\oplus \mathscr{R}_{\mid_U} \arrow[r, "\gamma"] \arrow[d] & \mathscr{R}^0_U \arrow[r] \arrow[d]       & 0 \arrow[r] \arrow[d]                     & \cdots \\
0 \arrow[r] & 0 \arrow[r]                              & 0 \arrow[r]                & \mathscr{R}_{\mid_U} \arrow[r, "r_{\mid U}"]                                            & {\mathscr{F}^{0,\infty}_{\mid_U}} \arrow[r] & {\mathscr{F}^{1,\infty}_{\mid_U}} \arrow[r] & \cdots 
\end{tikzcd}	
\end{equation}
Indeed, the first line is the cone of $\psi_U$ and the second line is the restriction of the $\mathrm{cone} (\RRR, \F^\infty)$ which is globally defined on $X$. 

Since $k=0$ and since $r: \RRR \rightarrow H^0(\F^\infty)$ is surjective, the cohomology of $\mathrm{cone}(\RRR, \F^{\infty})$ concentrates in the $-1$ degree and we denote it as $H^{-1}$.
Therefore, the $0$-th cohomology of the first line in (\ref{QuasiComplex}) vanishes.
In particular, $\gamma$ is surjective.
By the discussion following (\ref{GActionOnFiber}), $\gamma$ induces a surjection  on each fiber, so that $\mathrm{ker} (\gamma)$ defines a vector bundle.

Therefore, (\ref{QuasiComplex}) can be modified to
\begin{equation} \label{IllusieShorterComplex}
\begin{tikzcd}
0 \arrow[r] & \mathscr{R}^{-\ell}_U \arrow[r] \arrow[d] & \cdots \arrow[r] \arrow[d] & \mathscr{R}^{-2}_U \arrow[r] \arrow[d] & \mathrm{ker}(\gamma) \arrow[r] \arrow[d] & 0 \arrow[r] \arrow[d]                       & \cdots \\
0 \arrow[r] & 0 \arrow[r]                               & \cdots \arrow[r]           & 0 \arrow[r]                            & \mathscr{R}_{\mid U} \arrow[r]           & {\mathscr{F}^{0,\infty}_{\mid U}} \arrow[r] & \cdots
\end{tikzcd}	
\end{equation}
which is still a quasi-isomorphism.

If $\ell = 1$, the first line of (\ref{IllusieShorterComplex}) concentrates at one degree $\mathrm{ker}(\gamma)$.
By (\ref{IllusieShorterComplex}),  $\ker (\gamma) \rightarrow R$ is exactly the complex we want.
Note that $R$ is $G$-trivial, so it satisfies the additional assumption (\ref{AdditionalCondition}).

Now we assume $\ell \geq 2$.
Note that $\ker \br{\gamma}$ is not always $G$-trivial so we can not apply the induction argument directly.

To overcome this issue, we consider a variant of (\ref{IllusieShorterComplex}) by adding $\RRR^0_U$ at degree $-2$ and $-1$ in an obvious way, so that we have the following quasi-isomorphism
\begin{equation} \label{IllusieShorterComplex2}
\begin{tikzcd}
0 \arrow[r] & \mathscr{R}^{-\ell}_U \arrow[r] \arrow[d] & \cdots \arrow[r] \arrow[d] & \mathscr{R}^{-2}_U \oplus \RRR^0_U \arrow[r] \arrow[d] & \mathrm{ker}(\gamma) \oplus \RRR^0_U \arrow[r] \arrow[d] & 0 \arrow[r] \arrow[d]                       & \cdots \\
0 \arrow[r] & 0 \arrow[r]                               & \cdots \arrow[r]           & 0 \arrow[r]                            & \mathscr{R}_{\mid U} \arrow[r]           & {\mathscr{F}^{0,\infty}_{\mid U}} \arrow[r] & \cdots
\end{tikzcd}	
\end{equation}

Since $\gamma$ is surjective on each fiber, we have $\ker(\gamma) \oplus \RRR^0_U \simeq \RRR^{-1}_U \oplus \RRR_{\mid_U}$ which is now $G$-trivial. Moreover, the first line of complex (\ref{IllusieShorterComplex2}) is of length $\ell - 1$. Therefore the second line of (\ref{IllusieShorterComplex2}) satisfies the induction hypothesis.\footnote{The second line of (\ref{IllusieShorterComplex2}) does not come from a $\ox$-complex, but we can still use the induction assumption, see Footnote \ref{fnlabel2}.}


Therefore, there exists a bounded $G$-equivariant complex of finite dimensional smooth vector bundles $\br{Q^\prime, d^{Q^\prime}}$ with $Q^{\prime i } = 0$ if $i \geq 0$, such that if $\LL^\prime = \OO^\infty_X Q^\prime$, then we have a $G$-equivariant quasi-isomorphism of $\OO_X^\infty$-complexes
\begin{equation} \label{InduQuasi}
\LL^\prime \rightarrow \mathrm{cone} \br{\RRR, \F^\infty}.
\end{equation}
By \cite[Proposition 1.4.4]{KaS90}, the quasi-isomorphism (\ref{InduQuasi}) gives a quasi-isomorphism
\begin{align} \label{InduQuasi2}
	\mathrm{cone} \br{\LL^{\prime \bullet-1}, \RRR^{\bullet}} \rightarrow \F^\infty, 
\end{align}
which is $G$-equivariant by construction.

So we can take $Q^\bullet = \mathrm{cone} \br{Q^{\prime \bullet -1},R^\bullet}$. Then $Q^i = 0$, $i > 0$, and $Q^\bullet$ satisfies the additional assumption (\ref{AdditionalCondition}). Moreover, $\OO_X^\infty Q \rightarrow \F^\infty$ is a $G$-equivariant quasi-isomorphism, so the proof of our theorem when $k=0$ is done.

The proof for $k \geq 1$ is essentially the same as in \cite[Proporsition 6.3.2]{BSW} by induction on $k$.
We recall briefly the steps.

We can truncate the complex $\F^\infty$ into 
\begin{equation}
\tau_{\leq k -1} \F^\infty: 0 \longrightarrow \F^{0,\infty} \cdots \longrightarrow \F^{k-2,\infty} \longrightarrow \operatorname{Ker} d^{\F^\infty}_{\mid_{\F^{k-1,\infty}}} \longrightarrow 0.	
\end{equation}

By cone construction, we have a quasi isomorphism
\begin{align}
	\cone \br{ \cone^{\bullet -1} \br{\tau_{\leq k -1} \F^\infty, \F^\infty}, \tau_{\leq k -1} \F^{\infty, \bullet}} \rightarrow \F^{\infty, \bullet}.
\end{align}

Since $\tau_{\leq k -1} \F^\infty$ satisfies the induction hypothesis for $k-1$ and $\cone \br{\tau_{\leq k -1} \F^\infty, \F^\infty}$ satisfies the induction hypothesis for $k = 0$.
By induction assumption, the proof of our theorem is complete.
\end{proof}

Put
\begin{align} \label{FOI}
\overline{\F}^\infty = \ox^\infty \br{\lxb} \WOX	 \F,
\end{align}
and equip it with the differential $d^{\overline{\F}^\infty}$ given by
\begin{align} \label{dFOI}
d^{\FOI} = \dbx +d^{\F^\infty}.
\end{align}
Then $\FOI$ is a $G$-equivariant $\ox$-complex.

By Poincar\'{e} Lemma \cite[Lemma I.3.29]{demailly1997complex} and by a theorem of Malgrange \cite[Corollary VI.1.12]{MR}, we have a quasi-isomorphism of $G$-equivariant $\ox$-complexes,
\begin{align}
\F \rightarrow \FOI.	
\end{align}


The following result is established in \cite[Theorem 6.3.6]{BSW} when $G$ is trivial.
\begin{thm} \thlabel{EssSurB}
Given an object $\F$ in $\DbCohXG$, there exists an object $\br{E, \AEpp}$ in $\BXG$ and a $G$-equivariant morphism of $\ox^\infty \br{\lxb}$-modules $\phi: \E \rightarrow \overline{\F}^\infty$, which is a quasi-isomorphism of $G$-equivariant $\ox$-complexes, and induces a quasi-isomorphism of $G$-equivariant $\ox^\infty$-complexes $\br{D, v_0} \rightarrow \F^\infty$. In particular, $\E$ and $\F$ are isomorphic in $\DbCohXG$.
\end{thm}
\begin{proof}

We take $(Q,d^Q)$ as in Proposition \ref{Illusie}.
Set $E = \lxb \widehat{\otimes} Q$ and $\E = \ox^\infty (E)$.

Now we will construct by induction an antiholomorphic $G$-superconnection  \footnote{Here we will write $v_1$ an antiholomorphic connection $\nabla^{Q\prime\prime}$ on $Q$.}  $\AEpp = \sum_{i\geq 0} v_i$ on $E$ and a morphism of $G$-equivariant $\ox^\infty \br{\lxb}$-modules $\phi = \sum_{i\geq 0} \phi_i$ from $\E$ to $\overline{\F}^\infty$ such that $\phi \AEpp  =  d^{\overline{\F}^\infty} \phi$.

If $k=0$, we take $v_0 = d^Q$ and $\phi_0$ as in Proposition \ref{Illusie}.


Assume $k \geq 1$, and assume that we have constructed $v_i$, $\phi_i$, $0 \leq i \leq k-1$, such that if $A_{\leq k-1} = \sum_{i \leq k-1} v_i$ and if $\phi_{\leq k-1} = \sum_{i \leq k-1} \phi_i$, then up to terms in $\Lambda^{ \geq k} (\overline{T^*X})$,
\begin{align} 
&(\AEpp_{\leq k-1 })^2 = 0,  && \phi_{\leq k-1} \AEpp_{\leq k-1} - d^{\overline{\F}^\infty}_{\leq k-1} \phi_{\leq k-1}  = 0. 	
\end{align}
We will construct $v_k$, $\phi_k$ such that up to terms in $\Lambda^{ \geq k+1} (\overline{T^*X})$,
\begin{align} \label{ThmSurEqu3}
& (\AEpp_{\leq k-1} + v_k)^2 = 0, && (\phi_{\leq k-1}+ \phi_k) (\AEpp_{\leq k-1} + v_k	) - d^{\overline{\F}^\infty} (\phi_{\leq k-1}+ \phi_k)=0.
\end{align}

In \cite[Theorem 6.3.6]{BSW}, the authors construct possibly non $G$-invariant $\underline{v}_{k}$ and $\underline{\phi}_k$ such that up to terms in $\Lambda^{ \geq k+1} (\overline{T^*X})$, an obvious analogue of  (\ref{ThmSurEqu3}) hold.

Set 
\begin{align}
&	v_k = \frac{1}{|G|} \sum_{g\in G} g \underline{v}_k g^{-1}, && \phi_k = \frac{1}{|G|} \sum_{g\in G} g \underline{\phi}_k  g^{-1}.
\end{align}
Then $v_k$, $\phi_k$ are $G$-invariant and we will show that (\ref{ThmSurEqu3}) holds.

Since $k \geq 1$, up to terms in $\Lambda^{ \geq k+1} (\overline{T^*X})$,
\begin{align} 
(\AEpp_{\leq k-1} +  v_k  )^2 = (\AEpp_{\leq k-1})^2 + 	\brr{\AEpp_{\leq k-1},  v_k  }, \label{EquModkPlus1}\\
(\AEpp_{\leq k-1} +   \underline{v}_k  )^2 = (\AEpp_{\leq k-1})^2 + 	\brr{\AEpp_{\leq k-1},  \underline{v}_k  } \nonumber.
\end{align}
Since $\AEpp_{\leq k-1}$ is $G$-invariant, 
\begin{align} \label{Equk3}
(\AEpp_{\leq k-1})^2 + 	\brr{\AEpp_{\leq k-1},  v_k  } = 	 \frac{1}{|G|} \sum_{g\in G} g \br{ \br{\AEpp_{\leq k-1}}^2 + \brr{\AEpp_{\leq k-1}, \underline{v}_k}}  g^{-1}.
\end{align}
By (\ref{EquModkPlus1})-(\ref{Equk3}), we find that up to terms in $\Lambda^{ \geq k+1} (\overline{T^*X})$,
\begin{align}
	(\AEpp_{\leq k-1} +  v_k  )^2 = 0.
\end{align}
For a similar reason, up to terms in $\Lambda^{ \geq k+1} (\overline{T^*X})$,
\begin{align}
	(\phi_{\leq k-1}+ \phi_k) (\AEpp_{\leq k-1} + v_k	) - d^{\overline{\F}^\infty} (\phi_{\leq k-1}+ \phi_k)=0.
\end{align}

Put $\AEpp = A_{\leq n}$ and $\phi = \phi_{\leq n}$, then $\AEpp$ is an antiholomorphic $G$-superconnection and $\phi \AEpp = d^{\overline{\F}^\infty} \phi$.

In summary, we have constructed an antiholomorphic $G$-superconnection $\br{E, \AEpp}$  and a morphism of $G$-equivariant $\ox$-complexes $\phi: \E \rightarrow \overline{\F}^{\infty}$.
It induces a morphism $\phi_0: \br{ \ox^\infty\br{D}, v_0} \rightarrow \br{\F^\infty, d^{\F^\infty}}$. Moreover, the induced $\phi_0$ is exactly what we constructed in Proposition \ref{Illusie}, which is a quasi-isomorphism by construction.

As in (\ref{FiltrationE}) and (\ref{FOI}), $\E$, $\overline{\F}^\infty$ has a canonical filtration.  
Let $(\E_r,d_r)|_{r \geq 0}$, $(\overline{\F}^\infty_r, \underline{d}_r) |_{r \geq 0}$ be the associated spectral sequence. 
We also denote
\begin{align}
\phi_r: (\E_r,d_r) \rightarrow	(\overline{\F}^\infty_r, \underline{d}_r)\end{align}
the morphisms of the $\ox$-complexes induced by $\phi$.
The induced morphism at $r=0$ is equal to $\phi_0$ which is a quasi-isomorphism by construction.
The proof our theorem is complete.


\end{proof}

\subsection{Fully faithful} \label{SecFullFaith}

\begin{thm} \label{Fullyfaithful}
The functor $\underline{F}_X: \BUXG \rightarrow \DbCohXG$ is fully faithful.
\end{thm}
\begin{proof}

Assume that $\br{E,\AEpp}$, $\br{\underline{E}, \AUEpp}$ are objects in $\BUXG$, and that $\E$, $\underline{\E}$ the corresponding sheaves of $\ox$-modules.
We need to show the corresponding map
\begin{align} \label{FulFaiEqu1}
 \Hom_{\BUXG} \br{\br{E,\AEpp}, \br{\underline{E}, \AUEpp}} \rightarrow \Hom_{\DbCohXG} \br{\E, \underline{\E}}	
\end{align}
is bijective.

Firstly, we prove the injectivity.
Let $\phi_1$, $\phi_2 \in  \Hom_{\BUXG} \br{\br{E,\AEpp}, \br{\underline{E}, \AUEpp}}$ such that their images in $\Hom_{\DbCohXG} \br{\E, \underline{\E}}$ coincide. 
Then  $\phi_1$, $\phi_2$ as morphisms in $\underline{B}(X)$ and their images in $\Hom_{\DbCohX} \br{\E, \underline{\E}}$ also coincide.

In \cite[Theorem 6.5.1]{BSW}, the authors proved that the functor $\underline{F}_X: \underline{\mathrm{B}} (X) \rightarrow \DbCohX$ is fully faithful. 
Therefore $\phi_1=\phi_2$ as morphisms in $\underline{\mathrm{B}}(X)$, so there exists $\underline{\psi} \in C^\infty \br{X, \mathrm{Hom}^{-1} \br{E, \underline{E}}}$ such that 
\begin{align}  \label{HomEqu2}
\phi_1-\phi_2 = \AUEpp \underline{\psi} + \underline{\psi} \AEpp.
\end{align}

Put
\begin{align} \label{Average}
\psi = 	\frac{1}{|G| } \sum_{g \in G} g \underline{\psi} g^{-1} .
\end{align}
Then $\psi \in C^\infty \br{X, \mathrm{Hom}^{-1} \br{E, \underline{E}}}^G$.
Since $\phi_1$, $\phi_2$, $\AEpp$ and $\AUEpp$ are $G$-invariant, the same equation as in (\ref{HomEqu2}) holds where $\underline{\psi}$ is replaced by $\psi$.
Therefore $\phi_1 = \phi_2$ in $\BUXG$.

Secondly we prove the surjectivity.
A morphism in $\Hom_{\DbCohXG} \br{\E, \underline{\E}}$ is represented by
\begin{equation} \label{HomRoofA}
\begin{tikzcd}
\mathscr{E} & \mathscr{F} \arrow[l, "\mathrm{q.i.}"'] \arrow[r] & \underline{\mathscr{E}},
\end{tikzcd}
\end{equation}
where $\F$ is an object in $\DbCohXG$ and `q.i.' stands for `quasi-isomorphism'.

By Theorem \ref{EssSurB} and by proceeding as in \cite[(6.5.2)-(6.5.4)]{BSW},
there exists a $G$-equivariant antiholomorphic superconnection $\br{\underline{\underline{E}}, A^{\underline{\underline{E}} \prime\prime}}$, with $\underline{\underline{\E}}$ the associated sheaf and a $G$-equivariant quasi-isomorphism $\phi^\prime: \underline{\underline{\E}} \rightarrow \E$ such that (\ref{HomRoofA}) is equal to
\begin{equation} \label{HomRoofB}
	\begin{tikzcd}
\mathscr{E} & \underline{\underline{\E}} \arrow[l, "\mathrm{\phi^\prime}"'] \arrow[r] & \underline{\mathscr{E}}
\end{tikzcd}
\end{equation}
in $\DbCohXG$.
Moreover, the  morphisms $\phi^\prime$ is a morphism in $\BUXG$.

By \cite[Proposition 6.4.1]{BSW}, $\phi^\prime$ has a homotopic inverse  $\underline{\phi}$ in $\underline{\BB} (X)$. That is, there exist $\underline{\psi}_1 \in C^\infty \br{X, \mathrm{End}^{-1} \br{\underline{ \underline{E}}}}$ and $\underline{\psi}_2 \in C^\infty \br{X, \mathrm{End}^{-1} \br{E}}$ such that
\begin{align} \label{Homotopy1}
	& \underline{\phi} \phi^{\prime} - \mathrm{Id}_{\underline{\underline{E}}} = \brr{A^{\underline{\underline{E}} \prime\prime}, \underline{\psi}_1}, && \phi^{\prime} \underline{\phi}  - \mathrm{Id}_{E} = \brr{A^{E\prime\prime}, \underline{\psi}_2}.
\end{align}

We construct $\phi$, $\psi_1$ and $\psi_2$ through $\underline{\phi}$, $\underline{\psi}_1$ and $\underline{\psi}_2$ as in (\ref{Average}).

Since $\phi$, $\AEpp$, and $A^{\underline{\underline{E}} \prime\prime}$ are $G$-invariant, the same equation as in (\ref{Homotopy1}) holds where $\underline{\phi}$, $\underline{\psi}_1$ and $\underline{\psi}_2$ are replaced by $\phi$, $\psi_1$ and $\psi_2$ respectively.
Therefore, $\phi$ is the homotopic inverse of $\phi^{\prime}$ in $\BUXG$.
Using this homotopic inverse, we get the desired morphism $\br{E,\AEpp} \rightarrow \br{\underline{E},\AUEpp}$. The proof of surjectivity is complete, therefore our theorem is proved.
\end{proof}




\subsection{Compatibility with pullbacks, tensor products and direct images} \label{Section6-4}
We extend the results in \cite[Propositions 6.6.1, 6.7.1, 6.8.2]{BSW} to the case when $G$ is non-trivial.
\subsubsection{Compatibility with pullbacks}
We use the same assumptions as in Subsection \ref{BlockPullback} and we use the corresponding notation. In particular, we assume $f: X \rightarrow Y$ is a $G$-equivariant holomorphic map between compact complex $G$-manifolds.

\begin{prop} \label{PropCompatiblePullback}
The following diagram commutes up to canonical $G$-isomorphism,
\begin{equation}
\begin{tikzcd}
{\underline{\rm{B}} (Y,G)} \arrow[r, "\underline{F}_Y"] \arrow[d, "f_b^*"] & {\rm{D}_{\rm{coh}}^{\rm{b}}(Y,G)} \arrow[d, "Lf^*"] \\
{\underline{\rm{B}} (X,G)} \arrow[r, "\underline{F}_X"]                    & {\rm{D}_{\rm{coh}}^{\rm{b}}(X,G)} .                 
\end{tikzcd}
\end{equation}
	
\end{prop}
\begin{proof}
The proof is the same as in \cite[Proposition 6.6.1]{BSW}.
\end{proof}

\subsubsection{Compatibility with tensor products} \label{SubsectionComTen}
We use the notation of Subsections \ref{DerivedTensor} and \ref{BlockTensor}.


\begin{prop} \label{PropCompatibleTensor}
The following diagram commutes up to canonical $G$-isomorphism,
\begin{equation}
\begin{tikzcd}
{\underline{\mathrm{B}}(X, G) \times \underline{B}(X, G)} \arrow[rr, "\underline{F}_X \times \underline{F}_X"] \arrow[d, "\widehat{\otimes}_b"] &  & {\mathrm{D}^b_{\mathrm{coh}} (X,G)\times \mathrm{D}^b_{\mathrm{coh}} (X,G)} \arrow[d, "\widehat{\otimes}^L_{\mathcal{O}_X}"] \\
{\underline{\mathrm{B}}(X, G)} \arrow[rr, "\underline{F}_X"]                                                                                    &  & {\mathrm{D}^b_{\mathrm{coh}} (X,G)}                                                                                
\end{tikzcd}
\end{equation}
\end{prop}
\begin{proof}
The proof is essentially the same as in \cite[Proposition 6.7.1]{BSW}.	
\end{proof}

\subsubsection{Compatibility with direct images}
We use the same notations as in Subsection \ref{DerivedDirect}.
Let $f: X \rightarrow Y$ be a $G$-equivariant holomorphic map between compact complex $G$-manifolds.
Let $\br{E, \AEpp}$ be an object in $\BUXG$, so that the corresponding sheaves of $\ox$-modules $\E$ is an object in $\DbCohXG$.
Since $\E$ is a bounded complex of soft $\ox^\infty$-modules, by (\ref{SoftDirect}), we have
\begin{align}
R f_* \E = f_* \E .	
\end{align}
Also $f^*$ maps $\oy^\infty (\lyb)$ to $f_* \ox^\infty (\lxb)$.
The differential of $f_* \E$ verifies Leibniz's rule with respect to multiplication by $\oy^\infty (\lyb)$, so that $f_* \E$ is a $G$-equivariant complex of $\oy^\infty (\lyb)$-modules.
\begin{prop} \label{PropCompatibleDirectImage}
There exists an object $\br{\underline{E}, A^{\underline{E} \prime\prime}}$ in $\rm{B} (Y, G)$ and a $G$-equivariant morphism of $\oy^\infty (\lyb)$-modules $\phi: \underline{\E} \rightarrow f_* \E$, which is also a quasi-isomorphism of $\oy$-complexes.	
\end{prop}
\begin{proof}
 The proof is essentially the same as in \cite[Proposition 6.8.2]{BSW} if we use \thref{EssSurB} instead of 	\cite[Proposition 6.3.6]{BSW}.
\end{proof}

\section{Equivariant generalized metric and equivariant Chern character forms} \label{EquiGeneMetr1}
The purpose of this section is to construct equivariant Chern character forms for antiholomorphic $G$-superconnections. Also, we show that the equivariant Chern character extends to $\DbCohXG$ and that it factors through $K(X, G)$.

In Section \ref{EquiGeneMetr}, we introduce the $G$-invariant generalized metric on $D$. Given an antiholomorphic $G$-superconnection $\br{E, \AEpp}$, a non-canonical identification as in \eqref{EIsoE0}, and a $G$-invariant generalized metric, we construct the adjoint $\AEzerop$ of $\AEzeropp$.

In Section \ref{EquiCherCharForm}, given $g\in G$, we construct the equivariant Chern character forms $\chgBC \br{\AEzeropp, h}$, and we establish their main properties.

In Section \ref{Section7-3}, we describe the behavior of the equivariant Chern character under pullbacks and tensor products. We also evaluate the equivariant Chern character of a cone.

In Section \ref{EquiChernCharacter}, using the results of Section \ref{SectionEquivalence}, we show that the equivariant Chern character defines a map $\chgBC: \chgBC: \DbCohXG \to \HEXgBCC$.

In Section \ref{Section7-5}, we show that $\chgBC$ factors through $K(X, G)$, \ie it induces a map $\chgBC: K(X, G) \rightarrow \HEXgBCC$.

\subsection{Equivariant generalized metric and adjoint} \label{EquiGeneMetr}
We make the same assumption as in Section \ref{EquAntSup} and we use the corresponding notations.
The constructions in this subsection is established in \cite[Section 7.1]{BSW} when $G$ is trivial, and the extension to the equivariant case is quite natural.

We recall that in \cite[Section 7.1]{BSW}, a degree $\mathrm{deg}_{-}$ is defined on $\lxc$ such that if $\alpha \in \Lambda^p \br{T^* X}$, $\beta \in \Lambda^{q} \br{\overline{T^* X}}$, then
\begin{align} \label{Equ6-30}
\mathrm{deg}_{-} \alpha \wedge \beta = q - p.	
\end{align}

We follow \cite[Section 3.5]{B13}, \cite[Section 4.4]{BSW}. If $e \in T^*_{\C} X$, 
set
\begin{align}
\widetilde{e} = -e.	
\end{align}
We still denote by $\widetilde{\phantom{x}}$ the corresponding anti-automorphism of the algebras $\lxc$. If $\alpha \in \lxpc$,
\begin{align} \label{Equ6-1}
\widetilde{\alpha} = (-1)^{p (p+1)/2} \alpha.
\end{align}
Put 
\begin{align}
\alpha^* = \overline{\widetilde{\alpha}}.	
\end{align}

We fix a non-canonical $G$-equivariant identification as in  (\ref{EIsoE0}) and (\ref{EIsoE02}). Then 
\begin{align} \label{Splitting2}
\lx \widehat{\otimes} E \simeq \lxc \widehat{\otimes} D.
\end{align}

If $A \in \Hom (D, \overline{D}^*)$, we denote $A^* \in \Hom (D, \overline{D}^*)$ the conjugate of the transpose of $A$. 
Also $\Hom (D, \overline{D}^*)$ is naturally graded by counting the difference of the degrees in $\overline{D}^*$ and $D$, 
so we can equip $\lxc \widehat{\otimes} \Hom(D, \overline{D}^*)$ with the obvious antilinear involution $*$,
and with the degree induced by $\mathrm{deg}_{-}$ and the degree on $\Hom (D, \overline{D}^*)$.
An element of $\lxc \widehat{\otimes} \Hom (D, \overline{D}^*)$ is said to be self-adjoint if it is invariant under $*$.

If $s \in \Omega (X, D)$, $s^\prime \in \Omega (X, \overline{D}^*)$, we write $s$, $s^\prime$ in the form
\begin{align} \label{Equ61}
&s= \sum \alpha_i r_i, && s^\prime = \sum \beta_j t_j,	
\end{align}
with $\alpha_i$, $\beta_j \in \Omega (X, \C)$ and $r_i \in C^\infty (X, D)$, $t_j \in C^\infty (X, \overline{D}^*)$.
Then we recall \cite[Definition 7.1.2]{BSW},
\begin{defn}
Put 
\begin{align} \label{DefTheta}
\theta (s, s^\prime) = \frac{i^n}{(2 \pi)^n} \sum \int_X \widetilde{\alpha}_i \wedge \overline{\beta}_j \left\langle r_i, \overline{t}_j \right\rangle.	
\end{align}
\end{defn}
It is clear that $\theta$ is independent of the writing of (\ref{Equ61}) and is $G$-invariant.

If $h \in \lxcHomDDbar$, we can write $h$ in the form
\begin{align}
&h = \sum_{i=0}^{2n} h_i,	 && h_i \in \Lambda^i (T_\C^* X) \widehat{\otimes} \Hom (D, \overline{D}^*).
\end{align}

\begin{defn} \label{Def6-2}
We call $h \in \Omega \br{X, \mathrm{Hom} \br{D, \overline{D}^*}}$ is a $G$-invariant generalized metric on $D$ if $h$ is $G$-invariant, of degree $0$, self adjoint, and if $h_0$ is a Hermitian metric on $D$.
Let $\MDG$ be the set of all $G$-invariant generalized metrics on $D$.
We call $h \in \MDG$ is pure if $h=h_0$.
\end{defn}

Let $h \in \MDG$ .
\begin{defn}
If $s$, $s^\prime \in \Omega (X, D)$, put
\begin{align} \label{DefThetah}
\theta_h (s, s^\prime) = \theta (s, h s^\prime).	
\end{align}
\end{defn}
Since $\theta$, $h$ are $G$-invariant, $\theta_h$ is also $G$-invariant.


As in \cite[Section 7.1]{BSW}, we can extend $\AEpp$ to an operator $A^{\Lambda \br{T^* X} \widehat{\otimes} E\prime\prime}$ acting on $C^\infty \br{X, \Lambda \br{T^* X} \widehat{\otimes} E}$ such that if 
$\alpha \in \Omega \br{X, \C}$, $s \in C^\infty \br{X, \Lambda \br{T^* X} \widehat{\otimes} E}$, 
\begin{align} \label{Equ6-2}
	A^{\Lambda \br{T^* X} \widehat{\otimes} E\prime\prime} \br{\alpha s} = \br{\dbx \alpha} s + (-1)^{\mathrm{deg} \alpha} \alpha A^{\Lambda \br{T^* X} \widehat{\otimes} E\prime\prime} s.
\end{align}
In the sequel, we will use the notation $\AEpp$ instead of $A^{\Lambda \br{T^* X} \widehat{\otimes} E\prime\prime}$.

\begin{defn} \label{DefAdjoint}
Let $A^{E_0 \prime}$ denote the formal adjoint of $\AEzeropp$ with respect to $\theta_h$.	
\end{defn}
Then $A^{E_0 \prime}$ verifies (\ref{AntGSup1}) when replacing $\dbx$ by $\partial^X$ and 
\begin{align}
A^{E_0\prime,2}	= 0.
\end{align}

Let $\ndp$ be the adjoint of $\ndpp$ with respect to $\theta_h$.
If $i=0$, or $i \geq 2$, let $v_i^*$ be the adjoint of $v_i$ with respect to $\theta_h$. By (\ref{DecOfAnSupCon}), we obtain
\begin{align} \label{DecAEzerop}
\AEzerop = v_0^* +\ndp + \sum_{i \geq 2} v_i^*.	
\end{align}
Recall that $B$ was defined in (\ref{DefB}). Its adjoint $B^*$ is given by
\begin{align} \label{DefBstar}
B^* = v_0^* + \sum_{i \geq 2} v_i^*. 	
\end{align}
Then equation (\ref{DecAEzerop}) can be written in the form
\begin{align} \label{DecAEzerop2}
\AEzerop = \ndp + B^*.	
\end{align}

Since $h$ is of degree $0$ and $\AEzeropp$ is of degree $1$, we have $A^{E_0 \prime}$ is of degree $-1$.
Moreover, since $\AEzeropp$ and $\theta_h$ are $G$-invariant, $\AEzerop$ is also $G$-invariant.

Set 
\begin{align} \label{Superconnection}
\AEzero = \AEzeropp + \AEzerop.
\end{align}
Then $\AEzero$ is a $G$-invariant superconnection on $D$.
Set 
\begin{align} \label{DefC}
& \nabla^D = \ndpp + \ndp, && C = B + B^*.	
\end{align}
Then
\begin{align}
\AEzero = \nabla^D +C.	
\end{align}

The curvature of $\AEzero$ is given by
\begin{align}
A^{E_0, 2} = [\AEzeropp, \AEzerop].	
\end{align}
Then $A^{E_0, 2}$ is a $G$-invariant smooth section of $\lxcEndD$ of degree $0$\footnote{Here, the degree is induced by $\operatorname{deg}_{-}$ on $\lxc$ and the obvious degree on $\End \br{D}$.}. It verifies the Bianchi identities,
\begin{align} \label{Bianchi1}
&[\AEzeropp, A^{E_0, 2}] = 0, && 	[\AEzerop, A^{E_0, 2}] = 0.
\end{align}

\subsection{The equivariant Chern character forms} \label{EquiCherCharForm}
Let us extend \cite[Section 8.1]{BSW} to $G$-equivariant case.

Let $(E, \AEpp)$ be an antiholomorphic $G$-superconnection on $X$.
We fix a splitting as in (\ref{EIsoE0}), (\ref{EIsoE02}), and (\ref{Splitting2}). 
Let $h$ be a $G$-invariant generalized metric, and let $\AEzero$ be the associated superconnection defined in (\ref{Superconnection}).



For $g\in G$, recall that $X_g \subset X$ is the set of fixed points of $g$. Then $\br{E, \AEpp}$, $D$, $\br{E_0, \AEzeropp, \AEzerop}$ and $h$ restrict to the corresponding objects on $X_g$.
Moreover, $g$ preserves the fibers of $D|_{X_g}$, $E_0|_{X_g}$, and commutes with the operators $\AEzeropp |_{X_g}$ and $A^{E_0\prime}|_{X_g}$.

Let $\mathrm{Tr_s}: \lxgc \widehat{\otimes } \End (D) |_{X_g} \rightarrow \lxgc$ be the supertrace and classically it vanishes on supercommutaters.

We fix a square root of $i = \sqrt{-1}$.
Our formulas will not depend on this choice. Let $\varphi$ denote the morphism of $\lxgc$ that maps $\alpha$ to $(2i \pi)^{- \mathrm{deg} \alpha /2} \alpha$.
\begin{defn}
Set
\begin{align} \label{DefChern}
\chg \br{\AEzeropp, h} = \varphi \Trs \brr{g \exp\br{- A^{E_0 , 2}|_{X_g}}} \in \Omega \br{X_g, \C}.
\end{align}
\end{defn}

\begin{rmk} \label{rmk61}
The form $\chg \br{\AEzeropp, h}$ depends only on the restriction of $\br{E_0, \AEzeropp, h}$ on $X_g$ and the cyclic group $\anbr{g} \subset G$ generated by $g$  . 	
\end{rmk}

Consider $\Omega\br{X_g,  \End (D)}$  as a trivial vector bundle on $\MDG$.
Then $g$ acts trivially on $\MDG$ and preserves the fiber of $\Omega\br{X_g,  \End (D) |_{X_g}}$.

Let $\dMDG$ denote the de Rham operator on $\MDG$. 
Let $h$ be the tautological section on $\MDG$ with values in $\Omega\br{X,  \Hom (D, \overline{D}^*)}$.
Then $h^{-1} \dMDG h$ is a $1$-form on $\MDG$ with values in  $\Omega\br{X,  \End (D)}$. 
Let $h^{-1} \dMDG h|_{X_g}$ be the $1$-form on $\MDG$ with values in $\Omega\br{X_g,  \End (D)|_{X_g}}$.

\begin{thm} \label{ThmDefChern}
The form $\chg(\AEzeropp, h)$ lies in $\OEXgC$, it is $d^{X_g}$-closed, and its Bott-Chern cohomology class does not depend on $h \in \MDG$.

More precisely, $\varphi \Trs \brr{\br{ h^{-1} \dMDG h|_{X_g}} g \exp \br{-A^{E_0,2}|_{X_g}}}$ is a 1-form on $\MDG$ with values in $\OEXgC$, and moreover,
\begin{align} \label{ThmChern-Equ1}
\dMDG \chg \br{\AEzeropp, h} = - \frac{\overline{\partial}^{X_g} \partial^{X_g}}{2i \pi} \varphi \Trs \brr{ \br{h^{-1} \dMDG h|_{X_g}  } g\exp \br{-A^{E_0,2}|_{X_g}}}.
\end{align}

If $\sigma$ is a $G$-invariant smooth section of $\Aut^0 (E_0)$, the Bott-Chern cohomology class of $\chg (\AEzeropp, h)$ is unchanged when replacing $\AEzeropp$ by $\sigma \AEzeropp \sigma^{-1}$.
In particular, the Bott-Chern cohomology class of $\chg \br{\AEzeropp, h}$ does not depend on the non-canonical identification in  (\ref{EIsoE0}) and (\ref{EIsoE02}), and depends only on $\AEpp|_{X_g}$.

\end{thm}
\begin{proof}

By  \cite[(8.1.4), (8.1.5)]{BSW}, we know that 
\begin{align}
\exp \br{-A^{E_0,2}|_{X_g}} \in \bigoplus_{p,q} \Omega^{p,q} \br{X_g, \End^{p-q} \br{D}|_{X_g}}.	
\end{align}
This can also be deduced from the fact that $A^{E_0}$ is of degree $0$.
Since $g$ preserves the degree, $g \exp \br{-A^{E_0,2}|_{X_g}}$ lies in the same space. Therefore,  $ \chg \br{\AEzeropp, h}$ lies in $\OEXgC$.

Since $\AEzeropp |_{X_g}$ commutes with $g$, using the Bianchi identity (\ref{Bianchi1}) and the fact that supertraces vanish on supercommutators, we know that $\chg \br{\AEzeropp, h}$ is $d^{X_g}$-closed.

We have
\begin{align} \label{ThmChern-Equ4}
\dMDG \brr{\AEzeropp, \AEzerop} = \brr{\AEzeropp, \brr{\AEzerop, h^{-1} \dMDG h}}.	
\end{align}
By (\ref{ThmChern-Equ4}) and since $g$ commutes with $\AEzerop |_{X_g}$ and $\dMDG$, we have an identity of $\Omega \br{X_g, \C}$-valued $1$-forms on  $\MDG$,
\begin{multline} \label{ThmChern-Equ3}
\dMDG \Trs \brr{ g \exp \br{- \AEzerosq|_{X_g}}} \\= - \Trs \brr{ \brr{\AEzeropp|_{X_g}, \brr{\AEzerop|_{X_g}, g h^{-1} \dMDG h|_{X_g}} }  \exp \br{-\AEzerosq|_{X_g}}}.	
\end{multline}
Using (\ref{Bianchi1}) and (\ref{ThmChern-Equ3}), we get (\ref{ThmChern-Equ1}). In particular, the Bott-Chern class of (\ref{DefChern}) is independent of the choice of $h$.

If $\sigma$ is a $G$-invariant section of $\Aut^0 \br{E_0}$, then $ \sigma^* h \sigma$ is also a $G$-equivariant generalized metric.
Proceeding as in \cite[(8.1.11), (8.1.12)]{BSW}, we have 
\begin{align}
\chg \br{\sigma \AEzeropp \sigma^{-1}, h} = \chg \br{\AEzeropp, \sigma^* h \sigma} \quad in \; \OEXgC.
\end{align}
So the Bott-Chern cohomology class of $\chg \br{\AEzeropp, h}$ is unchanged when replacing $\AEzeropp$ by $\sigma \AEzeropp \sigma^{-1}$.
This gives the last part of the statements, and completes the proof of theorem.
\end{proof}


\begin{defn}
Let $\chgBC \br{\AEpp} \in \HEXgBCC$ be the common Bott-Chern cohomology class of the forms $\chg \br{\AEzeropp, h}$.	
\end{defn}

\begin{rmk}
When $g$ is the identity, then $\chgBC \br{\AEpp} $ has real coefficients \cite[Theorem 8.1.12]{BSW}.
\end{rmk}

\subsection{The equivariant Chern character of pullbacks, tensor products and cones} \label{Section7-3}

We generalize \cite[Proposition 8.3.1, 8.4.1, 8.7.1]{BSW} to $G$-equivariant case. The proof is trivial and left to the readers.
\subsubsection{The equivariant Chern character of pullbacks}
Let $Y$ be a compact complex $G$-manifold. We use the same notation as in Section \ref{BlockPullback}. In particular, $f: X \rightarrow Y$ is a $G$-invariant holomorphic map.
Therefore, for any $g \in G$, $f$ induces a holomorphic map $X_g \rightarrow Y_g$.
Let $\br{F, \AFpp}$ be an antiholomorphic $G$-superconnection on $Y$, and let $(E, \AEpp) = f_b^* (F, \AFpp)$ be the $G$-equivariant antiholomorphic superconnection on $X$ that was defined in Section \ref{BlockPullback}.

We fix a non-canonical identification of $F$ as in  (\ref{EIsoE0}) and (\ref{EIsoE02}), that induces a  corresponding identification of $E$ on $X$. 
Let $h$ be a $G$-invariant generalized metric on $D_F$.
Then $f^* h$ is a $G$-equivariant generalized metric on $D_E$.

\begin{prop} \label{EquiPullbackChern}
The following identities hold
\begin{align} \begin{aligned}\label{PropBlockPullbackChern}
\chg (A^{f_b^* F_0^{\prime\prime}}, f^*h) = f^* \chg (A^{F_0 \prime \prime}, h) \quad in \; \OEXgC . \\
\chgBC \br{A^{f_b^* F\prime\prime}} = f^* \chgBC \br{\AFpp}	\quad \text{in} \; \HEXgBCC .
\end{aligned}
\end{align}
\end{prop}

\subsubsection{Equivariant Chern character and tensor products}
Let $\br{E, \AEpp}$, $(\underline{E}, A^{\underline{E}^{\prime\prime}})$ be antiholomorphic $G$-superconnections on $X$.


We fix a non-canonical isomorphisms for $E$, $\underline{E}$ as in  (\ref{EIsoE0}) and (\ref{EIsoE02}). It induces the non-canonical isomorphism of $E \widehat{\otimes}_b \underline{E}$.
If $h$ and $\underline{h}$ are $G$-equivariant metrics on $D$ and $\underline{D}$ respectively, then $h \widehat{\otimes} \underline{h}$ is a $G$-equivariant generalized metric on $D \widehat{\otimes} \underline{D}$.

\begin{prop} \label{PropBlockTensorChern}
The following identities hold,
\begin{align} \begin{aligned}
\chg \br{A^{E_0 \widehat{\otimes}_b \underline{E}_0 \prime\prime}, h \widehat{\otimes} \underline{h} } &= \chg \br{A^{E_0 \prime\prime}, h} \chg \br{A^{\underline{E}_0 \prime\prime}, \underline{h}} \text{ in }  \OEXgC ,\\	
\chgBC \br{A^{E \widehat{\otimes}_b \underline{E} \prime\prime}} &= \chgBC \br{\AEpp} \chgBC \br{A^{\underline{E} \prime\prime}}  \text{ in }  \HEXgBCC.
\end{aligned}	
\end{align}
\end{prop}
 
\subsubsection{The equivariant Chern character of a cone}
We follow \cite[Section 8.7]{BSW} and use the notation of Section \ref{Complex} and \ref{BlockCone}.
In particular $\phi$ is a $G$-equivariant morphism $E \rightarrow \underline{E}$ of degree $0$, and $\br{C, A_\phi^{C \prime\prime}}$ is the corresponding cone. 
\begin{thm} \label{ThmQuaiChern}
The following identity holds,
\begin{align}
\chgBC \br{A_\phi^{C\prime\prime}} = \chgBC\br{A^{\underline{E} \prime\prime}} - \chgBC\br{\AEpp}\quad in \; \HEXgBCC.	
\end{align}
In particular, if $\phi$ is a quasi-isomorphism, then
\begin{align}
\chgBC \br{\AEpp} = \chgBC \br{A^{\underline{E} \prime\prime}} \quad in \; \HEXgBCC.	
\end{align}
\end{thm}

\subsection{The equivariant Chern character on $\DbCohXG$} \label{EquiChernCharacter}

\begin{thm} \label{ThmDerivedChern}
The class $\chgBC (\AEpp) \in \HEXgBCC$ depends only on the isomorphism  class of $F_X (E, \AEpp)$ in $\DbCohXG$.	
More generally, it depends only on the isomorphism class of $F_{X_g} \br{E|_{X_g}, \AEpp|_{X_g}}$ in $\mathrm{D}^b_{\mathrm{coh}} \br{X_g, \anbr{g}}$.
\end{thm}
\begin{proof}
It is a consequence of Remark \ref{rmk61}, Theorems \ref{Fullyfaithful} and \ref{ThmQuaiChern}. 
\end{proof}

By Theorem \ref{EssSurB}, if $\F$ is an object in $\DbCohXG$, there is an antiholomorphic $G$-superconnection $ \br{E, \AEpp}$ such that $F_X \br{E, \AEpp}$ is isomorphic to $\F$. 
If $(\underline{E}, \AUEpp)$ is another such a pair, by Theorem \ref{ThmDerivedChern}, 
\begin{align}
\chgBC \br{\AEpp} = \chgBC \br{A^{\underline{E} \prime\prime}} \quad in \; \HEXgBCC.	
\end{align}
\begin{defn} \label{ChernDbCoh}
If $\F$ is an object in $\DbCohXG$,
let $\chgBC \br{\F} \in \HEXgBCC$ be the common Bott-Chern class of the above $\chgBC \br{\AEpp}$.	
\end{defn}

By Theorem \ref{ThmDerivedChern}, the equivariant Chern character class $\chgBC \br{\F}$ only depends on $\F|_{X_g}$ in $\mathrm{D}^b_{\mathrm{coh}} \br{X_g, \anbr{g}}$.

In the following theorem we use the notation in Section \ref{GPullBack} and assume $\F$ is an object in $\DbCohYG$.
 \begin{thm} \label{ChernPullback}
 The following identity holds,
 \begin{align} 
 \chgBC \br{Lf^* \F} = f^* \chgBC \br{\F} \quad \text{in}  \; \HEXgBCC.	
 \end{align}
 \end{thm}
 \begin{proof}
  It is a consequence of Propositions \ref{PropCompatiblePullback} and (\ref{PropBlockPullbackChern}).	
 \end{proof}

 If $\F$, $\underline{\F}$ are objects in $\DbCohXG$, their derived tensor product $\F \widehat{\otimes}^L_{\ox} \underline{\F}$ was defined in Section \ref{DerivedTensor}.
 \begin{thm} \label{ChernTensor}
 The following identity holds,
 \begin{align}
 \chgBC \br{\F \widehat{\otimes}^L_{\ox} \underline{\F}} = \chgBC \br{\F} \chgBC \br{\underline{\F}} \quad \text{in} \; \HEXgBCC.	
 \end{align}
 \end{thm}
\begin{proof}
It is a consequence of Propositions \ref{PropCompatibleTensor} and  \ref{PropBlockTensorChern}.
\end{proof}

\subsection{The equivariant Chern character on $K(X,G)$} \label{Section7-5}
In this section we follow \cite[Section 8.9]{BSW} and use the notation in Section \ref{DerivedCategory}.
In particular, the Grothendieck group $K(X,G)$ of $G$-equivariant coherent sheaves on $X$ was defined.

By  (\ref{IsoKDbcohAndKXG}), Theorem \ref{ThmQuaiChern} and results of Section \ref{EquiChernCharacter},  we have the equivariant Chern character on $K(X,G)$,
\begin{align}
\mathrm{ch}_{g,\mathrm{BC}}: K(X,G) \rightarrow  	\HEXgBCC.
\end{align}	




Since all the $\mathscr{H}^i \F$ are $G$-coherent, they can be viewed as elements of $K(X,G)$. 
\begin{thm}
If $\F \in \DbCohXG$, then
\begin{align} \label{DbcohToKXG}
\chgBC (\F) = \sum_i  (-1)^i \chgBC ( \mathscr{H}^i \F).	
\end{align}
\end{thm}
\begin{proof}
It is an immediate consequence of (\ref{MapDbcohToKXG}). 	
\end{proof}

\section{The main result} \label{SectionMainResult}
Let $f: X \rightarrow Y $ be a $G$-invariant holomorphic map between compact complex $G$-manifolds. Recall $f_!: K(X, G) \rightarrow K(Y,G)$   is the direct image defined in (\ref{DirectK}), 
and $f_*: \HEXgBCC \rightarrow \HEYgBCC$ is the pushforward on the Bott-Chern cohomology defined in \eqref{PushforwardCoh}.
In this article, we will prove 
\begin{thm} \label{MainTheorem}
If $\F \in K(X,G)$, then
\begin{align}
\tdgBC \br{TY} \chgBC \br{f_! \F} = f_* \brr{\tdgBC \br{TX} \chgBC  \br{\F}}	\; \text{in} \; \HEYgBCC.
\end{align}
	
\end{thm}
\begin{proof}
The theorem will be established in three steps.

\textbf{Step 1}: the case when $f$ is an embedding. This will be given in Section \ref{SectionEmbedding}.

\textbf{Step 2}: the case when $f$ is a projection. This will be give in Sections \ref{SectionSubmersion1}-\ref{SectionSubmersion2}.

\textbf{Step 3}: the general case. 

Let us complete the \textbf{Step 3} admitting \textbf{Steps 1} and \textbf{2}.
Let $i: x \in X \rightarrow \br{x, f(x)} \in X\times Y$ be the embedding associated to $f$, let $p: X\times Y \rightarrow Y $ be the obvious projection. Then $f = pi$. 

We have the following diagram,
\begin{equation} \label{DiagRRG}
\begin{tikzcd}[column sep=4em, row sep=2.5em]
{K(X,G)} \arrow[r, "i_!"] \arrow[d, "\tdgBC(TX) \chgBC" description] \arrow[rr, "f_!", bend left=30] & {K(X\times Y, G)} \arrow[r, "p_!"] \arrow[d, "\tdgBC(T(X\times Y)) \chgBC" description] & {K(Y,G)} \arrow[d, "\tdgBC(TY) \chgBC" description] \\
\HEXgBCC \arrow[r, "i_*"] \arrow[r] \arrow[rr, "f_*", bend right=30] & \HEXYgBCC \arrow[r, "p_*"] & \HEYgBCC .    
\arrow[phantom, "\mathsf{A}" {description}, from=1-1, to=2-2]
\arrow[phantom, "\mathsf{B}" {description}, from=1-2, to=2-3]
\arrow[phantom, "\mathsf{C}" {description, pos=0.5, yshift=4ex}, from=1-1, to=1-3]
\arrow[phantom, "\mathsf{D}" {description, pos=0.5, yshift=-4ex}, from=2-1, to=2-3]
\end{tikzcd}
\end{equation}

By the first two steps, the diagrams $\mathsf{A}$ and $\mathsf{B}$ in (\ref{DiagRRG}) commute.
By functoriality (\ref{Equ3-1}) and the results in Section \ref{Section2-2}, the diagrams $\mathsf{C}$ and $\mathsf{D}$ in (\ref{DiagRRG}) also commute.
Therefore, the big diagram commutes which completes the \textbf{Step 3}.
\end{proof}

\section{A Riemann-Roch-Grothendieck theorem for equivariant embedding} \label{SectionEmbedding} \label{SectionEmbedding}
The purpose of this section is to prove the Riemann-Roch-Grothendieck theorem for equivariant embeddings.

In Section \ref{Section9-1}, we give a technical lemma.

In Section \ref{MainImmersion}, we establish our main theorem for an equivariant embedding $i_{X, Y}: X\to Y$. Using the technical lemma given in Section \ref{Section9-1}, we can reduce the problem to the case when $G$ acts trivially on $X$ and $Y$.

In Section \ref{ProofEmbedding}, we prove the Riemann-Roch-Grothendieck theorem for equivariant embeddings when $G$ acts trivially on $X$ and $Y$ by using the method of deformation to the normal cone.
\subsection{A technical lemma} \label{Section9-1}
Let $Z$ be a compact complex $G$-manifold.
Let $X$, $Y$ be two compact complex submanifolds of $Z$.
We assume that $U = X \cap Y$ is a submanifold of $Z$.
Assume also that the $G$-action on $Z$ preserves $X$, $Y$ and in particular $U$. 
We use the obvious embedding as indicated in the following commutative diagram,
\begin{equation} \label{DiagEmb}
\begin{tikzcd}
U \arrow[r, "{i_{U,X}}"] \arrow[d, "{i_{U,Y}}"'] & X \arrow[d, "{i_{X,Z}}"] \\
Y \arrow[r, "{i_{Y,Z}}"]                         & Z  .                     
\end{tikzcd}	
\end{equation}

Let  $\widetilde{N}$ be the excess normal bundle on $U$ associated to (\ref{DiagEmb}), \ie
\begin{align}
\widetilde{N} = \frac{TZ|_U}{TX|_U + TY|_U}.
\end{align}
Then, $\widetilde{N}$ is $G$-equivariant. We have the short exact sequence of equivariant holomorphic vector bundles on $U$,
\begin{align} \label{Equ7-34}
0 \rightarrow N_{U/Y} \rightarrow N_{Y/Z} \rightarrow \widetilde{N} \rightarrow 0.	
\end{align}

Note that if $X$, $Y$ intersects transversally in $Z$, then $\widetilde{N} = 0$.

Recall that for an embedding, the direct image functor is exact so is equal to the derived direct image.
We will not distinguish these two in this section.

Let $\F$ be an object in $\DbCohXG$. 
We consider $\OO_U \br{\Lambda \br{\widetilde{N}^*}}$ as a complex of zero differentials, then it is an object in $\DbCohUG$.

The following proposition is a generalization of \cite[Proposition 9.1.1]{BSW}.

\begin{prop} \label{EquiTranProp}
There exists an isomorphism in $\DbCohYG$,
\begin{align} \label{Equ7-12}
Li^*_{Y,Z} i_{X,Z,*} \F \simeq i_{U,Y,*} \br{Li^*_{U,X} \F \wo^L_{\OO_U} \OO_U \br{\Lambda  \br{\widetilde{N}^*} }}.
\end{align}
In particular, if $X$ and $Y$ intersects transversally in $Z$, then we have an isomorphism in $\DbCohYG$,
\begin{align} \label{Equ8-3}
Li^*_{Y,Z} i_{X,Z,*} \F \simeq i_{U,Y,*} Li^*_{U,X} \F.	
\end{align}
\end{prop}

\begin{proof}

We assume $\F = \underline{F}_X \br{\br{E_X, A^{E_X\prime\prime}}}$ where  $\br{E_X, A^{E_X\prime\prime}}$ is an antiholomorphic $G$-superconnection on $X$. 


As discussed in Subsection \ref{SubsectionComTen}, $i_{X,Z,*} \F$ can be viewed as a $G$-equivariant $\OO_Z^\infty \br{\Lambda \br{ \overline{T^*Z}}}$-module.
Also it defines an object in $\DbCohZG$, that coincides with its right-derived direct image.

By Proposition \ref{PropCompatibleDirectImage}, there exists an antiholomorphic $G$-superconnection $\br{E_Z, A^{E_Z\prime\prime}}$ on $Z$ with $\E_Z = \underline{F}_Z \br{\br{E_Z,A^{E_Z\prime\prime }}}$, and a quasi-isomorphism of $G$-equivariant $\OO_Z$-complexes $r_{Z,X}: \E_Z \rightarrow i_{X,Z,*}\F$, which induces a morphism of $\OO_Z^\infty \br{\Lambda \br{\overline{T^*Z}}}$-modules.
By Proposition \ref{PropCompatiblePullback}, $Li^*_{Y,Z} i_{X,Z,*} \F$ is represented by $i_{Y,Z,b}^* \E_Z$.

To make our notation simpler, we will use the notation $\F$ instead of $i_{X,Z,*} \F$.
A similar notation will be used when considering the embedding $i_{U,Y}$.

The map $r_{Z,X}$ induces a map $r_{Y,U}: i_{Y,Z,b}^* \E_Z \rightarrow  i^*_{U,X,b} \F$, which is a $G$-equivariant morphism of $\OO_Y^\infty \br{\lyb}$-modules and of $\oy$-complexes.

Let $\mathscr{T}_U = \br{\Lambda \br{\widetilde{N}^*} \widehat{\otimes} \Lambda \br{\overline{T^*U}}, \db^U }$.

Then $\mathscr{T}_U$ is an antiholomorphic $G$-superconnection on $U$, and it  represents $\OO_U \br{\Lambda  \br{\widetilde{N}}^* }$ in $\DbCohUG$.

The map $r_{Y,U}$ can be extended naturally to a map
\begin{align} \label{rYU}
	i_{Y,Z,b}^* \E_Z \rightarrow i^*_{U,X,b} \F \wo_b \mathscr{T}_U,
\end{align}
which we still denote it as $r_{Y,U}$.

To establish our proposition, we need to show that map $r_{Y,U}$ in (\ref{rYU}) is a quasi-isomorphism of objects in $\DbCohYG$. The proof will be obtained via local arguments and it is irrelevant to the group actions.

Let $D_Z$ be the diagonal of the antiholomorphic $G$-superconnection of $\E_Z$, and $v_{Z,0}$ be the corresponding map on $D_Z$ defined in (\ref{DecOfAnSupCon}).
Let $\mathscr{H} \E_Z$ be the cohomology of the $\OO_Z$-complex $\E_Z$, and let $HD $ be the cohomology of $\br{ \OO_Z^\infty \br{D_Z}, v_{Z,0}}$.

Since $r_{Z,X}$ is a quasi-isomorphism, on $Z \backslash X$, $\mathscr{H} \E_Z = 0$.
By \cite[Theorem 5.3.4]{BSW}, on $Z \backslash X$, $HD_Z = 0$ , so that on $Y \backslash U$, $H i_{Y,Z}^* D_Z = 0$.
Using again \cite[Theorem 5.3.4]{BSW}, we find that on $Y \backslash U$, $\mathscr{H} i_{Y,Z,b}^* \E_Z = 0$.

Take $x \in U$.
If $V \subset Z$ is a small neighborhood of $x$, we choose a holomorphic coordinate system on $V$ such that $x$ is represented by $0 \in \C^n$, and $X$, $Y$ are two subspaces $H_X$, $H_Y$ of $\C^n$, so that $U$ is represented by $H_X \cap H_Y$.
If $K$ is a vector subspace of $H_Y$ such that $H_X \oplus K = H_X + H_Y$, then $K$ represents $N_{U,Y}$.
Let $N$ be a subspace of $\C^n$ such that $\br{N_X + N_Y} \oplus N = \C^n$.
Let $V_{H_X}$, $V_K$, and $V_N$ be open neighborhoods of $0$ in $H_X$, $K$ and $N$ respectively, so that $V_{H_X} \times V_K \times V_N \subset V$.
By \cite[Theorem 5.2.1]{BSW}, if $V_{H_X}$ is small enough, after a gauge transformation of total degree $0$, $A^{E_X \prime\prime}$ can be written in the form
\begin{align} \label{LocalAEX}
A^{E_X \prime\prime} = \nabla^{D_X\prime\prime} + v_{X,0}.	
\end{align}
We note here that the above transformation is not necessarily $G$-invariant.

Let $\pi_{H_X}$, $\pi_{K}$ and $\pi_N$ be the projections of $\C^n$ on $H_X$, $K$ and $N$ respectively. Let $y_K$, $y_N$ be the generic section of $K$ and $N$, and let $\br{\Lambda \br{K^*}, i_{y_K}}$, $\br{\Lambda \br{N^*}, i_{y_N}}$ denote the Koszul complex of $K$ and $N$.
Set 
\begin{align}
\mathbf{E}_Z = \Lambda \br{\overline{\C}^n} \wo \pi^*_{H_X} D_X \wo \pi_K^* \Lambda\br{ K}^* \wo \pi_{N^*} \Lambda \br{N^*}.	
\end{align}
Let $A^{\mathbf{E}_Z \prime\prime}$ be the antiholomorphic superconnection
\begin{align}
A^{\mathbf{E}_Z \prime\prime} = \pi^*_{H_X} \br{\nabla^{D_X \prime\prime} + v_{X,0}} + \pi_K^* \br{\db^K + i_{y_K}} + 	\pi_N^* \br{\db^N + i_{y_N}}.
\end{align}
Then $A^{\mathbf{E}_Z \prime\prime}$ is an antiholomorphic superconnection on $\C^n$ near $0$.
Let $\mathcal{E}_Z$ denote the corresponding complex of $\OO_V$-modules.
Let $r_{\C^n, H_X}$ be the projection $\pi_{H_X}^* D_X \wo \pi_K^* \Lambda\br{ K^*} \wo \pi_N^* \Lambda \br{N^*} \rightarrow D_X$.
Then $r_{\C^n, H_X}$ extends to a morphism $\mathcal{E}_Z \rightarrow \F$ which has the same properties as $r_{Z,X}$. 
Because of known properties of the Koszul complex, if the considered neighborhoods are small enough, $r_{\C^n, H_X}$ is a quasi-isomorphism.

Near $z \in Z$, both $r_{Z,X}$ and $r_{\C^n, H_X}$ provide quasi-isomorphism of $\E_Z$, $\mathcal{E}_Z$ with $\F$, so that for $V$ small enough, $\E_Z$ and $\mathbf{E}_Z$ are isomorphic in $\mathrm{D}^{\mathrm{b}}_{\mathrm{coh}} \br{V_{H_X} \times V_K \times V_N}$.
By \cite[Theorem 6.5.1]{BSW}, there is a corresponding quasi-isomorphism $\phi: \E_Z \rightarrow \mathcal{E}_Z$.
So $\phi$ induces a morphism $i^*_{Y,Z,b} \E_Z \rightarrow i^*_{Y,Z,b} \mathcal{E}_Z$.
By \cite[Proposition 6.4.1]{BSW}, for $z \in V_{H_X} \times V_K \times V_N$, $\phi_z: D_{Z,x} \rightarrow \br{\pi^*_{H_X} D_X \wo \pi^*_K \Lambda \br{K^*} \wo \pi_N^* \Lambda \br{ N^*}}_z$ is a quasi-isomorphism.
In particular, this will be true for $z\in \br{V_{H_X} \times V_K \times V_N} \cap Y$.
This shows that near $z\in Y$, $\phi$ induces a quasi-isomorphism $i_{Y,Z,b}^* \E_Z \rightarrow i_{Y,Z,b}^* \mathcal{E}_Z$.

Since near $x\in Y$, the restriction of $\br{\Lambda \br{K^*}, i_{y_K}}$ to $Y$ is a Koszul complex, $r_{K,U}: i^*_{Y,Z,b} \mathcal{E}_Z \rightarrow i^*_{U,X,b} \F \wo_b \Lambda \br{N^*}$ is a quasi-isomorphism.
This shows that $r_{Y,U}: 	i_{Y,Z,b}^* \E_Z \rightarrow i^*_{U,X,b} \F \wo_b \mathscr{T}_U
$ is a quasi-isomorphism.
The proof of our proposition is complete.
\end{proof}

\subsection{The Riemann-Roch-Grothendieck for equivariant embedding} \label{MainImmersion}
Let $i_{X,Y}: X \rightarrow Y$ be a $G$-equivariant holomorphic embedding of compact complex $G$-manifolds.
Then for any $g\in G$, $i_{X_g, Y_g}: X_g \rightarrow Y_g$ is an embedding of compact complex manifolds.

Let $\F$ be an object in $\DbCohXG$.
In the rest part of the section, we will prove a Riemann-Roch-Grothendieck theorem for $i_{X,Y}$.
\begin{thm} \label{Thm7-4}
For any $g\in G$, the following identity holds,
\begin{align} \label{Equ7-13}
\chgBC \br{i_{X,Y,*} \F} = i_{X_g,Y_g, *}	\frac{\chgBC \br{\F}}{\tdgBC \br{N_{X/Y}}} \quad \text{in} \; \HEYgBCC.
\end{align}
\end{thm}
\begin{proof}
The theorem will be proved in two steps.

\textbf{Step 1}. We will prove the theorem when $G$ acts trivially on $X$ and $Y$ in Section \ref{ProofEmbedding}.

\textbf{Step 2}. The general case.

Let us complete \textbf{Step 2} admitting \textbf{Step 1}.
We replace $\br{X,Y,Z,U = X \cap Y,G}$ by $\br{X,Y_g,Y, X_g = X \cap Y_g,\anbr{g}}$ in Proposition \ref{EquiTranProp}. We have an analogue of the diagram (\ref{DiagEmb}) with obvious notations,
\begin{equation}
\begin{tikzcd}
X_g \arrow[r, "{i_{X_g,X}}"] \arrow[d, "{i_{X_g,Y_g}}"'] & X \arrow[d, "{i_{X,Y}}"] \\
Y_g \arrow[r, "{i_{Y_g,Y}}"]                         & Y .                      
\end{tikzcd}	
\end{equation}
 We still denote $\widetilde{N}$ the associated excess normal bundle on $X_g$. 
By Proposition \ref{EquiTranProp}, we have
\begin{align} \label{Equ7-14}
	 L i_{Y_g,Y}^* i_{X,Y,*} \F \simeq i_{X_g, Y_g, *} \br{Li_{X_g, X}^* \F  \wo^L_{\OO_{X_g}} \OO_{X_g} \br{\Lambda  \br{\widetilde{N}^*} }} \quad  \text{in} \; \DbCohYgG.
\end{align}

By Theorems \ref{ThmDefChern}, the equivariant Chern character of the left hand side of (\ref{Equ7-14}) is given by
\begin{align} \label{Equ7-15}
	\chgBC \br{ L i_{Y_g,Y}^* i_{X,Y,*} \F} = \chgBC \br{i_{X,Y,*} \F}   \quad \text{in} \; \HEYgBCC.
\end{align}

Since $\anbr{g}$ acts trivially on $X_g$, $Y_g$, by \textbf{Step 1}, the equivariant Chern character of the right hand side of  (\ref{Equ7-14}) is given by
\begin{multline}
\chgBC 	\br{i_{X_g, Y_g, *} \br{Li_{X_g, X}^* \F  \wo^L_{\OO_{X_g}} \Lambda  \br{\widetilde{N}^*} }} \\
=   i_{X_g, Y_g, *} \frac{\chgBC \br{Li_{X_g, X}^* \F  \wo^L_{\OO_{X_g}} \Lambda  \br{\widetilde{N}^*} }}{\tdgBC \br{N_{X_g/Y_g}}} \quad \text{in} \; \HEYgBCC.
\end{multline}
By Theorems \ref{ThmDefChern} and \ref{ChernTensor}, 
\begin{align}
	\chgBC \br{Li_{X_g, X}^* \F  \wo^L_{\OO_{X_g}} \Lambda  \br{\widetilde{N}^*} } = \chgBC \br{\F} \chgBC \br{ \Lambda  \br{\widetilde{N}^*} }.
\end{align}

We claim
\begin{align} \label{Equ7-17}
	\chgBC \br{\Lambda  \br{\widetilde{N}^*}} = \frac{\tdgBC\br{N_{X_g/Y_g}}}{\tdgBC\br{N_{X/Y}}} \quad \text{in} \; \HEXgBCC.
\end{align}
Indeed,
by Definitions \ref{DefTdg} and \ref{DefChg}, since the fiberwise action of $g$ on $\widetilde{N}_{\mid X_g}$ has no eigenvalue $1$, we have 
\begin{align} \label{Equ7-16}
\chgBC \br{\Lambda \br{\widetilde{N}^*}} = \tdgBC^{-1} \br{\widetilde{N}} \quad \text{in} \; \HEXgBCC.
\end{align}
By (\ref{Equ7-34}), we have an exact sequence of $\anbr{g}$-equivariant holomorphic vector bundles on $X_g$,
\begin{align} \label{Equ7-18}
0 \longrightarrow N_{X_g/Y_g} \longrightarrow N_{X/Y} \longrightarrow \widetilde{N} \longrightarrow 0.	
\end{align}
By (\ref{Equ7-18}) and by multiplicativity of the Todd class, we have 
\begin{align} \label{Equ7-19}
\tdgBC^{-1} \br{\widetilde{N}} = 	\tdgBC^{-1}\br{N_{X/Y}}\tdgBC\br{N_{X_g/Y_g}} \quad \text{in} \; \HEXgBCC.
\end{align}
By (\ref{Equ7-16}), (\ref{Equ7-19}), we get (\ref{Equ7-17}). 

By (\ref{Equ7-15})-(\ref{Equ7-17}), we complete the proof of the theorem.
\end{proof}

\subsection{The case $G$ acts trivially on $X$ and $Y$} \label{ProofEmbedding}
Now we assume that $G$ acts trivially on $X$ and $Y$ and all the manifolds considered in this section.
Let $\widetilde{Y}$ be the blow-up of $Y$ along $X$ with the exceptional divisor $\PPP (N_{X/Y})$.
We have the identity of sets 
\begin{align}
	\widetilde{Y} = Y\backslash X \sqcup \mathbf{P} \br{N_{X/Y}} .
\end{align}

Let $W $ be the blow-up of $Y \times \PPP^1$ along $X \times \infty$. 
The associated exceptional divisor is given by
\begin{align}
P = \PPP \br{N_{X\times \infty/ Y \times \PPP^1}}.	
\end{align}
Put
\begin{align}
A =  N_{X/Y} \boxtimes N^{-1}_{\infty / \PPP^1}.	
\end{align}
Then $A$ is a vector bundle on $X \times \infty$.
So
\begin{align}
P = \PPP (A \oplus \C).
\end{align}

Let $q_{W,Y}: W \rightarrow Y$, $q_{W, \PPP^1}: W \rightarrow \PPP^1$ be the obvious maps. 
For $z \in \PPP^1$, put
\begin{align}
Y_z = q^{-1}_{W, \PPP^1} (z).	
\end{align}
Then
\begin{align}
Y_z = 
\begin{cases}
 Y ,  &\text{if} \; z \neq \infty, \\
P \cup \widetilde{Y},  &\text{if} \; z = \infty.	
\end{cases}
\end{align}
Moreover, $P$ and $\widetilde{Y}$ meet transversally along $\PPP(N_{X/Y})$ (See Figure \ref{Fig1}).

\begin{figure} 
    \includegraphics[width=3in]{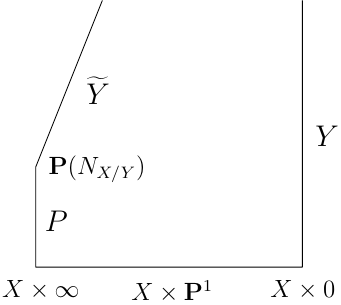}
	\caption{Deformation to the normal cone}
	\label{Fig1}
	\centering
    \end{figure}

Also $X\times \PPP^1$ embeds in $W$ and $X\times \infty $ embeds in $P$. We denote $i_{X\times \infty, P}: X \times \infty \rightarrow P$ the embedding. 
We note here that 
\begin{align} \label{Equ7-35}
N_{X \times \infty /P } \simeq N_{X/Y}.	
\end{align}

\textbf{Step 1}. Let us reduce our embedding problem for $i_{X, Y}$ to $i_{X\times \infty, P}$.

Let 
\begin{align}
&i_{Y,W}: Y = q_{W,\PPP^1}^{-1}(0) \rightarrow	 W, & i_{P,W}:P\rightarrow W, && i_{\widetilde{Y}, W}: \widetilde{Y} \rightarrow W, 
\end{align}
be the obvious embeddings. Other embeddings will be denoted in the same way.
Let $\pi_{X\times \PPP^1, X}: X\times \PPP^1 \rightarrow X$ be the natural projection. 

Let $\F \in \DbCohXG$.
By (\ref{Equ8-3}) and proceeding as in \cite[(9.3.4)-(9.3.6)]{BSW}, we have 
\begin{align} \label{IdenPullback}
\br{Li^*_{Y,W}} i_{X\times \PPP^1,W,*} \br{L \pi^*_{X\times \PPP^1,X} \F} &\simeq i_{X,Y,*} \F \quad \quad \; \text{in} \; \DbCohYG, \\
\br{Li^*_{P,W}} i_{X\times \PPP^1,W,*} \br{L \pi^*_{X\times \PPP^1,X} \F} &\simeq i_{X\times \infty,P,*} \F \quad \text{in} \; \DbCohPG. \nonumber
\end{align}

Taking equivariant Chern character of both sides of (\ref{IdenPullback}), using Theorem \ref{EquiPullbackChern}, and (\ref{IdenPullback}), we get
\begin{align} \label{Equ8-4}
i^*_{Y, W} \chgBC \br{i_{X\times \PPP^1, W, *} \br{L \pi^*_{X\times \PPP^1,X} }\F}	= \chgBC \br{i_{X, Y, *} \F} \quad \text{in} \; H^{(=)}_{\mathrm{BC}} (Y,\C), \\
i^*_{P, W} \chgBC \br{i_{X\times \PPP^1, W, *} \br{L \pi^*_{X\times \PPP^1,X} }\F}	= \chgBC \br{i_{X \times \infty, P, *} \F} \quad \text{in} \; H^{(=)}_{\mathrm{BC}} (P,\C). \nonumber
\end{align}

Let $q_{P,X}: P \rightarrow X$ be the obvious projection. We claim
\begin{align} \label{DeformationOfChern}
\chgBC \br{i_{X,Y,*} \F} =i_{X, Y,*} q_{P, X,*} \br{\chgBC \br{i_{X\times \infty, P, *} \F}}	 \quad \text{in} \; H^{(=)}_{\mathrm{BC}} (Y,\C).
\end{align}
Indeed, the proof is similar to \cite[(9.3.9)-(9.3.17)]{BSW}. Using (\ref{Equ8-4}), the fact that $\br{X \times \PPP^1} \cap \widetilde{Y } = \emptyset$, and $i_{X\times \PPP^1,W,*} \br{L \pi^*_{X\times \PPP^1,X} \F} $ is supported in $X \times \PPP^1$, the term 
\begin{align}
\frac{\db^Y \partial^Y}{2 i \pi} q_{W,Y,*} \brr{\chgBC \br{i_{X\times \PPP^1, W, *} \br{L \pi^*_{X\times \PPP^1,X} }\F} q^*_{W, \PPP^1} \log \br{|z|^2}}
\end{align}
gives the transgression between the two sides of (\ref{DeformationOfChern}).


\textbf{Step 2}. Let us reduce the evaluation of $\chgBC \br{i_{X\times \infty, P, *} \F}$ to $ \chgBC \br{i_{X\times \infty, P, *} \ox}$.

From now on, we identify $X$ with $X \times \infty$.
One feature of the embedding $X\rightarrow P$ compared to the embedding $X \rightarrow Y$ is that we have the following commutative diagram,
\begin{equation} \label{Equ8-5}
\begin{tikzcd}
X \arrow[r, "{i_{X,P}}"] \arrow[rd, "\mathrm{id}"'] & P \arrow[d, "{q_{P,X}}"] \\
                                           & X.                      
\end{tikzcd}
\end{equation}

Using the projection formula \cite[\href{https://stacks.math.columbia.edu/tag/0B54}{Tag0B54}]{stacks-project} for the holomorphic map $i_{X, P}$, and for the objects $\ox$ in $\DbCohXG$ and $Lq^*_{P,X} \F$ in $\DbCohPG$,
we have 
\begin{align} \label{Projection1}
Lq^*_{P,X} \F 	\widehat{\otimes}^{L}_{\OO_P} i_{X, P,*} \ox \simeq i_{X, P,*} \br{L i_{X, P}^* Lq^*_{P,X} \F } \quad \text{in} \; \DbCohPG.
\end{align}
Since the diagram (\ref{Equ8-5}) commutes,  we can rewrite (\ref{Projection1}) as 
\begin{align}  \label{Equ8-6}
	Lq^*_{P,X} \F 	\widehat{\otimes}^{L}_{\OO_P} i_{X, P,*} \ox \simeq i_{X, P,*} \F \quad \text{in} \; \DbCohPG.
\end{align}
Taking equivariant Chern character of both sides of (\ref{Equ8-6}) and using Theorems \ref{ChernPullback}, \ref{ChernTensor}, we obtain
\begin{align} \label{Equ8-7}
q^*_{P,X}\chgBC \br{ \F 	}\chgBC\br{ i_{X, P,*} \ox} = \chgBC \br{i_{X, P,*} \F}	  \quad \text{in} \; H^{(=)}_{\mathrm{BC}} (Y,\C).
\end{align}
By (\ref{DeformationOfChern}), (\ref{Equ8-7}), we have 
\begin{align} \label{Equ7-32}
	\chgBC \br{i_{X,Y,*} \F} = i_{X,Y,*} \brr{ \chgBC \br{\F} q_{P,X,*}\br{ \chgBC \br{i_{X, P, *} \ox} }} \quad \text{in} \; H^{(=)}_{\mathrm{BC}} (Y,\C).
\end{align}



\textbf{Step 3}. Evaluate $q_{P,X,*} \br{\chgBC \br{i_{X, P, *} \ox}}$.

Let $U = \OO_P (-1)$ be the universal line bundle on $P$. We have the exact sequence of $G$-equivariant holomorphic vector bundles on $P$,
\begin{align}
0 \rightarrow U \rightarrow A \oplus \C \rightarrow \br{A \oplus \C}/U \rightarrow 0.	
\end{align}
The image $\sigma$ of $1 \in \br{A \oplus \C}/U$ is a holomorphic section of $\br{A \oplus \C}/U$ that vanishes exactly on $X \subset P$.

The associated Koszul complex $\br{\Lambda \br{\br{A \oplus \C}/U}^*, i_{\sigma}}$ on $P$ is $G$-equivariant and provides a resolution of $i_{X , P, *} \OO_{X }$.


By Theorem \ref{ThmDerivedChern}, we have
\begin{align} \label{Equ7-10}
	\chgBC \br{i_{X, P, *} \ox}	= \chgBC \br{ \Lambda \br{\br{A \oplus \C}/U}^*, i_{\sigma}} \quad \text{in} \; H^{(=)}_{\mathrm{BC}} (P,\C).
\end{align}
By \cite[Theorem 6.7]{B95}, we know that 
\begin{align} \label{Equ7-31}
q_{P,X, *} \br{\chgBC 	\br{\Lambda \br{\br{A \oplus \C}/U}^*, i_\sigma}} = \mathrm{Td}_{g,\mathrm{BC}}^{-1} \br{N_{X / P}} \quad \text{in} \; H^{(=)}_{\mathrm{BC}} (X,\C).
\end{align}
By (\ref{Equ7-10}), (\ref{Equ7-31}), we obtain
\begin{align} \label{Equ7-33}
	q_{P,X,*} \br{\chgBC \br{i_{X, P, *} \ox}} = \mathrm{Td}_{g,\mathrm{BC}}^{-1} \br{N_{X / P}} \quad \text{in} \; H^{(=)}_{\mathrm{BC}} (P,\C).
\end{align}

Using (\ref{Equ7-35}), (\ref{Equ7-32}) and (\ref{Equ7-33}), we complete the proof of the theorem. \qed

\section{The uniqueness of the equivariant Chern character} \label{SectionUnicity}
The purpose of this section is to establish a unicity theorem for the equivariant Chern characters which is inspired by \cite{G10}.

In Section \ref{UniMain}, we state the main theorem of this section.

In Section \ref{PreUni}, we establish some auxiliary results that will be used in the proof of the unicity theorem.

In Section \ref{Section8-3}, following \cite[Theorem 8]{G10}, we prove our unicity theorem.
\subsection{The unicity theorem} \label{UniMain}
Let $G$ be a finite group and $g\in G$. 
Recall  that $\anbr{g} \subset G$ is the cyclic group generated by $g$. 

\begin{thm} \label{Thm8-1}
Assume for each compact complex $G$-manifold $X$,
there is a morphism of groups $\cgBCBar: K(X,G) \rightarrow \HEXgBCC$  such that
\begin{enumerate}  [start=1]
\item \label{Unicity4} The map 
\begin{align}
\cgBCBar: K(X,G) \rightarrow \HEXgBCC 	
\end{align}
factors through  $K(X_g, \anbr{g})$, so that the following diagram \begin{equation}
\begin{tikzcd}
{K(X,G)} \arrow[d] \arrow[rd, "\cgBCBar"] &   \\
{K(X_g,\anbr{g})} \arrow[r, "\cgBCBar"]          & \HEXgBCC
\end{tikzcd}	
\end{equation}
commutes.
Here, the vertical arrow is the obvious restriction.
\item \label{Unicity1}	If $E$ is an equivariant vector bundle, $\cgBCBar$ is defined in usual way.
\item \label{Unicity2} $\cgBCBar$ is functorial under equivariant pullbacks.
\item \label{Unicity3} $\cgBCBar$ verifies Riemann-Roch-Grothendieck for equivariant embeddings.
\end{enumerate}
Then $\cgBCBar = \cgBC$.
\end{thm}
\begin{proof}
Thanks to (\ref{Unicity4}), we need only prove
\begin{align} \label{UnicityEqu1}
\cgBC = \cgBCBar
\end{align}
for manifolds with trivial $G$-action.
The proof will be given in Section \ref{Section8-3}.
\end{proof}

\subsection{Preliminaries and auxiliary results} \label{PreUni}
In the rest part of this section, we assume that $G$ acts trivially on $X$.
Let $\F$ be a $G$-coherent sheaf on $X$. We denote $\F_{\mathrm{tor}}$ the subsheaf of the torsion elements. 
If $\F^* = \Hom_{\ox} \br{\F, \ox}$ is the dual sheaf of $\F$, if $\sigma : \F \rightarrow \F^{**}$ is the evaluation map, by \cite[Section 5.4]{K14}, $\F_{\mathrm{tor}} = \ker \br{\sigma}$.

Since $\F$ is a $G$-sheaf, we find $\F^*$, $\F^{**}$, and $\F_{\mathrm{tor}}$ are also $G$-sheaves. By \cite[pp.~237, 240]{GR},  $\F^*$, $\F^{**}$, and $\F_{\mathrm{tor}}$ are coherent, so that they are $G$-coherent.

The following proposition is essential to establishing the unicity theorem.
\begin{prop} \label{Prop8-1}
	There is a modification $\sigma: \widetilde{X} \rightarrow X$ which is a finite composition of blow-up with smooth centers such that $\sigma^* \F / (\sigma^* \F)_{\mathrm{tor}}$ is locally free.
Moreover, the exceptional divisor of $\sigma$ and the support of $(\sigma^* \F)_{\mathrm{tor}}$ are both contained in a normal crossing divisor.
\end{prop}

To prove Proposition \ref{Prop8-1}, we need the following two theorems.
\begin{thm} \label{Thm8-2}
There is a modification $\sigma_1: \overline{X} \rightarrow X$ which is a finite sequence of blow-up along smooth centers such that $\sigma_1^* \F / \br{\sigma_1^*\F}_{\mathrm{tor}}$ is locally free.
\end{thm}
\begin{proof}
This is Hironaka's flattening theorem \cite[Theorem 4.4]{H75}. 	
In the reference, $\sigma_1^* \F / \br{\sigma_1^*\F}_{\mathrm{tor}}$ is $\OO_{\overline{X}}$-flat. However, since $\sigma_1^* \F / \br{\sigma_1^*\F}_{\mathrm{tor}}$ is coherent, and by  \cite[Theorem 7.10]{H86}, it is locally free. 
\end{proof}

\begin{thm} \label{Thm8-3}
Let $Y \subset X$ be a closed nowhere dense analytic subspace.
There exists a modification $\sigma_2: \widetilde{X} \rightarrow X$ which is a composition of a finite sequence of blow-up along smooth centers, such that $\sigma^{-1}_2 Y$ is a divisor with normal crossing.	
\end{thm}
\begin{proof}
This is Hironaka's desingularization	, see \cite{H64}, and also \cite[Theorem 5.4.2]{AHV}.
\end{proof}


\begin{proof} [Proof of Proposition \ref{Prop8-1}]
Take $\sigma_1: \overline{X} \rightarrow X$ in Theorem \ref{Thm8-2}, so that $\sigma_1^* \F / \br{\sigma_1^*\F}_{\mathrm{tor}}$ is locally free.

Recall that $\mathrm{supp}(\sigma_1^*\F)_\mathrm{tor}$ is a closed nowhere dense analytic subspace of $\overline{X}$ \cite[Theorem 5.5.8]{K14}.
Take $Y$ to be the union of $\mathrm{supp}(\sigma_1^*\F)_\mathrm{tor}$ and the exceptional divisor of $\sigma_1$. 
Applying Theorem \ref{Thm8-3} for $Y \subset \overline{X}$, there exists a modification $\sigma_2: \widetilde{X} \rightarrow \overline{X}$ which is a composition of a finite sequence of blow-up along smooth centers such that $\sigma_2^{-1} Y$ is a divisor with normal crossing.

Set $\sigma = \sigma_2 \sigma_1$. It remains to prove that $\sigma^* \F / \br{\sigma^* \F}_{\mathrm{tor}}$ is locally free.
We have
\begin{align} \label{Equ8-10}
\sigma_2^* \br{\sigma_1^* \F}_{\mathrm{tor}} \subseteq 	\br{\sigma^* \F}_{\mathrm{tor}}.
\end{align}
Since $\sigma_1^* \F / \br{\sigma_1^*\F}_{\mathrm{tor}}$ is locally free, $\sigma_2^* \br{\sigma_1^* \F / \br{\sigma_1^*\F}_{\mathrm{tor}}}$ is also locally free.
Therefore 
\begin{align} \label{Equ8-11}
  \sigma_2^* \br{\sigma_1^* \F}_{\mathrm{tor}} \supseteq 	\br{\sigma^* \F}_{\mathrm{tor}}.
\end{align}
By \eqref{Equ8-10},  \eqref{Equ8-11}, we have $\sigma_2^* \br{\sigma_1^* \F}_{\mathrm{tor}} =	\br{\sigma^* \F}_{\mathrm{tor}}$ so that  $\sigma^* \F / (\sigma^* \F)_{\mathrm{tor}} = \sigma_2^* \br{\sigma_1^* \F / \br{\sigma_1^*\F}_{\mathrm{tor}}}$ is locally free.
\end{proof}

If $k \in \N^*$, let $D = \sum_{i=1}^k D_i$ be a divisor of $X$ with normal crossing.
 We denote $\mathcal{I}_D$ the sheaf of ideals of $\OO_X$ consisting of germs of holomorphic functions on $X$ which vanish on $D$.
Let $\F$ be a $G$-coherent sheaf on $X$.
If $\mathrm{supp} \F \subset D$, by R\"uckert Nullstellensatz \cite[p.~67]{GR} and by the fact that $X$ is compact,  there exists $n \in \N^*$ such that
\begin{align} \label{Equ8-12}
\mathcal{I}_D^n \F = 0.
\end{align}
If $\mathcal{I}_D \F = 0$, then $\F$ is an $\ox/\mathcal{I}_D$-module, so it is the direct image of a $G$-coherent sheaf on $D$.
Let $\iota_i: D_i \rightarrow X$ be the embedding.
\begin{prop} \label{Prop8-2}
If $\mathrm{supp}\F \subset D $, then there exist $G$-coherent sheaves $\xi_i $ on $ D_i$ such that
\begin{align}
  \sum_{i=1}^{k} \iota_{i, *} \xi_i = \F  \quad \text{in} \; K(X,G).
\end{align}

\end{prop}
\begin{proof}
We prove the proposition by induction on $k \in \N^*$.
Assume $k=1$, by \eqref{Equ8-12}, we have
\begin{align} \label{Equ8-8}
\F = \F/ \mathcal{I}_{D_1} \F  + \mathcal{I}_{D_1} \F/\mathcal{I}_{D_1}^2 \F   + \cdots + \mathcal{I}_{D_1}^{n-1} \F/\mathcal{I}_{D_1}^n \F     \quad \text{in} \; K(X, G).
\end{align}
Since for $1\leq p \leq n$, $\mathcal{I}_{D_1} \br{\mathcal{I}_{D_1}^{p-1} \F/\mathcal{I}_{D_1}^p \F} = 0$, by the consideration below \eqref{Equ8-12},
we see that each of the terms in (\ref{Equ8-8}) are direct images of $G$-coherent sheaves on $D_1$,  which gives the proposition when $k=1$.

Assume $k\geq 2$ and our proposition holds for $k-1$. 

We have  
\begin{align} \label{Equ8-9}
  \F = \F/ \mathcal{I}_{D_k} \F  + \cdots + \mathcal{I}_{D_k}^{n-1} \F/\mathcal{I}_{D_k}^n \F   + \mathcal{I}_{D_k}^n \F  \quad \text{in} \; K(X, G).
 \end{align}
  
For $n$ large enough, $\mathcal{I}_{D_k}^n \F$ is supported in $\sum_{i=1}^{k-1} D_i $. 
By the induction assumption, it is the direct image of $G$-coherent sheaves on the disjoint union $\sqcup_{i=1}^{k-1} D_i$.
As before,  the other terms are direct images of $G$-coherent sheaves on $D_k$.
The induction argument is complete.
 \end{proof}

\subsection{Proof of Theorem \ref{Thm8-1}} \label{Section8-3}
We proceed as in \cite[Theorem 8]{G10} by induction on the dimension of $X$. If $\mathrm{dim}  X = 0$, then $\cgBC = \cgBCBar$ is clear by condition (\ref{Unicity1}) of the theorem.
We assume that  $\mathrm{dim} X \geq 1$. We also assume that for any compact complex manifold $X^\prime$ with $\mathrm{dim} X^\prime \leq \mathrm{dim} X-1$, $\cgBC = \cgBCBar$ on $K \br{X^\prime, G}$.

Let $\F$ be a $G$-equivariant coherent sheaf on $X$.
By Proposition \ref{Prop8-1}, we can find a modification $\sigma: \widetilde{X} \rightarrow X$ which is a finite composition of blow-ups with smooth centers such that $\sigma^* \F / (\sigma^*  \F)_{\mathrm{tor}}$ is locally free.
Moreover, the exceptional divisor of $\sigma$ and the support of $(\sigma^* \F)_{\mathrm{tor}}$ are both contained in a normal crossing divisor $D = \sum_i D_i$ of $\widetilde{X}$.

By \cite[Theorem 9.4.1]{BSW}, $\sigma^*: \HEXBCC \rightarrow H^{(=)}_\mathrm{BC} \br{\widetilde{X}, \C}$ is injective, so we only need to prove 
\begin{align} \label{Equ8-13}
\sigma^* \cgBC \br{ \F} = \sigma^* \cgBCBar \br{ \F}.
\end{align}
By Theorem \ref{ChernPullback} and condition (\ref{Unicity2}) in the theorem, we know that \eqref{Equ8-13} is equivalent to
\begin{align} \label{Equ8-1}
	\cgBC \br{L\sigma^* \F} =  \cgBCBar \br{ L \sigma^* \F}.
\end{align}

By Theorem \ref{DbcohToKXG}, we have 
\begin{align}  \label{Equ8-2}
	\cgBC \br{L\sigma^* \F} &=  \chgBC \br{\sigma^* \F}  + \sum_{i\geq 1} (-1)^i \chgBC \br{L^i \sigma^* \F}. \\
	&= \chgBC \br{\sigma^* \F / \br{\sigma^* \F}_{\mathrm{tor}}} + \chgBC \br{\br{\sigma^* \F}_{\mathrm{tor}}} + \sum_{i\geq 1} (-1)^i \chgBC \br{L^i \sigma^* \F} \nonumber .
\end{align}
Similar identities hold for $\cgBCBar$.

By condition (\ref{Unicity1}) in the theorem and since $\sigma^* \F / \br{\sigma^* \F}_{\mathrm{tor}}$ is locally free, we obatin
\begin{align} \label{Equ8-14}
  \chgBC \br{\sigma^* \F / \br{\sigma^* \F}_{\mathrm{tor}}} = \cgBCBar \br{\sigma^* \F / \br{\sigma^* \F}_{\mathrm{tor}}}.
\end{align}

Since $\br{\sigma^* \F}_{\mathrm{tor}}$ is supported in $D$ and since $\chgBC, \cgBCBar$ satisfy RRG for embeddings, by Proposition \ref{Prop8-2} and induction assumption, we have 
\begin{align}
  \chgBC \br{\br{\sigma^* \F}_{\mathrm{tor}}} &= \cgBCBar \br{\br{\sigma^* \F}_{\mathrm{tor}}}  .
\end{align}

Since $\sigma$ is a biholomorphism on $\widetilde{X} \backslash D$, we know that $L^{\geq 1} \sigma^* \F$  are also supported in $D$.
As before, we get 
\begin{align} \label{Equ8-15}
\begin{aligned}
\chgBC \br{L^{\geq 1} \sigma^* \F} &= \cgBCBar \br{L^{\geq 1} \sigma^* \F} .
\end{aligned}
\end{align} 
By \eqref{Equ8-2}-\eqref{Equ8-15}, the proof of the theorem is complete.
\qed

\section{Submersions and elliptic superconnection forms} \label{SectionSubmersion1}
The purpose of this section is to state our main result in the case of an equivariant projection $p: M = X \times S \to S$, and to define infinite-dimensional equivariant elliptic superconnection forms on $S$ which will be used to establish this main result.

In Section \ref{Section91}, we state the main theorem for $\F \in \DbCohMG$ and prove that we can reduce the problem to the case when $G$ acts trivially on the base manifold $S$.

In Section \ref{Section92}, using the result of Section \ref{SectionEquivalence}, we show that in the proof, $\F$ can be replaced by an antiholomorphic $G$-superconnection $\E = \br{E, \AEpp}$. We will then view $p_* \E$ as equipped with an infinite dimensional antiholomorphic $G$-superconnection $\ApEzeropp$. 

In Section \ref{Section93}, given a splitting $\E \simeq \E_0$, and $G$-invariant Hermitian metrics $g^{TX}$, $g^D$ on $TX$, $D$, we obtain the adjoint $\ApEzerop$ of $\ApEzeropp$.

In Section \ref{ElliCherForm}, given $g\in G$ we construct infinite-dimensional Chern character forms $\chg \br{\ApEzeropp, \omega^X, g^D}$ and establish some basic properties similar to those in Section \ref{EquiCherCharForm}. These forms are called the equivariant elliptic Chern character forms.

In Section \ref{ChernDirectImage}, using the results in \cite[Section 11]{BSW}, we prove that the elliptic Chern character $\chgBC \br{\ApEzeropp}$ coincide with $\chgBC \br{Rp_* \E}$.
\subsection{A Riemann-Roch-Grothendieck theorem for submersions} \label{Section91}
Let $X$, $S$ be compact complex $G$-manifolds of dimension $n$, $n^\prime$.
Put
\begin{align}
M = X \times S.
\end{align}
Then, $M$ is a $G$-manifold.

Let $p: M \rightarrow S, q: M \rightarrow X$ be the projections.
Then,
\begin{align} \label{MDecSX}
TM = q^* TX  \oplus  p^* TS .	
\end{align}
Let $\F$ be an object in $\DbCohMG$. As in Section \ref{DerivedDirect}, $Rp_* \F \in \DbCohSG$.

Let $g\in G$, let $M_g \subset M, X_g \subset X$ and $S_g \subset S$ be the fixed points. Then 
\begin{align}
  M_g = X_g \times S_g.
\end{align}
Let $p_g: M_g \rightarrow S_g, q_g: M_g \rightarrow X_g$ be the corresponding map on the fixed points. 

The purpose of the remaining parts of the paper is to prove the following theorem.
\begin{thm} \label{Thm9-1}
The following identity holds for $g\in G$,
\begin{align} \label{Equ9-5}
\chgBC \br{Rp_*\F} = p_{g, *} \brr{q_g^* \tdgBC \br{TX} \chgBC \br{\F}}	 \quad \text{in} \;\HESgBCC.
\end{align}
\end{thm}
\begin{proof} [First part of the proof]
We have the following commutative diagram with obvious embeddings $j, j^\prime$ and projections $p, p^\prime$,
\begin{equation} \label{Equ9-12}
\begin{tikzcd}
X\times S_g \arrow[r, "j^\prime"] \arrow[d, "p^\prime"] & M \arrow[d, "p"] \\
S_g \arrow[r, "j"]                                      & S .              
\end{tikzcd}	
\end{equation}
Let us reduce the proof of our theorem to the projection $p^\prime$.
Using the base change theorem \cite[Proposition 3.1.0]{SGA6} to \eqref{Equ9-12}, we have the canonical isomorphism,
\begin{align} \label{Equ9-3}
 Lj^* Rp_* \F \simeq  Rp^{\prime}_* Lj^{\prime*} \F \quad \text{in} \;  \DbCohSgG. 
\end{align}
Applying $\chgBC$ to both sides of \eqref{Equ9-3}, we obtain
\begin{align} \label{Equ9-6}
  \chgBC \br{ Lj^* Rp_* \F } = \chgBC \br{ Rp^{\prime}_* Lj^{\prime*} \F } \quad \text{in} \; H^{(=)}_{\mathrm{BC}}(S_g, \mathbf{C}).
\end{align}

By Theorem \ref{ThmDerivedChern}, we have 
\begin{align}  \label{Equ9-13}
  \chgBC \br{ Lj^* Rp_* \F } = \chgBC \br{ Rp_* \F } \quad \text{in} \; H^{(=)}_{\mathrm{BC}}(S_g, \mathbf{C}).
\end{align}
By \eqref{Equ9-6}, \eqref{Equ9-13}, we get 
\begin{align}  \label{Equ9-14}
  \chgBC \br{ Rp_* \F }=   \chgBC \br{ Rp^\prime_* Lj^{\prime *} \F } \quad \text{in} \; H^{(=)}_{\mathrm{BC}}(S_g, \mathbf{C}).
\end{align}
Assume at the moment that the RRG theorem holds for $p^\prime: X \times S_g \rightarrow S_g$ and the object $Lj^{\prime *} \F$ in $\mathrm{D}^{\mathrm{b}}_{\mathrm{coh}} \br{X \times S_g, \anbr{g}}$, then
\begin{align} \label{Equ9-7}
	\chgBC \br{Rp^\prime_* Lj^{\prime *}\F} = p_{g,*} \brr{q_g^{*} \tdgBC \br{TX} \chgBC \br{Lj^{\prime *} \F}}	 \quad \text{in} \;\HESgBCC.
\end{align}

As in \eqref{Equ9-13}, we have 
\begin{align} \label{Equ9-2}
\chgBC \br{L j^{\prime *} \F} =  \chgBC \br{\F}  \quad \text{in} \; H^{(=)}_{\mathrm{BC}}(M_g, \mathbf{C}).	
\end{align}

By \eqref{Equ9-2}, we can rewrite \eqref{Equ9-7} as
\begin{align} \label{Equ9-11}
 \chgBC \br{Rp^\prime_* Lj^{\prime *}\F} =  p_{g, *} \brr{q_g^* \tdgBC \br{TX} \chgBC \br{\F}} \quad \text{in} \;\HESgBCC.
\end{align}
By \eqref{Equ9-14}, \eqref{Equ9-11}, we get \eqref{Equ9-5}.

In the rest parts of the paper, we will prove \eqref{Equ9-7} or the case when $G$ acts trivially on $S$.
\end{proof}
In the sequel, we will assume that $G$ acts trivially on $S$.

\subsection{Replacing $\F$ by $\E$} \label{Section92}
Let $\F \in \DbCohMG$. By Theorem \ref{EssSurB} and Definition \ref{ChernDbCoh}, we may replace $\F$ by some $\br{E, \AEpp} \in \mathrm{B}(M,G)$, and we also denote $\E = \br{\om^\infty \br{E}, \AEpp}$.

We will use the notations of Sections \ref{EquiAntiSupe} and \ref{EquiGeneMetr}.

Now $p_* \E$ is an $\OO_S$-module.
It defines by the infinite-dimensional vector bundle $C^\infty \br{X, E|_X}$.
It is a $\lsb$-module and $\AEpp$ induces an antiholomorphic $G$-superconnection $\ApEpp$ on $p_* \E$. We will not distinguish the sheaf $p_* \E$ and the vector bundle.


By \cite[Theorem 9.5]{Bor87} and by \cite[Proposition IV.4.14]{demailly1997complex}, and since $\E$ is a bounded complex of soft $\om$-modules, we have
\begin{align}
Rp_* \E = p_* \E.	
\end{align}
Therefore the $G$-equivariant $\OO_S$-complex $\br{p_* \E, \ApEpp}$ defines an object in $\DbCohSG$.

Let $i$ be the embedding of the fibers $X$ into $M$.
Then the pullback $\br{i^*_b E, A^{i_b^* E \prime\prime}}$ is a family of antiholomorphic $G$-superconnections along the fibers $X$, whose corresponding associated element in $\DbCohXG$ is just $Li^* \E$.

Let $\DD$ be the diagonal bundle associated with $p_* \E$. Then
\begin{align} \label{InfiDiag}
\DD = p_* i_b^* \E,	
\end{align}
and the corresponding $A^{\DD\prime\prime}$ (which is an analogue of $v_0$ for $p_* \E$) is given by 
\begin{align}
A^{\DD\prime\prime}	= A^{i_b^* E \prime\prime}.
\end{align}

By (\ref{MDecSX}), we have,
\begin{align} \label{MDecSX2}
\lmb = p^* \lsb \widehat{\otimes} q^* \lxb.	
\end{align}
By (\ref{DefE0}) and (\ref{MDecSX2}), we obtain,
\begin{align} \label{E0M}
E_0 = p^* \lsb \widehat{\otimes} q^* \lxb \widehat{\otimes} D.	
\end{align}
Note that 
\begin{align} \label{E0X}
i^*_b E_0 = q^* \lxb \widehat{\otimes} D.	
\end{align}

Over $X \times S$, we fix splitting as in (\ref{EIsoE0}) and (\ref{EIsoE02}), so that 
\begin{align}
E \simeq E_0.
\end{align}
These splittings induce corresponding splittings of $i_b^* E$, and we have corresponding isomorphism,
\begin{align} \label{ibEIsoibE0}
i_b^* E \simeq	i_b^* E_0.
\end{align}
By (\ref{ibEIsoibE0}), we get the non-canonical isomorphism
\begin{align} \label{Equ9-50}
\DD = \Omega^{0,\bullet} \br{X, D|_X}.	
\end{align}

Put 
\begin{align}
p_* \E_0 = \lsb \widehat{\otimes} p_* i^*_b \E.
\end{align}

By (\ref{InfiDiag}), (\ref{E0M}) and (\ref{E0X}), we have the identification,
\begin{align} \label{pEIsopE0}
p_* \E \simeq p_* \E_0.	
\end{align}
To keep in line with the previous notation, we will denote $\ApEzeropp$ the operator correpsonding to $\ApEpp$ via the isomorphism (\ref{pEIsopE0}).

\subsection{The adjoint of $\ApEzeropp$} \label{Section93}

Let $g^{TX}$ be a $G$-invariant Hermitian metric on $TX$, and let $\omega^X$ denote the corresponding fundamental (1,1)-form along the fibres $X$.
Let $g^{T_{\R}X}$, $g^{T^*_{\R}X}$ be the induced metrics on $T_{\R}X$, $T^*_{\R}X$.
We denote by $\langle  \;\rangle$ the corresponding scalar product on $T_{\R}X$.
If $J^{T_\R X}$ denote the complex structure of $T_\R X$, then
\begin{align} \label{omega}
\omega^X \br{\cdot,\cdot} = \anbr{\cdot, J^{T_\R X} \cdot}.	
\end{align}
Let $g^{\lxb}$ denote the metric induced by $g^{TX}$ on $\lxb$. 
Let $dx$ be the volume form on $X$ which is induced by $g^{TX}$, then 
\begin{align}
  dx = \frac{(- \omega^X)^n}{n!}.
\end{align}

Let $g^D$ be a $G$-invariant Hermitian metric on the $\Z$-graded vector bundle $D$.

We equip $\lxb \widehat{\otimes} D$ with the metric $g^{\lxb} \widehat{\otimes}g^D$.

\begin{defn}
If $s$, $s^\prime \in \DD = \Omega^{0,\bullet} \br{X, D|_X}$, put
\begin{align}
\alpha \br{s, s^\prime} = \br{2 \pi}^{-n} \int_X \anbr{s, s^\prime}_{g^{\lxb} \otimes g^D} dx.	
\end{align}
	
\end{defn}
Then $\alpha$ is a Hermitian product on $\DD$, that defines an invariant pure metric $g^\DD$ in $\McalD$.

Here we use the notation of Section \ref{EquiGeneMetr} where the adjoint of an antiholomorphic $G$-superconnection with respect to a $G$-invariant Hermitian metric was defined.
\begin{defn}
Let $\ApEzerop$	denote the adjoint of $\ApEzeropp$ with respect to $g^\DD$.
\end{defn}
\begin{defn}
Put
\begin{align}
\ApEzero = \ApEzeropp + \ApEzerop.
\end{align}
\end{defn}
Since $\omega^X$, $g^D$ are $G$-invariant, $\ApEzero$ is a $G$-equivariant superconnection on $p_* \E$.
\subsection{The elliptic superconnection forms} \label{ElliCherForm}
By \cite[(10.6.2)]{BSW}, the operator $\ApEzerosq$ is a fiberwise elliptic operator, so $\exp \br{- \ApEzerosq}$ exists and is a fiberwise trace class operator.

We will now imitate the construction of Section \ref{EquiCherCharForm} to an infinite-dimensional context.
\begin{defn} \label{Defn9-1}
Put
\begin{align}
\chg \br{\ApEzeropp, \omega^X, g^D} = \varphi \Trs \brr{g \exp \br{-\ApEzerosq}}.	
\end{align}
	
\end{defn}

Then $\chg \br{\ApEzeropp, \omega^X, g^D}$ is a smooth even form on $S$. 

Let $\mathscr{P}_G$ be the collection of $G$-invariant parameters $\omega^X$, $g^D$.
Let $\mathbf{d}^{\mathscr{P}_G}$ be the de Rham operator on $\mathscr{P}_G$.

The metric $\alpha$ depends smoothly on $\omega^X$, $g^D$.
Then $\alpha^{-1} \mathbf{d}^{\mathscr{P}_G} \alpha$ is a 1-form with values in $\alpha$-self-adjoint endomorphisms of $\DD$.

Now we have an obvious analogue of Theorem \ref{ThmChern-Equ1}.
\begin{thm}
The form $\chg \br{\ApEzeropp, \omega^X, g^D}$ lies in $\Omega^{(=)} \br{S, \C}$, it is closed, and its Bott-Chern cohomology class does not depend on $\omega^X$, $g^D$ or on the splitting of $E$.
Also $\varphi \Trs \brr{ \br{\alpha^{-1} \mathbf{d}^{\mathscr{P}_G} \alpha} g \exp \br{-\ApEzerosq}}$ is a 1-form on $\mathscr{P}_G$ with values in $\Omega^{(=)}\br{S, \C}$, and moreover,
\begin{align}
\mathbf{d}^{\mathscr{P}_G} \chg \br{\ApEzeropp, \omega^X, g^D} = - \frac{\dbs \ds}{2 i \pi} \varphi \Trs \brr{ \br{\alpha^{-1} \mathbf{d}^{\mathscr{P}_G} \alpha } g \exp\br{-\ApEzerosq}}.	
\end{align}
	
\end{thm}
\begin{proof}
The proof is algebraically the same as in Theorem \ref{ThmChern-Equ1} while the analytical details follow the same steps as in the proof of \cite[Theorem 10.7.2]{BSW}.	
\end{proof}

\begin{defn}
Let $\chgBC \br{\ApEpp} \in \HESBCC$	 be the common Bott-Chern cohomology of the forms $\chgBC \br{\ApEzeropp, \omega^X, g^D}$.
\end{defn}

\subsection{The Chern character of the direct image} \label{ChernDirectImage}
\begin{thm} \label{Thm10-1-new}
There exist a finite-dimensional antiholomorphic $G$-superconnection $\br{\underline{E}, \AUEpp}$ on $S$ with associated sheaf of complex $\underline{\E}$,
and a morphism of $G$-equivariant $\OO_S^\infty \br{\lsb}$-modules $\phi: \underline{\E} \rightarrow p_* \E$ which is a quasi-isomorphism of $\OO_S$-complexses.
Furthermore, for any $s \in S$, $\phi_s: \br{\underline{D}, \underline{v}_0}_s \rightarrow \br{\DD, A^{\DD\prime\prime}}_s$ is a quasi-isomorphism.

Also, we have an isomorphism,
\begin{align} \label{InfEquFin}
\underline{\E} \simeq Rp_* \E 	\quad \mathrm{in} \; \DbCohSG.
\end{align}
And, the following identity holds,
\begin{align} \label{InfEquFin2}
\chgBC \br{\ApEpp} = \chgBC \br{\underline{\E}} = \chgBC \br{Rp_* \E} \quad \mathrm{in} \; \HESBCC.	
\end{align}
\end{thm} \label{ThmInfAndFin}
\begin{proof}
By Proposition \ref{PropCompatibleDirectImage}, there exists an object $\underline{\E}$ in $\mathrm{B} \br{S, G}$ and a morphism of $G$-equivariant $\OO_S^\infty \br{\lsb}$-modules $\phi: \underline{\E} \rightarrow p_* \E_0$, which is also a quasi-isomorphism of $\OO_S$-complexes, so (\ref{InfEquFin}) holds.

The fact that $\phi_s$ is a quasi-isomorphism is independent of the action of $G$ and has already been proved in \cite[Theorem 11.2.1]{BSW}.

Using the above results, and proceeding as in  \cite[Theorem 11.3.1]{BSW}, the first identity of \eqref{InfEquFin2} holds. The second identity of \eqref{InfEquFin2} follows from \eqref{InfEquFin}.
\end{proof}

\section{The hypoelliptic superconnection and its exotic version} \label{SectionHypo1}
The purpose of this section is to construct an antiholomorphic $G$-superconnection over $S$ with fiberwise hypoelliptic curvature.
This superconnection is a deformation of the elliptic superconnection of Section \ref{SectionSubmersion1}. Our constructions are extensions of the corresponding constructions in \cite[Sections 13, 16]{BSW} to the equivariant case.

In Section \ref{Section112-new}, we introduce $\X$, the total space of an extra copy $\wTX$ of $TX$ and an antiholomorphic $G$-superconnection $K_\X$ on $\X$. The direct image  $\pi_* K_\X$ is an infinite dimensional antiholomorphic $G$-superconnection on $X$.

In Section \ref{Section12-2-0}, we extend the constructions in Section \ref{Section112-new} to a family over $S$ and get an infinite dimensional antiholomorphic $G$-superconnection $\br{E_S, \AESpp}$.

In Section \ref{SplittingES}, given a $G$-invariant metric $\gwTX$ on $\wTX$, we obtain a corresponding splitting of $E_S \simeq E_{S,0}$. Therefore $\AESpp$ induces a corresponding infinite dimensional antiholomorphic $G$-superconnection $\AAAA^{\prime\prime}_{Y}$ on $E_{S,0}$.

In Section \ref{Section113}, given metrics $g^{TX}$, $\gwTX$ and $g^D$ on $TX$, $\wTX$ and $D$, we construct an exotic version of nonpositive $G$-invariant Hermitian form on the diagonal bundle of $E_S$. This Hermitian form was introduced in \cite{B13} and \cite{BSW} and depends on a parameter $\theta = \br{c, d}$. Using this Hermitian form, we can construct the adjoint $\AAAA^\prime_{Y, \theta}$ of $\AAAA^{\prime\prime}_Y$.

In Section \ref{SectionCurvature}, we study the curvature $\AAAA^2_{Y, \theta}$ which was given in \cite{BSW}. The curvature is a fiberwise hypoelliptic operator.

In Section \ref{HypoSuperConnForm}, we construct the equivariant hypoelliptic superconnection forms, that depend on $G$-invariant metrics $g^{TX}$, $\gwTX$, $g^D$ and the parameters $\theta$. As before, we show that the Bott-Chern cohomology class is independent of these datas.

In Section \ref{Deformationb}, we show that when $c=0$ and replacing $\gwTX$ by $b^4 \gwTX$, as $b \to 0$, the equivariant hypoelliptic superconnections converge to the equivariant elliptic superconnection forms.

\subsection{The total space of $TX$} \label{Section112-new}

Let $X$ be a compact complex $G$-manifold.
Let $\pi: \X \rightarrow X$ be the total space of $TX$, the fiber being denoted $\wTX$ to distinguish it from the usual tangent bundle $TX$, so that $TX$ and $\wTX$ are canonically isomorphic.
The conjugate bundle to $\wTX$ will be denoted $\wTXb$ and its dual $\wTXbs$.
The real bundle associated with $\wTX$ is denoted by $\wTRX$.
Since $G$ acts on $X$, all the objects above carry natural $G$-actions.
In particular, $\X$ is a $G$-manifold.

We have the exact sequences of holomorphic $G$-vector bundles on $\X$,
\begin{align}
\begin{aligned} 
0 \rightarrow \pi^* \wTX \rightarrow T \X \rightarrow \pi^* TX \rightarrow	0, \\
0 \rightarrow  \pi^* T^* X \rightarrow T^* \X \rightarrow \pi^* \wTXs \rightarrow 0.
\end{aligned}
\end{align}

Let $y \in C^\infty \br{\X, \pi^* TX}$ be the tautological section which is holomorphic. 
The  holomorphic map
\begin{align}
  i_y: \pi^* \Lambda^{\bullet } \br{T^*X} \rightarrow \pi^* \Lambda^{\bullet - 1 } \br{T^*X}
\end{align}
is $G$-equivariant so that it defines a complex of $G$-equivariant holomorphic vector bundles $  \br{\pi^* \Lambda \br{T^*X}, i_y}$.
We endow $\lx$ with the degree $\mathrm{deg}_{-}$ defined in \eqref{Equ6-30} so that the degree of $ \Lambda^k \br{T^*X}$ is $-k$  and $i_y$ is of degree $1$.

Put 
\begin{align} \label{Equ11-5-new}
K_\X =  \Lambda \br{\overline{T^* \X}} \wo \pi^* \Lambda \br{T^* X}.
\end{align}
Then  $ \br{K_\X, \db^\X + i_y} $ is an antiholomorphic $G$-superconnection on $\X$.
 
Let $\sigma$ be the involution of $\X$ given by $\br{x, \wy} \rightarrow \br{x, - \wy}$. 
The action of $\sigma$ on $\X$ commutes with the action of $G$, therefore $\X$ carries a $G \times \Z_2$ action.
We make $\sigma$ acts trivially on $X$ so that $\pi$ is $G\times \Z_2$-equivariant.

The actions of $\sigma$ on $\X$ and $X$ induce the corresponding actions on $ \Lambda \br{\overline{T^* \X}}$ and $ \Lambda \br{T^* X}$, so that $\sigma$ acts on $K_\X$. Since $i_y$ is of degree $1$ and since $\sigma$ changes $y$ to $-y$, $i_y$ commutes with $\sigma$. Therefore, $\br{K_\X, \db^\X + i_y}$ is an antiholomorphic $G\times \Z_2$-superconnection.

 Put
\begin{align}
&\textbf{I} = C^\infty \br{\wTX, \pi^*\lwxbs }.
\end{align}
Then $\II$ is an infinite-dimensional vector bundle on $X$. Let $\underline{\sigma}^*$ denote the induced action of $\sigma$ on $\II$.
Namely, if $N^{\lwxbs}$ is the number operator of $\lwxbs$, and if $s\in \II$,
then 
\begin{align} \label{Equ11-7-new}
\underline{\sigma}^* s (x, \widehat{y}) = (-1)^{N^{\lwxbs}} s (x, - \widehat{y}).	
\end{align}

We denote by $\pi_* K_\X$ the infinite-dimensional vector bundle $C^\infty \br{\wTX, K_\X|_{\wTX}}$ on $X$.
As in Section \ref{Section92}, the antiholomorphic $G \times \Z_2$-superconnection  $ \br{K_\X, \db^\X + i_y} $ induces an antiholomorphic $G \times \Z_2$-superconnection $\br{\pi_* K_\X, A^{\pi_* K_\X\prime\prime}}$  on $X$ with associated diagonal bundle $\lx \wo \II$.

 
\subsection{An antiholomorphic $G \times \Z_2$-superconnection on $S$} \label{Section12-2-0}
Let $S$ be a compact complex manifold. We make the same assumptions as in Sections \ref{Section91}, \ref{Section92}. 
Now we extend the constructions in  Section \ref{Section112-new} to a family over $S$.

Put
\begin{align}
	\M = S\times \X.
\end{align}
Consider the obvious projections as indicated in the following diagram.
\begin{equation}
\begin{tikzcd}
\mathcal{X} \arrow[r, "\pi"]                                                                    & X                                \\
\mathcal{M} \arrow[u,"\underline{q}"] \arrow[r, "\underline{\pi}"] \arrow[rd, "\underline{p}"'] & M \arrow[u, "q"'] \arrow[d, "p"] \\
                                                            & S .                              
\end{tikzcd}	
\end{equation}
We make $\sigma$ acting trivially on $S$ so that all the maps in the diagram above are $G\times \Z_2$-equivariant.

Let $\br{E, \AEpp}$ be an antiholomorphic $G$-superconenction on $M$ with associated diagonal bundle $D$.
We make $\sigma$ acts trivially on $E$ so that $\br{E, \AEpp}$ is an antiholomorphic $G\times\Z_2$-superconnection on $M$.

Put
\begin{align}
&\br{E_\M, A^{E_\M \prime\prime}} = \br{\upi^*_b E \wo_b \underline{q}^{ *}_b K_\X, A^{\upi^*_b E \wo_b \underline{q}^{ *}_b K_\X}}.
\end{align}
Then $E_\M$ is an antiholomorphic $G\times\Z_2$-superconnection on $\M$  with associated diagonal bundle $\upi^* D \wo (\pi \uq )^* \Lambda \br{T^*X}$.

Put 
\begin{align}
	E_S = \up_* E_\M.
\end{align}
Then $E_S$ is an infinite-dimensional vector bundle on $S$.
The antiholomorphic $G \times \Z_2$-superconnection $\br{E_\M, A^{E_\M \prime\prime}}$ on $\M$ induces an  antiholomorphic $G\times\Z_2$-superconnection  $\br{E_S, \AESpp}$  on $S$ with associated diagonal bundle $\Omega \br{X, D|_X \wo \II}$.

\subsection{A splitting of  $E_S$} \label{SplittingES}
Let $g^{\wTX}$ be a $G$-invariant Hermitian metric on $\wTX$, let $\nabla^{\wTX}$ be the Chern connection on $\br{\wTX, g^{\wTX}}$ which is $G$-invariant by construction. Let $R^{\wTX}$ be the corresponding curvature.

Let $T^H \X \subset T\X$ be the smooth horizontal $G$-subbundle associated with the connection $\nabla^{\wTX}$.
The bundle $T^H \X$ is also preserved by $\sigma$ and we have the identification of smooth $G\times\Z_2$-equivariant vector bundles on $\X$,
\begin{align} \label{Iden1}
T^H \X = \pi^* TX,	
\end{align}
and also identification of $G\times\Z_2$-equivariant smooth vector bundles on $\X$,
\begin{align} \label{Iden2}
T\X = T^H \X \oplus \pi^* \wTX.	
\end{align}
By (\ref{Iden1}), (\ref{Iden2}), we get the $G \times \Z_2$-equivariant smooth identifications,
\begin{align} \label{Iden3}
&T \X = \pi^* \br{TX \oplus \wTX}, &&\lxxbs = \pi^* \br{\lxb \widehat{\otimes} \lwxbs}.
\end{align}
The identifications \eqref{Iden3} give an splitting,
\begin{align}
\pi_* K_\X \simeq \lxc \wo \II.	
\end{align}

We fix a splitting of $E$ as in (\ref{EIsoE0}), (\ref{EIsoE02}), so that 
\begin{align} \label{Equ11-1}
  E\simeq E_0.
\end{align}

Set 
\begin{align}
 E_{S,0} = \lsb \wo \Omega \br{X, D|_X \wo \II}.	
\end{align}
Then the above two splittings give the identification,
\begin{align} \label{Equ10-3}
E_S \simeq E_{S, 0}.
\end{align}

\begin{defn}
Let $\AAppy$ be the antiholomorphic $G \times \Z_2$-superconnections on $E_{S,0}$ that corresponds to $\AESpp$.	
\end{defn}


\subsection{A $G \times \Z_2$-superconnection on $S$} \label{Section113}
Let $\textbf{I}^c$ denote the vector bundle of elements of $\textbf{I}$ which have compact support.
Then $\textbf{I}^c$ inherits from  $g^{\wTX}$ an $L_2$-Hermitian metric $g^{\textbf{I}^c}$, such that if $r,  r' \in \textbf{I}^c$, then
\begin{align}
\anbr{r, r'}_{g^{\textbf{I}^c}} = \int_{\wTRX} \anbr{r, r'}_{g^{\lwxbs}} \frac{dY}{(2\pi)^n}.	
\end{align}
The metric $g^{\II^c}$ is $G \times \Z_2$-invariant.

Let $g^D$ be a $G$-invariant Hermitian metric on $D$.
Therefore, $D|_X \wo  \II^c$ is equipped with a $G\times \Z_2$-invariant Hermitian metric $g^{D|_X \wo \II^c}$. 

Recall that $\widetilde{\phantom{x}}: \Lambda \br{T^*_\C X} \to  \Lambda \br{T^*_\C X}$ was defined in \eqref{Equ6-1}.

\begin{defn}
	If $s$, $s^\prime \in \Omega \br{X, D|_X \wo \II^c}$, such that
\begin{align}
	&  s = \sum \alpha_i r_i, &&  s^\prime = \sum \alpha'_j r'_j,
\end{align}
with $\alpha_i$, $\alpha_j^\prime \in \Omega \br{X, \C}$, and $r_i$, $r_j^\prime \in C^\infty \br{X, D|_X \wo \II^c}$, put 
\begin{align}
\delta_{X} \br{s, s^\prime}	= \frac{i^n}{(2\pi)^n} \sum  \int_X \widetilde{\alpha_i} \wedge \overline{\alpha'_j} \anbr{r_i, r'_j}_{g^{D|_X \wo \II^c}}.	
\end{align}

\end{defn}
Then $\delta_{X}$ is a $G \times \Z_2$-invariant non-degenerate Hermitian form on $\Omega \br{X, D|_X \wo \II^c}$.

Let $g^{TX}$ be a $G$-invariant Hermitian metric on $TX$, and let $\omega^X$ be the corresponding fundamental $(1,1)$-form that was defined in \eqref{omega}.

\begin{defn}
If $\theta = (c,d) \in [0,1] \times ]0, + \infty[$, put
\begin{align}
\omega^X_\theta = \br{1-c+\frac{cd}{2} |Y|^2_{\gwTX}} \omega^X.	
\end{align}
\end{defn}

Then $\omega^X_\theta$ is $G \times \Z_2 $-invariant.


\begin{defn} \label{Def10-20}
If $s$, $s^\prime \in \Omega \br{X, D|_X \wo \II^c}$, we define
\begin{align} \label{Equ10-12}
\epsilon_{X, \theta} \br{s, s^\prime} = \delta_X \br{\underline{\sigma}^* s, e^{-i\omega_\theta^X} s^\prime}.	
\end{align}

\end{defn}
Then $\epsilon_{X, \theta}$ is a $G\times \Z_2$-invariant non-degenerate Hermitian form on $\Omega \br{X, D|_X \wo \II^c}$.


In Section \ref{EquiGeneMetr}, we can extend $\epsilon_{X, \theta}$ to a $G\times \Z_2$-invariant pure \footnote{The notion of a pure Hermitian form as used here aligns with the concept of a pure metric introduced in Definition \ref{Def6-2}.} Hermitian form on $\lsc \wo \Omega \br{X, D|_X \wo \II^c}$ as in \eqref{DefTheta} which we still denote it as $\epsilon_{X, \theta}$.

\begin{defn}
	Let $\AAAA^\prime_{Y, \theta}$ be the adjoint of $\AAppy$ with respect to $\epsilon_{X, \theta}$ as in Definition \ref{DefAdjoint} \footnote{We should extend $\AAppy$ to an operator acting on $\Omega\br{S, \Omega \br{X, D|_X \wo \II^c}}$ as in Section \ref{EquiGeneMetr}.}  .
	Let 
	\begin{align}
	\AAAA_{Y, \theta} = \AAppy + \AAAA^\prime_{Y, \theta}.
	\end{align}
\end{defn}
Since $\AAAA''_Y$ is $G\times \Z_2$-equivariant and since $\epsilon_{X, \theta}$ is $G\times \Z_2$-invariant, $\AAAA_{Y, \theta}$ is $G\times \Z_2$-equivariant. 
Note that $\AAAA'_{Y, \theta}, \AAAA_{Y, \theta}$ depend on $\br{g^{TX}, g^D, \gwTX, \theta }$ and the splitting \eqref{Equ11-1}.
When $c=0$ in $\theta$, $\omega^X_\theta$ is independent of $d$ and  we will use the notation $\AAAA'_Y, \AAAA_Y$ for  $\AAAA'_{Y, \theta}, \AAAA_{Y, \theta}$.

\subsection{The curvature $\AAAA_{Y,\theta}^2$} \label{SectionCurvature}

When identifying $TX$ and $\wTX$, we denote $\wgTX$ the metric on $TX$ corresponding to $\gwTX$.
Let $\wnTX$ be the Chern connection on $\br{TX, \wgTX}$.
Let $\wnlTCX_U$ be the induced connection on $\lxc$.
Let $\widehat{\tau}$ be the torsion of $\wnTX$ \cite[Definition 1.21]{BGV}.

Using the splitting \eqref{Equ11-1}, we can endow $D$ with a $G$-invariant antiholomorphic connection $\nabla^{D\prime\prime}$. 
We can take $\nabla^{D\prime}$ to be the adjoint of $\nabla^{D\prime\prime}$ with respect to $g^D$ so that we can endow $D$ with a $G$-invariant connection $\nabla^D = \nabla^{D\prime\prime}+ \nabla^{D\prime}$.


Put
\begin{align}
  \mathbf{F} = q^* \lxc \wo D \wo q^* \II.
\end{align}
Then $\textbf{F}$ is an infinite-dimensional vector bundle on $M$ such that 
\begin{align} \label{Equ11-2-new}
  C^\infty \br{M, \textbf{F}} = C^\infty \br{S, \Omega \br{X, D|_X \wo \II}}.
\end{align}

Let us introduce two connections on $\textbf{F}$ along the fiber $X$.

Let $\nabla^{\lwxbs}$ be the connection on $\lwxbs$ induced by $\nabla^{\widehat{TX}}$.
Using $\nabla^{\lwxbs}$, we can construct a connection on $\II$ as follows.
If $U \in T_\R X$, let $U^H \in T^H_\R \X$ denote the horizontal lift of $U$ associated with the connection $\nwTX$.
If $s$ is a smooth section of $\textbf{I}$, set
\begin{align}
\nabla^{\textbf{I}}_U s = \nabla^{\lwxbs}_{U^H} s.	
\end{align}

Let $\wnFF_U$ be the derivative on $\mathbf{F}$ associated with $\wnlTCX_U$, $\nabla^D_U$, $\nabla^{\II}_U$.


Put 
\begin{align}
\gamma = \nabla^{TX} - \wnTX.	
\end{align}
Then $\gamma$ is a section of $T^*X \otimes \End \br{TX}$.
If $u \in TX$, then $\gamma u$ is a section of $T^* X \otimes TX$.
If $e \in TX$, we will use the notation $e_* \in \overline{T^* X}$ corresponds to $e$ via the metric $g^{TX}$.

\begin{defn}
If $U \in T_\R X$, such that $U = u+\overline{u}, u \in TX$, put
\begin{align}
\wnFFone_U =  \wnFF_U +  i_{\widehat{\tau} \br{U,\cdot}}	 - i_u q^* \dbx i \omega^X + 	\overline{\br{\gamma u}_*}.
\end{align}
	
\end{defn}
By \eqref{Equ11-2-new}, $\wnFFone_U$ acts on $ C^\infty \br{S, \Omega \br{X, D|_X \wo \II}}$.


Let $\nabla^V$ be the differential along the fiber $\wTX$.
Let $\Delta^V_{\gwTX}$ be the Laplacian along the fibers $\wTX$ with respect to the metric $\gwTX$.

Let $\wy$ be the tautological section of $\pi^* \wTX$ on $\X$, let $\wyb$ be the conjugate section of $\pi^* \wTXb$.
Let $w_1, \cdots, w_n$ be a basis of $TX$.
The corresponding basis of $\wTX$ is denoted $\widehat{w}_1, \cdots, \widehat{w}_n$.
In the sequel, we assume that $\widehat{w}_1, \cdots, \widehat{w}_n$ is an orthonormal basis of $\br{\wTX, \gwTX}$.

Recall that with respect to the identification \eqref{Equ11-1}, a $G$-invariant smooth section $B$ of $  \lmb \widehat{\otimes} \End(D) $ was defined in \eqref{DefB}.
Similarly to \eqref{DefBstar}, we can define $B^*$ from $B$ associated to the pure metric $g^D$. Set 
\begin{align}
  C = B + B^*.
\end{align}

\begin{thm} \label{Thm13-20}
We have the identity of operator acting on $\Omega\br{S, \Omega \br{X, D|_X \wo \textnormal{\textbf{I}}}}$,
\begin{align}  \label{Equ13-2}
\begin{aligned}
 \AAAA_{Y,\theta}^2  = & \frac{1}{2} \br{-\Delta^V_{\gwTX} + |Y|^2_{g^{TX}}} + \wb{w}^i \wedge \br{i_{\overline{w}_i} + \overline{w}_{i*} } - i_{\wb{w}_i} i_{w_i} +  \wnFFone_Y + i_Y C \\
 &- q^* \dbx \dx i \omega^X  - \nabla^V_{R^{\wTX} \wY} - \anbr{R^{\wTX} \wb{w}_i, \wb{w}^j} \wb{w}^i i_{\wb{w}_j} + A^{E_0,2}   \\
&+c\left(\frac{d}{2}|Y|_{g^{\widehat{T X}}}^2-1\right)\left(\frac{1}{2}|Y|_{g^{T X}}^2-i_y q^* \partial^X i \omega^X\right)+c d y_* i_{\overline{\widehat{y}}} \\
&+c\left(\frac{d}{2}|Y|_{g^{\widehat{T X}}}^2-1\right) \left(\overline{\br{\gamma y}_*}-q^* \bar{\partial}^X \partial^X i \omega^X\right)+c d\left(q^* \bar{\partial}^X i \omega^X\right) i_{\overline{\widehat{y}}} \\
&+c\left(d \widehat{y}_*\left(\bar{y}_*-q^* \partial^X i \omega^X\right)+\left(\frac{d}{2}|Y|_{\gwTX}^2-1\right) \overline{\widehat{w}}^i \overline{w}_{i*}\right) \\
&+c d q^* i \omega^X\left(\nabla^{V}_{\overline{\widehat{y}}}+\overline{\widehat{w}}^i i_{\overline{\widehat{w}}_i}\right) .
\end{aligned}
\end{align}
\end{thm}

\begin{proof}
The identity has already been given in \cite[(13.11.8), (13.11.9) and (16.3.1)]{BSW}.	
\end{proof}

For any $s\in S$, $\AAAA_{Y, \theta}^2$ acts on $\Lambda \br{T_s^* S} \wo \Omega \br{X, D|_{X\times \brrr{s}} \wo \II}$.
By  \eqref{Equ13-2} and by H\"ormander \cite{H67}, $\frac{\partial}{\partial t} + \AAAA_{Y, \theta}^2 $ is hypoelliptic  on $\Lambda \br{T_s^* S} \wo \Omega \br{X, D|_{X\times \brrr{s}} \wo \II}$.

\subsection{The hypoelliptic superconnection forms} \label{HypoSuperConnForm}
Let $\II^{\mathcal{S}}$ and $\II^{(2)}$ denote the vector bundle of elements of $\II$ which are rapidly decreasing together with all the derivatives and square integrable respectively. Let $\mathrm{L^2} \br{X, D|_X \wo \II^{(2)}}$ be the space of $\mathrm{L^2}$ sections of $D|_X \wo \II^{(2)}$ on $X$.

We consider $ \AAAA_{Y,\theta}^2$ with domain $\Lambda \br{T_s^* S} \wo \Omega \br{X, D|_{X\times \brrr{s}} \wo \II^{\mathcal{S}}}$, as an unbounded operator on $\Lambda \br{T_s^* S} \wo  \mathrm{L^2} \br{X, D|_{X\times \brrr{s}} \wo \II^{(2)}}$.
We still denote $ \AAAA_{Y,\theta}^2$ as its closure, and will not distinguish them.

By \cite[Theorem 15.7.1]{BL08}, \cite[Section 11.3]{B13}, \cite[Section 16.3]{BSW}, and by the theorem of Hille-Yosida \cite[Section IX-7, p.266]{Y68}, we find that there is a unique well-defined strongly continuous one-parameter semigroup $\exp \br{-t \AAAA_{Y, \theta}^2}|_{t \geq 0}$.

By \cite[Section 3.3]{BL08}, \cite[Section 11.3]{B13} and \cite[Section 16.3]{BSW}, we know that the heat operator $\exp \br{- t \AAAA^2_{Y,\theta}} |_{t>0}$ is fiberwise trace class and has a smooth kernel $\exp \br{- t \AAAA^2_{Y,\theta}} \br{\br{x,y}, \br{x^\prime, y^\prime}}$ with respect to $\frac{dx^\prime dY^\prime}{\br{2\pi}^{2n}}$ which is rapidly decreasing in the variables $y, y^\prime$ together with all the derivatives.

\begin{defn}
Put
\begin{align} \label{Equ11-7-new}
  \chg \br{\AApp_Y, \omega^X, g^D, \gwTX} = \varphi \Trs \brr{g \exp \br{-\AAAA^2_{Y}}} \quad in \; \Omega^{(=)} \br{S, \C} .	
\end{align}
For $\theta  \in [0,1] \times ]0, + \infty[$, put
\begin{align} \label{Equ13-4}
\chg \br{\AApp_Y, \omega^X, g^D, \gwTX, \theta} = \varphi \Trs \brr{g \exp \br{-\AAAA^2_{Y,\theta}}} \quad in \; \Omega^{(=)} \br{S, \C}.	
\end{align}	
The form in \eqref{Equ11-7-new} is called hypoelliptic superconnection form and the form in (\ref{Equ13-4}) is called exotic hypoelliptic superconnection form. 
\end{defn}

\begin{thm} \label{Thm12-1-new}
The form $\chg \br{\AApp_Y, \omega^X, g^D, \gwTX, \theta}$ is closed, and its Bott-Chern cohomology class does not depend on $\omega^X$, $g^D$, $\gwTX$, $\theta$, or on the splitting of $E$.
In particular, the class of the forms $\chg \br{\AApp_Y, \omega^X, g^D, \gwTX,\theta}$ and $ \chg \br{\AApp_Y, \omega^X, g^D, \gwTX}$ coincide in Bott-Chern cohomology.

\end{thm}
\begin{proof}
The proof is essentially the same as in \cite[Theorem 16.3.1]{BSW}.	
\end{proof}

\begin{defn}
Let $\chgBC \br{\AAAA_Y^{\prime\prime}} \in \HESBCC$ denote the common class  of the forms $ \chg \br{\AApp_Y, \omega^X, g^D, \gwTX, \theta}$.	
\end{defn}

\subsection{The hypoelliptic and elliptic superconnection forms} \label{Deformationb}
In this section, we take $\gwTX = b^4 g^{TX}$.

Recall that the elliptic superconnection forms $\chg \br{\ApEzeropp, \omega^X, g^D}$ were defined in Definition \ref{Defn9-1}.
The following theorem is an extension of \cite[Theorem 7.6.2]{B13} and \cite[Theorem 14.2.1]{BSW}.
\begin{thm} \label{Thm12-3}
As $b\rightarrow 0$, we have the convergence of smooth forms on $S$,
\begin{align} \label{Equ12-4}
\chg \br{\AAAA^{\prime\prime}_Y, \omega^X, g^D, b^4 g^{TX}} \rightarrow \chg \br{\ApEzeropp, \omega^X, g^D}.	
\end{align}
In particular,
\begin{align} \label{Equ11-4-new}
  \chgBC \br{\AAAA_Y^{\prime\prime}} = \chgBC \br{\ApEpp} = \chgBC \br{Rp_* \E} \quad \mathrm{in} \; \HESBCC.	
\end{align}
\end{thm}
\begin{proof}
In \cite[Theorem 7.6.2]{B13}, \eqref{Equ12-4} is proved when $D$ is a holomorphic vector bundle on $M$.
By exactly the same method, we obtain \eqref{Equ12-4}.
Then \eqref{Equ11-4-new} is a consequence of Theorems \ref{Thm10-1-new}, \ref{Thm12-1-new} and \eqref{Equ12-4}.
The proof of our theorem is complete.
\end{proof}

\section{Exotic superconnections and Riemann-Roch-Grothendieck} \label{SectionSubmersion2}

The purpose of this section is to prove Theorem \ref{Thm9-1}, which is the final step of our main result. More precisely, we will set $\underline{\theta}_t$ as $\br{1, t^2}$ and study the limit of $\chg \br{\AAAA^{\prime\prime}_Y, \omega^X/t, g^D, \gwTX/t^3, \underline{\theta}_t}$ when $t \to 0$.

In Section \ref{SectionMain13}, we state the main theorem of this section.

In Section \ref{Localization}, we localized the problem to $\pi^{-1} X_g$.

In Section \ref{Trivialization}, we give a trivialization and replace $\X$ by $T_{\R,x}X \times \widehat{T_{\R,x}X}$.

In Section \ref{Rescalings}, we give some rescalings and Getzler rescalings which was introduced in \cite[Section 9]{B13} and \cite[Section 15]{BSW}.

In Section \ref{Section13-5-0}, we study the limit of the target operator when $t \to 0$.

In Section \ref{Section13-6-0}, we finally prove the main theorem of this section, hence complete the proof of the main theorem of this paper.
\subsection{The main result of this section} \label{SectionMain13}

In this section, we take $c=1$.
Put 
\begin{align}
\underline{\theta}_t = \br{1, t^{2}}.
\end{align}

The aim of this section is to prove
\begin{thm} \label{Thm13-5}
As $t\rightarrow 0 $, we have the convergence of smooth forms on $S$,
\begin{align}
\chg \br{\AAAA^{\prime\prime}_Y, \omega^X/t, g^D, \gwTX/t^3, \utht} \rightarrow p_* \brr{q^* \mathrm{Td}_g \br{TX, \wgTX} \chg \br{\AEzeropp, g^D}}.	
\end{align}

\end{thm}
\begin{proof}
The proof will be given in Section \ref{FinalProof}.
\end{proof}

\begin{thm} \label{Thm13-6}
The following identity holds,
\begin{align}
\chgBC \br{\AApp_Y} = p_* \brr{q^* \tdgBC \br{TX} \chgBC \br{\AEpp}}	 \; \text{in} \; \HESBCC.
\end{align}	
\end{thm}
\begin{proof}
	By Theorems \ref{Thm12-1-new} and \ref{Thm13-5}, we get our theorem.
\end{proof}

Now we can complete the proof of the Theorems \ref{MainTheorem} and  \ref{Thm9-1}.
\begin{proof}[The proof of Theorems \ref{MainTheorem} and  \ref{Thm9-1}]
	By Theorems \ref{Thm12-3} and \ref{Thm13-6}, we get Theorem \ref{Thm9-1}.
	As explained in Section \ref{SectionMainResult}, this also completes the proof of the main theorem \ref{MainTheorem}.
\end{proof}

\subsection{Localize the problem to $\pi^{-1} X_g$} \label{Localization}

Let $\AAAA_{ Y, \utht, t}$ be the superconnection $\AAAA_{Y,\theta}$ associated with $\omega^X/t$, $g^D$, $\gwTX/t^3$, $\utht$.

For $a>0$, put
\begin{align} \label{Equ10-30}
K_a s \br{x,Y} = s \br{x,aY}.	
\end{align}

Let 
\begin{align} \label{Equ13-42}
	\mathfrak{M}_{\utht,t} =t^{3N^{ \lwxbs  }/2} K_{t^{3/2}} \AAAA^{2}_{ Y, \utht, t}  K_{t^{-3/2}} t^{-3N^{\lwxbs  }/2}.
\end{align}

Let $\wnFFone_{t}$ be the analogue of $\wnFFone$ when $\omega^X$ is replaced by $\omega^X/t$.
By \eqref{Equ13-2}, \eqref{Equ13-42}, we have 
\begin{align} \label{Equ11-10-new}
\begin{aligned}
	\mathfrak{M}_{\utht,t} =  &  -\frac{1}{2} \Delta^V_{\gwTX} + \frac{t^{4}}{4}  |Y|_{g^{\widehat{T X}}}^2 |Y|_{g^{T X}}^2 + t^{3/2}\br{ \wb{w}^i \wedge i_{\overline{w}_i} - i_{\wb{w}_i} i_{w_i} } \\  
	& + t^{3/2} \wnFFone_{t,Y} + t^{3/2} i_Y C -  \nabla^V_{R^{\wTX} \wY} - \anbr{R^{\wTX} \wb{w}_i, \wb{w}^j} \wb{w}^i i_{\wb{w}_j} \\
	& + \frac{t^{5/2}}{2} |Y|^2_{g^{\widehat{TX}}} \overline{\widehat{w}}^i \overline{w}_{i*} - \frac{1}{2} |Y|^2_{g^{\widehat{TX}}} q^* \bar{\partial}^X \partial^X i \omega^X \\
	&+ t^{1/2} \left(\frac{t^{2}}{2}|Y|_{g^{\widehat{T X}}}^2-1\right) \left(\overline{\br{\gamma y}_*}-i_y q^* \partial^X i \omega^X\right) + t^{5/2} \br{y_* i_{\widehat{\widehat{y}}} +\widehat{y}_* \bar{y}_* } \\
	& +t \left(q^* \bar{\partial}^X i \omega^X\right) i_{\overline{\widehat{y}}} - t \widehat{y}_* q^* \partial^X i \omega^X +  t q^* i \omega^X\left(\nabla^{V}_{\overline{\widehat{y}}}+\overline{\widehat{w}}^i i_{\overline{\widehat{w}}_i}\right) + A^{E_0,2}. 
	\end{aligned}
	\end{align}

By (\ref{Equ13-4}), (\ref{Equ13-42}), we have 
\begin{align} \label{Equ13-8}
	\chg \br{\AAAA^{\prime\prime}_Y, \omega^X/t, g^D, \gwTX/t^3, \utht}	= \varphi \Trs \brr{g \exp \br{- \mathfrak{M}_{\utht,t}}}.
\end{align}
Let $P_{\utht,t} \br{\br{x,Y}, \br{x^\prime,Y^\prime}}$ be the smooth kernel for $\exp \br{-\mathfrak{M}_{\utht,t}}$ with respect to $\frac{d\widehat{x}^\prime dY^\prime}{\br{2\pi}^{2n}}$. So we have 
\begin{align} \label{Equ13-10}
	\chg \br{\AAAA^{\prime\prime}_Y, \omega^X/t, g^D, \gwTX/t^3, \utht}	= 
	\varphi \int_\X \Trs \brr{g P_{\utht,t} \br{g^{-1} \br{x,Y}, \br{x,Y}}} \frac{d\widehat{x} dY}{(2\pi)^{2n}}.
\end{align}

We denote $\widehat{d}(x,x^\prime)$ the Riemannian distance on $X$ with respect to $\wgTX$.
\begin{thm} \label{Thm13-1}
Given $k>0$, there exist $m\in \N$, $c^\prime >0$, $C^\prime >0$ such that for $t\in ]0,1]$,  $z = \br{x,Y}$, $z^\prime = \br{x^\prime, Y^\prime} \in \X$, we have 
\begin{align} \label{Equ13-38}
| P_{\utht,t} \br{z, z^\prime}| \leq \frac{C^\prime}{t^m} \exp \br{ - c^\prime  \br{ t^{2} |Y|^2_{\gwTX} + t^{2}|Y^\prime|^2_{\gwTX} +  t^{-1} \widehat{d}^2 \br{x, x^\prime}}}.	 
\end{align}
	
\end{thm}
\begin{proof}
Replacing $d$ by $t$ in \cite[Theorem 16.5.1]{BSW} and using the conjugation identities \eqref{Equ13-42} and \cite[(16.4.3), (16.5.2)]{BSW}, we obtain this theorem.
\end{proof}

Let $\widehat{a}_X$ be a lower bounded for the injective radius of $X$ with respect to $\wgTX$.
Let $N_{X_g/X}$ be the orthogonal bundle to $TX_g$ in $TX|_{X_g}$. We identify $X_g$ to the zero section of $N_{X_g/X}$.

Given $\eta>0$, let $\mathcal{V}_\eta$ be the $\eta$-neighborhood of $X_g$ in $N_{X_g/X}$.
Then there exists $\eta_0 \in ]0, a_X/32]$ such that if $\eta \in ]0, 8 \eta_0]$, the map $\br{x,Z} \in N_{X_g/X} \rightarrow \exp_x^X \br{Z} \in X$ is a diffeomorphism from $\mathcal{V}_\eta$ on the tubular neighbourhood $\mathcal{U}_\eta$ of $X_g$ in $X$.
In the sequel, we identify $\mathcal{V}_\eta$ and $\mathcal{U}_\eta$.
This identification is $g$-equivariant.
Let $\beta \in ]0,\eta_0]$ be small enough so that if $\widehat{d} \br{g^{-1}x, x} \leq \beta$, then $x\in \mathcal{U}_{\eta_0}$.

\begin{prop} \label{Prop13-2}
Given  $k>0$,  there are $c>0$, $C>0$ such that for $t\in ]0,1]$,
\begin{align}
\left|\int_{\pi^{-1} X \backslash \mathcal{U}_{\eta_0}} \Trs \brr{g P_{\utht,t} (g^{-1}z, z)}dz\right|	\leq C \exp \br{-c \beta^2/t}.
\end{align}

\end{prop}
\begin{proof}
By Theorem \ref{Thm13-1} and the assumption on $\beta$, for $z= \br{x, Y}\in \pi^{-1} X \backslash \mathcal{U}_{\eta_0}$,
\begin{align} \label{Equ11-8-new}
| P_{\utht,t} \br{g^{-1}z, z}|	 \leq C \exp \br{-c \br{t^{2}|Y|^2_{\gwTX} + \beta^2/t} }.
\end{align}
The proposition follows from \eqref{Equ11-8-new}.
\end{proof}

The above shows that the integral in the right-hand side of (\ref{Equ13-10}) localize near $\pi^{-1}X_g$.


\subsection{Replace $\X$ by $T_{\R,x}X \times \widehat{T_{\R,x}X}$} \label{Trivialization}
Recall that $\mathfrak{M}_{\utht,t}$ acts on $\Lambda \br{T_s^* S} \wo \Omega \br{X, D|_{X\times \brrr{s}} \wo \II}$.
Set 
\begin{align} \label{Equ10-9}
\mathbb{F} = q^* \br{\lxc \wo \lwxbs} \wo D.	
\end{align}
Then $\FF$ is a vector bundle on $M$ and 
\begin{align}
	\Omega \br{X, D|_{X\times \brrr{s}} \wo \II} = C^\infty \br{\X, \upi^* \FF|_{\X \times \brrr{s}}},
\end{align}

For $U \in T_\R X$, let $\wnF_U$ be the derivative on $\mathbb{F}$ induced by $\wnlTCX_U$, $\nabla^{\lwxbs}_U$,  and $\nabla^D_U$. Put
\begin{align} 
\wnFmo_{U} = \wnF_U +  i_{\widehat{\tau} \br{U,\cdot}}.	
\end{align}

Take $\epsilon \in ]0,\widehat{a}_X/2]$,  given $(x,s)\in X_g \times S$, using geodesic coordinates with respect to the metric $\wgTX$ centered at $x$, we can identify the open ball $B^{T_{\R, x} X} \br{0, \epsilon}$ with the corresponding open ball $B^X \br{x, \epsilon}$ in $X$.
Also along the geodesics based at $x$, we trivialize the vector bundle $\wTX$ by parallel transport with respect to $\nabla^{\wTX}$.
We identify the total space of $\X$ over $B^X \br{x,\epsilon}$ with $B^{T_{\R,x}X} \br{0,\epsilon} \times \widehat{T_{\R,x}X}$. 
Near $x$, we trivialize $\FF$ by parallel transport along the geodesics with respect to connection $\wn^{\FF,-1}$.
Also $ \Lambda \br{T^*_{\C,s} S}$ is already trivial on $X$.
Our operator $\mathfrak{M}_{\utht,t}$ acts now on sections of $\Lambda \br{T^*_{\C,s} S} \wo \FF_x$ over $B^{T_{\R,x}X} \br{0,\epsilon} \times \widehat{T_{\R,x}X}$.

As in \cite[Section 4.8]{BL08}, \cite[Section 9.1]{B13} and \cite[Section 15.3]{BSW},we can extend $\mathfrak{M}_{\utht,t}|_{B^{T_{\R,x}X} \br{0,\epsilon/2} \times \widehat{T_{\R,x}X}}$ to the full $\mathbb{H}_{x} = T_{\R,x}X\times \widehat{T_{\R,x}X}$.
In fact, since the coordinate system is not holomorphic, we cannot simply define the extended operator as being the obvious extension of $\mathfrak{M}_{\utht,t}$ associated with a metric $g^{T_{\R,x} X}$ on $T_{\R,x} X$ that equals to $\wgTX$ over $B^{T_{\R,x} X} \br{0, \epsilon/2} \times \widehat{T_{\R,x} X}$ and is constant at infinity as in \cite[Section 4.8]{BL08}.
However, because of the spinor interpretation of $\lxb$, we can extend the methods of \cite[Section 4.8]{BL08} to the present situation as in \cite[Section 9.1]{B13} and \cite[Section 15.3]{BSW}.
Ultimately, we can construct a $g$-invariant operator $\mathfrak{N}_{\utht,t,x}$ acting on $C^\infty \br{\mathbb{H}_x, \Lambda \br{T^*_{\C,s} S}  \wo \mathbb{F}_x}$ which is still hypoelliptic and which coincides with $\mathfrak{M}_{\utht,t}$ over $B^{T_{\R,x} X} \br{0, \epsilon/2} \times \widehat{T_{\R,x} X}$.
Let $Q_{\utht,t,x} \br{\br{U,Y}, \br{U^\prime, Y^\prime}}$ denote its smooth kernel on $\mathbb{H}_x$ with respect to the volume $\frac{dU^\prime dY^\prime}{\br{2\pi}^{2n}}$.

Let $dv_{X_g}$ be the volume element on $X_g$, and let $dv_{N_{X_g/X}}$ be the volume element along the fibers of $N_{X_g/X}$.
Let $k(x,y)$, $x\in X_g$, $y\in N_{X_g/X,x}$, $|y|\leq \eta_0$ be the smooth function with values in $\R_+$ such that on $\mathcal{U}_{\eta_0}$,
\begin{align}
dv_X(x,y) = k(x,y) 	dv_{N_{X_g/X}} (y) dv_{X_g}(x).
\end{align}
Note that
\begin{align}
k(x,0) = 1.	
\end{align}

\begin{prop} \label{Prop13-3}
For $k>0$,  there exist $c>0$, $C>0$, such that for $t\in ]0,1]$, $x\in X_g$,
\begin{align}
\left|\int_{\pi^{-1} \brrr{U \in N_{X_g/X,x,|U|\leq \eta_0}}} \left(\Trs \brr{g P_{\utht,t} \br{g^{-1}\br{U,Y}, \br{U,Y}}} k\br{x,U} \right. \right. \\
\left. \left. - \Trs \brr{g Q_{\utht,t,x} \br{g^{-1}\br{U,Y}, \br{U,Y}}}\right) dU dY \right|
\leq C \exp \br{-c \beta^2/t}. \nonumber
\end{align}	
\end{prop}
\begin{proof}
Proceeding as in \cite[Proposition 4.8.2]{BL08} and \cite[Theorem 6.3]{B08}, we can prove this proposition.
\end{proof}
 
\subsection{Rescalings} \label{Rescalings}

Let $w_1, \cdots, w_n$ be an orthonormal basis of $T_x X$ with respect to the metric $g^{\widehat{T_x X}}$, such that $w_1, \cdots, w_\ell$ is a basis of $T_x X_g$, and $w_{\ell+1}, \cdots , w_n$ is a basis of $N_{X_g/X,x}$.
Let  $\mathfrak{W}_x$ be another copy of $T_x X_g$.
Let $\mathfrak{w}_1, \cdots, \mathfrak{w}_\ell$ be the corresponding orthonormal basis of $\mathfrak{W}_x$.
Then $\mathfrak{w}_i$, $\overline{\mathfrak{w}}_i$, $1\leq i \leq \ell$ generate the algebra $\Lambda \br{\mathfrak{W}_{ \C, x}}$.

Set 
\begin{align}
\mathfrak{o} = \sum_{i=1}^{\ell} \br{w^i \mathfrak{w}_i + \overline{w}^i \overline{\mathfrak{w}}_i}.	
\end{align}
If $T \in \mathrm{End} \br{\Lambda \br{T^*_{\C, x} X_g}}$, let 
\begin{align}
T_t = 	\exp \br{-\mathfrak{o}/t^{3/2}} T \exp \br{\mathfrak{o}/t^{3/2}},
\end{align}
then $T_t$ is obtained from $T$ by replacing, for $1 \leq i \leq \ell$,  the operators $i_{w_i}$, $i_{\overline{w_i}}$ by $i_{w_i} + \mathfrak{w}_i/t^{3/2}$, $i_{\overline{w}_i} + \overline{\mathfrak{w}_i}/t^{3/2}$, while leaving unchanged the operators $w^i$, $\overline{w}^i$. This is an analogue of the Getzler rescaling.

Since $T_t \in \Lambda \br{\mathfrak{W}_{\C,x}} \wo \End \br{\Lambda \br{T^*_{\C,x} X_g}}$, we have $T_t \mathrm{1} \in \Lambda \br{\mathfrak{W}_{\C,x}} \wo \Lambda \br{T^*_{\C,x} X_g}$.

Note that $\beta = \prod_i^{\ell} w^i \mathfrak{w}_i$ is a canonical section of the trivial line bundle $\lambda = \Lambda^\ell \br{T^* X_g} \wo \Lambda^\ell \br{\mathfrak{W}_x}$, and that $\overline{\beta} = \prod_i^\ell \overline{w}^i \overline{\mathfrak{w}}_i$ is the corresponding conjugate section of $\overline{\lambda}$.

Let $\brr{T_t  \mathrm{1}}^{\mathrm{max}} \in \C$ be the coefficient of $\beta \overline{\beta}$ in the expansion of $T_t \mathrm{1}$. Then we have 
\begin{align}
  \Trs^{\Lambda \br{T^*_{\C,x} X_g}} \brr{T} = t^{3\ell} \brr{T_t 1}^{\mathrm{max}}.
\end{align}


If instead $F$ is a section of 
\begin{align}
\Lambda \br{T^*_{\C, s} S} \wo \Lambda \br{\mathfrak{W}_{\C,x}} \wo \End \br{\Lambda \br{T^*_{\C,x} X_g}} \wo \End \br{\Lambda \br{N^*_{X_g/X}} \wo \Lambda \br{\overline{\widehat{T^*X_x}}} \wo D_{s,x}}, 
\end{align}
we can define
\begin{align}
\Trs^{\Lambda \br{N^*_{X_g/X}} \wo \Lambda \br{\overline{\widehat{T^*X_x}}} \wo D_{s,x}} \brr{\brr{F \mathrm{1}}^{\mathrm{max}} }
\end{align}
in the obvious way, which takes value in $\Lambda \br{T^*_{\C, s} S}$.

For $a>0$, let $I_a$ be the map action on smooth functions on $\mathbb{H}_x = T_{\R,x} X \times \widehat{T_{\R,x}X}$ with values in $ \Lambda \br{T^*_{\C,s} S} \wo \mathbb{F}_{s,x}$ that is given by
\begin{align}
I_a s \br{U,Y} = s \br{aU,Y}.	
\end{align}

Let $N^{\Lambda \br{N^*_{X_g/X,\C}}}$ be the number operator of $\Lambda \br{N^*_{X_g/X,\C}}$.
Put
\begin{align} \label{Equ13-13}
\mathfrak{P}_{\utht,t,x} = t^{N^{\Lambda \br{N^*_{X_g/X,\C}}}}	\exp \br{-\mathfrak{o}/t^{3/2}} I_{t^{3/2}} \mathfrak{N}_{\utht,t,x}	I_{t^{-3/2}} \exp \br{\mathfrak{o}/t^{3/2}}  t^{-N^{\Lambda \br{N^*_{X_g/X,\C}}}}	.
\end{align}

Let $S_{\utht, t, x} \br{\br{U,Y} , \br{U^\prime,Y^\prime}}$ be the smooth kernel of $\exp \br{- \mathfrak{P}_{\utht,t,x}}$.
The following result is an analogue of \cite[Proposition 6.5]{B08} and \cite[Proposition 4.8.5]{B08b}.
\begin{prop}
The following identity holds,
\begin{align} \label{Equ13-14}
t^{3(n-\ell)} &\Trs \brr{g Q_{\utht,t,x} \br{g^{-1} \br{t^{3/2} U, Y} , \br{t^{3/2} U, Y}}}	\\
&= \Trs^{\Lambda \br{N^*_{X_g/X}} \wo \Lambda \br{\overline{\widehat{T^*X_x}}} \wo D_{s,x}} \brr{\brr{g S_{\utht,t,x} \br{g^{-1} \br{ U, Y} , \br{ U, Y}} \mathrm{1}}^{\mathrm{max}}}.\nonumber
\end{align}

\end{prop}

\subsection{The limit as $t\rightarrow 0 $ of $\mathfrak{P}_{\utht,t,x}$} \label{Section13-5-0}

\begin{prop}
There exist $c>0$, $C>0$ such that for $t\in ]0,1]$,  $x\in X_g$, $U\in N_{X_g/X,x}$, $|U|\leq \eta_0/{t^{3/2}}$, $Y \in \widehat{T_{\R,x} X}$,	then
\begin{align} \label{Equ13-18}
|S_{\utht, t, x} \br{g^{-1}\br{U,Y} , \br{U,Y}}|	 \leq C \exp \br{-c \br{|U|_{\gwTX}^{2/3} +  |Y|_{\gwTX}^{4/3}}}.
\end{align}

\end{prop}
\begin{proof}
It has been proved in 	\cite[(11.6.9)]{B13} when $D$ is a holomorphic vector bundle. By the same arguments, we can prove this proposition.
\end{proof}

We denote $j: X_g \rightarrow X$ the obvious embedding.
Consider the following operator on $TX|_{X_g} \times \widehat{TX}|_{X_g}$,
\begin{align} \label{Equ13-11}
\mathfrak{P}_{ 0 ,s} = -\frac{1}{2} \Delta^V_{\gwTX} + \sum_i^{\ell} \br{\wb{w}^i \overline{\mathfrak{w}}_i - i_{\wb{w}_i} \mathfrak{m}_i} - \nabla^V_{j^*R^{\wTX} Y} \\
- \anbr{j^* R^{\wTX} \wb{w}_i, \wb{w}^j} \wb{w}^i i_{\wb{w}_j} + \nabla^{T_{\R}X}_Y + j^* A^{E_0,2}. \nonumber
\end{align}
In (\ref{Equ13-11}), the tensors are evaluated at $(x,s)$.

Let $S_{0, s} \br{\br{U,Y} , \br{U^\prime,Y^\prime}}$ be the smooth kernel of $\exp \br{-\mathfrak{P}_{0,s}}$.
Proceeding as in \cite[Theorem 11.6.1]{B13} and \cite[Theorem 16.7.1]{BSW}, we have 
\begin{prop}
As $t\rightarrow 0$,
\begin{align} \label{Equ13-41}
\mathfrak{P}_{\utht, t ,x} \rightarrow \mathfrak{P}_{0,s}	,
\end{align}
in the sense that the coefficients of the operator $\mathfrak{P}_{\utht, t ,x}$ converge uniformly together with all their derivatives uniformly over compact sets towards the corresponding coefficients of $\mathfrak{P}_{0,s}	$.
Also,
\begin{align} \label{Equ13-19}
	S_{\utht, t, x} \br{\br{U,Y} , \br{U^\prime,Y^\prime}} \rightarrow S_{0, s} \br{\br{U,Y} , \br{U^\prime,Y^\prime}}.
\end{align}
\end{prop}
\begin{proof}
By \eqref{Equ13-2}, \eqref{Equ13-42},  \eqref{Equ13-13} and \eqref{Equ13-11}, the proof has already been given in \cite[Theorem 11.6.1]{B13} and \cite[Theorem 16.7.1]{BSW}.
\end{proof}

Recall that $\varphi: \Lambda \br{ T^*_\C S }\rightarrow \Lambda \br{T^*_\C S}$ was defined in Section \ref{EquiCherCharForm}.
\begin{thm} \label{Thm13-4}
The following identity holds,
\begin{align}
\varphi \int_{N_{X_g/X} \times \wTRX} 	\Trs^{\Lambda \br{N^*_{X_g/X}} \wo \Lambda \br{\overline{\widehat{T^*X_x}}} \wo D_{s,x}} \brr{\brr{S_{0,s} \br{g^{-1} \br{ U, Y} , \br{ U, Y}} \mathrm{1}}^{\mathrm{max}}} \frac{dUdY}{\br{2\pi}^{2n}} \nonumber\\
= \varphi \frac{1}{(2i\pi)^\ell} \brr{\mathrm{Td}_g \br{-R^{\wTX}|_{X_g}} \Trs \brr{g \exp \br{-A^{E_0,2}|_{X_g}}}}^{\mathrm{max}}_{s,x}.
\end{align}
\end{thm}
\begin{proof}
The only difference of 	$\mathfrak{P}_{0,s,x}$ and the operator defined in \cite[(9.1.9)]{B13} is that $D$ is a $G$-equivariant holomorphic vector bundle in \cite[(9.1.9)]{B13} and $\AEzerosq$ naturally becomes $R^D$ in that context.
However, since $j^* \AEzerosq$ is of even degree and commutes with other terms in $\mathfrak{P}_{0,s}$, the proof our theorem is the same as \cite[Theorem 10.2.2]{B13}.
\end{proof}

\subsection{The proof of Theorem \ref{Thm13-5}} \label{FinalProof} \label{Section13-6-0}
By (\ref{Equ13-8}), (\ref{Equ13-10}), 	
\begin{align} \label{Equ13-15}
\chg \br{\AAAA^{\prime\prime}_Y, \omega^X/t, g^D, \gwTX/t^3, \utht} = 	\varphi \int_\X \Trs \brr{g P_{\utht,t} \br{g^{-1} \br{x,Y}, \br{x,Y}}} \frac{d\widehat{x} dY}{(2\pi)^{2n}}.
\end{align}
By Propositions \ref{Prop13-2} and \ref{Prop13-3}, we get 
\begin{multline} \label{Equ13-16}
	\left|\varphi \int_\X \Trs \brr{g P_{\utht,t} \br{g^{-1} \br{x,Y}, \br{x,Y}}} \frac{d\widehat{x} dY}{(2\pi)^{2n}} \right.\\
	\left. -\varphi \int_{X_g} \br{\int_{\brrr{U \in N_{X_g/X}, |U|\leq \eta_0} \times \wTRX }  \Trs \brr{g Q_{\utht,t,x} \br{g^{-1}\br{U,Y}, \br{U,Y}}} \frac{dU dY}{(2\pi)^{2n}}} dv_{X_g}\right| \\
	\leq C \exp \br{-c \beta^2/t}. 
\end{multline}

By (\ref{Equ13-12}), (\ref{Equ13-13}) and (\ref{Equ13-14}), we have 
\begin{multline} \label{Equ13-17}
\int_{\brrr{U \in N_{X_g/X}, |U|\leq \eta_0} \times \wTRX }  \Trs \brr{g Q_{\utht,t,x} \br{g^{-1}\br{U,Y}, \br{U,Y}}} \frac{dU dY}{(2\pi)^{2n}}\\
= \int_{\brrr{U \in N_{X_g/X}, |U|\leq \eta_0/t^{3/2}} \times \wTRX}	 \Trs \brr{\brr{S_{\utht,t,x} \br{g^{-1} \br{ U, Y} , \br{ U, Y}} \mathrm{1}}^{\mathrm{max}}} \frac{dU dY}{(2\pi)^{2n}} .
\end{multline}
By (\ref{Equ13-18}), (\ref{Equ13-19}),  and the dominated convergence theorem, we find that,
\begin{align} \label{Equ13-39}
  \int_{\brrr{U \in N_{X_g/X}, |U|\leq \eta_0/t^{3/2}} \times \wTRX}	 \Trs \brr{\brr{S_{\utht,t,x} \br{g^{-1} \br{ U, Y} , \br{ U, Y}} \mathrm{1}}^{\mathrm{max}}} \frac{dU dY}{(2\pi)^{2n}} \\
  \rightarrow \int_{N_{X_g/X} \times \wTRX} 	\Trs \brr{\brr{S_{0,s} \br{g^{-1} \br{ U, Y} , \br{ U, Y}} \mathrm{1}}^{\mathrm{max}}} \frac{dUdY}{\br{2\pi}^{2n}} dv_{X_g}. \nonumber
\end{align}
By (\ref{Equ13-15}), (\ref{Equ13-16}), (\ref{Equ13-17}), and \eqref{Equ13-39}
we obtain,
\begin{align} \label{Equ13-40}
	&\chg \br{\AAAA^{\prime\prime}_Y, \omega^X/t, g^D, \gwTX/t^3, \utht} \\\rightarrow & \varphi  \int_{X_g}\br{\int_{N_{X_g/X} \times \wTRX} 	\Trs \brr{\brr{S_{0,s,x} \br{g^{-1} \br{ U, Y} , \br{ U, Y}} \mathrm{1}}^{\mathrm{max}}} \frac{dUdY}{\br{2\pi}^{2n}}} dv_{X_g}. \nonumber
\end{align}

Finally, by Theorem \ref{Thm13-4} and \eqref{Equ13-40}, the proof of our theorem is complete. \qed

\bibliographystyle{alpha}
\bibliography{Equivariant_Riemann-Roch-Grothendieck_theorem_paper.bib}

\end{document}